\documentclass[a4paper]{amsart}

\usepackage{amsmath,amssymb,amsfonts,amsthm}

\usepackage[dvipsnames, table]{xcolor}
\usepackage[colorlinks=true,citecolor=Aquamarine, linkcolor=DarkOrchid]{hyperref}
\usepackage[a4paper, 
    total={160mm,247mm},
    left=25mm,
    top=25mm,]{geometry}
\usepackage{mathtools}

\usepackage{enumerate}
\usepackage{xargs}
\usepackage{stmaryrd}

\usepackage{adjustbox}
\usepackage[all]{xy}

\usepackage{tikz-cd}
\usetikzlibrary{positioning, fit, shapes, arrows, trees}
\pgfdeclarelayer{bg}  
\pgfsetlayers{bg,main} 

\theoremstyle{plain}
\newtheorem{thm}{Theorem}[section]
\newtheorem*{thm*}{Theorem}
\newtheorem{prop}[thm]{Proposition}

\theoremstyle{definition}
\newtheorem{defi}[thm]{Definition}
\newtheorem*{defi*}{Definition}

\theoremstyle{remark}
\newtheorem{rem}[thm]{Remark}

\newtheorem{example}[thm]{Example}
\newtheorem{exple}[thm]{Example}
\AtBeginEnvironment{example}{
  \pushQED{\qed}
}
\AtEndEnvironment{example}{\popQED\endexample}
\AtBeginEnvironment{exple}{
  \pushQED{\qed}
}
\AtEndEnvironment{exple}{\popQED\endexample}

\newcommand{\QQ}{\mathbb{Q}}

\newcommand{\ZZ}{\mathbb{Z}}
\newcommand{\KK}{\mathbb{K}}

\newcommand{\Lc}{\mathcal{L}}

\newcommand{\To}{\longrightarrow}
\renewcommand{\leq}{\leqslant}
\renewcommand{\geq}{\geqslant}
\renewcommand{\emptyset}{\varnothing}
\newcommand{\defas}{\coloneqq}

\newcommandx{\levtrees}[1][1=n]{\operatorname{MT}^{lev}(#1)}
\newcommandx{\mergetrees}[1][1=n]{\operatorname{MT}(#1)}
\newcommandx{\levtreesT}[1][1=n]{\operatorname{MT}_T^{lev}(#1)}
\newcommandx{\mergetreesT}[1][1=n]{\operatorname{MT}_T(#1)}
\newcommandx{\hyptrees}[1][1=n]{\operatorname{HT}(#1)}
\newcommandx{\Leaves}[1][1=n]{\operatorname{Leaves}(#1)}

\newcommandx{\Gent}[2][1=1, 2=2]{\widetilde{\begin{tikzpicture}[grow=up, frontier/.style={distance from root=10pt}, level distance=0.8cm, inner sep=1pt, scale=0.7 ]
\coordinate (bot)
   child[sibling distance=1.2cm]{node(2){#2}
        }
    child[sibling distance=1.2cm]{node(1){#1}
        }
   ;
   \node[above=1cm of bot] (0) {0};
   \draw (2) edge[bend right, ->] (1) ;
 \draw (1.north) edge[bend left, ->] (0) ;
\end{tikzpicture}}}

\newcommandx{\Gentsst}[2][1=1, 2=2]{{\begin{tikzpicture}[grow=up, frontier/.style={distance from root=10pt}, level distance=0.8cm, inner sep=1pt, scale=0.7 ]
\coordinate (bot)
   child[sibling distance=1.2cm]{node(2){#2}
        }
    child[sibling distance=1.2cm]{node(1){#1}
        }
   ;
   \node[above=1cm of bot] (0) {0};
   \draw (2) edge[bend right, ->] (1) ;
 \draw (1.north) edge[bend left, ->] (0) ;
\end{tikzpicture}}}

\newcommandx{\Genb}[2][1=1, 2=2]{\widetilde{\begin{tikzpicture}[grow=up, frontier/.style={distance from root=10pt}, level distance=0.8cm, inner sep=1pt, scale=0.7 ]
\coordinate (bot)
   child[sibling distance=1.2cm]{node(2){#2}
        }
    child[sibling distance=1.2cm]{node(1){#1}
        }
   ;
   \node[above=0.7cm of bot] (0) {0};
 \draw (1.north) edge[bend left, ->] (0) ;
 \draw (2.north) edge[bend right, ->] (0) ;
\end{tikzpicture}}}

\newcommandx{\Ns}[1][1=\Lc]{\mathcal{N}(#1)}

\newcommand{\diagram}[1]{\SelectTips{cm}{10}\xymatrix{#1}}

\DeclareFontFamily{U}{mathx}{\hyphenchar\font45}
\DeclareFontShape{U}{mathx}{m}{n}{
      <5> <6> <7> <8> <9> <10>
      <10.95> <12> <14.4> <17.28> <20.74> <24.88>
      mathx10
      }{}
\DeclareSymbolFont{mathx}{U}{mathx}{m}{n}
\DeclareFontSubstitution{U}{mathx}{m}{n}
\DeclareMathAccent{\widecheck}{0}{mathx}{"71}

\newcommand{\chains}[1]{c^{#1}}
\newcommand{\chainstop}[1]{\widehat{c}^{\,#1}}
\newcommand{\chainsbottom}[1]{\widecheck{c}^{\,#1}}

\newcommand{\h}[1]{h^{#1}}
\newcommand{\htop}[1]{\widehat{h}^{#1}}
\newcommand{\hbottom}[1]{\widecheck{h}^{#1}}

     \newcommand{\udt}{\begin{tikzpicture}[scale=0.6, anchor=base, baseline]
   \tikzstyle{ver} = [circle, draw, fill, inner sep=0.5mm]
   \tikzstyle{edg} = [line width=0.6mm]
   \node[ver] at (1,0) {};
   \node[ver] at (2,0) {};
   \node[ver] at (3,0) {};
   \node      at (1,-0.6) {1};
   \node      at (2,-0.6) {2};
   \node      at (3,-0.6) {3};
 \end{tikzpicture}}
   \newcommand{\utd}{   
     \begin{tikzpicture}[scale=0.6, anchor=base, baseline]
   \tikzstyle{ver} = [circle, draw, fill, inner sep=0.5mm]
   \tikzstyle{edg} = [line width=0.6mm]
   \node[ver] at (1,0) {};
   \node[ver] at (2,0) {};
   \node[ver] at (3,0) {};
   \node      at (1,-0.6) {1};
   \node      at (2,-0.6) {3};
   \node      at (3,-0.6) {2};
 \end{tikzpicture}}
      \newcommand{\dut}{
     \begin{tikzpicture}[scale=0.6, anchor=base, baseline]
   \tikzstyle{ver} = [circle, draw, fill, inner sep=0.5mm]
   \tikzstyle{edg} = [line width=0.6mm]
   \node[ver] at (1,0) {};
   \node[ver] at (2,0) {};
   \node[ver] at (3,0) {};
   \node      at (1,-0.6) {2};
   \node      at (2,-0.6) {1};
   \node      at (3,-0.6) {3};
 \end{tikzpicture}}

 \newcommand{\utu}{  \begin{tikzpicture}[scale=0.6, anchor=base, baseline]
   \tikzstyle{ver} = [circle, draw, fill, inner sep=0.5mm]
   \tikzstyle{edg} = [line width=0.6mm]
   \node[ver] at (1,0) {};
   \node[ver] at (2,0) {};
   \node[ver] at (3,0) {};
   \node      at (1,-0.6) {1};
   \node      at (2,-0.6) {3};
   \node      at (3,-0.6) {2};
      \draw[edg] (1,0) to[bend left=90, looseness=2] (2,0);
 \end{tikzpicture}}
      
      \newcommand{\udd}{   \begin{tikzpicture}[scale=0.6, anchor=base, baseline]
   \tikzstyle{ver} = [circle, draw, fill, inner sep=0.5mm]
   \tikzstyle{edg} = [line width=0.6mm]
   \node[ver] at (1,0) {};
   \node[ver] at (2,0) {};
   \node[ver] at (3,0) {};
   \node      at (1,-0.6) {1};
   \node      at (2,-0.6) {2};
   \node      at (3,-0.6) {3};
      \draw[edg] (2,0) to[bend left=90, looseness=2] (3,0);
 \end{tikzpicture}}
 \newcommand{\dud}{\begin{tikzpicture}[scale=0.6, anchor=base, baseline]
   \tikzstyle{ver} = [circle, draw, fill, inner sep=0.5mm]
   \tikzstyle{edg} = [line width=0.6mm]
   \node[ver] at (1,0) {};
   \node[ver] at (2,0) {};
   \node[ver] at (3,0) {};
   \node      at (1,-0.6) {2};
   \node      at (2,-0.6) {1};
   \node      at (3,-0.6) {3};
      \draw[edg] (2,0) to[bend left=90, looseness=2] (3,0);
 \end{tikzpicture}}

 \newcommand{\uud}{  \begin{tikzpicture}[scale=0.6, anchor=base, baseline]
   \tikzstyle{ver} = [circle, draw, fill, inner sep=0.5mm]
   \tikzstyle{edg} = [line width=0.6mm]
   \node[ver] at (1,0) {};
   \node[ver] at (2,0) {};
   \node[ver] at (3,0) {};
   \node      at (1,-0.6) {1};
   \node      at (2,-0.6) {3};
   \node      at (3,-0.6) {2};
      \draw[edg] (1,0) to[bend left=60] (3,0);
 \end{tikzpicture}}
      \newcommand{\udu}{
       \begin{tikzpicture}[scale=0.6, anchor=base, baseline]
   \tikzstyle{ver} = [circle, draw, fill, inner sep=0.5mm]
   \tikzstyle{edg} = [line width=0.6mm]
   \node[ver] at (1,0) {};
   \node[ver] at (2,0) {};
   \node[ver] at (3,0) {};
   \node      at (1,-0.6) {1};
   \node      at (2,-0.6) {2};
   \node      at (3,-0.6) {3};
      \draw[edg] (1,0) to[bend left=60] (3,0);
 \end{tikzpicture}}

      \newcommand{\uuu}{
    \begin{tikzpicture}[scale=0.6, anchor=base, baseline]
   \tikzstyle{ver} = [circle, draw, fill, inner sep=0.5mm]
   \tikzstyle{edg} = [line width=0.6mm]
   \node[ver] at (1,0) {};
   \node[ver] at (2,0) {};
   \node[ver] at (3,0) {};
   \node      at (1,-0.6) {1};
   \node      at (2,-0.6) {2};
   \node      at (3,-0.6) {3};
      \draw[edg] (1,0) to[bend left=90, looseness=2] (2,0);
      \draw[edg] (2,0) to[bend left=90, looseness=2] (3,0);
 \end{tikzpicture}}
     
\usepackage{array}
\newcolumntype{P}[1]{>{\centering\arraybackslash}p{#1}}

\usepackage[most]{tcolorbox}

\author{}
 \author{B\'{e}r\'{e}nice Delcroix--Oger}
 \author{Cl\'{e}ment Dupont}

 \address{Institut Montpelli\'erain Alexander Grothendieck, Universit\'{e} de Montpellier, CNRS, Montpellier, France}
 \email{berenice.delcroix-oger@umontpellier.fr}
 \email{clement.dupont@umontpellier.fr}


\title{Lie-operads and operadic modules from poset cohomology}
\date{}

\begin{document}

\begin{abstract}
As observed by Joyal, the cohomology groups of the partition posets are naturally identified with the components of the operad encoding Lie algebras. This connection was explained in terms of operadic Koszul duality by Fresse, and later generalized by Vallette to the setting of decorated partitions. In this article, we set up and study a general formalism which produces \emph{a priori} operadic structures (operads and operadic modules) on the cohomology of families of posets equipped with some natural recursive structure, that we call ``operadic poset species''. This framework goes beyond decorated partitions and operadic Koszul duality, and contains the metabelian Lie operad and Kontsevich's operad of trees as two simple instances. In forthcoming work, we will apply our results to the hypertree posets and their connections to post-Lie and pre-Lie algebras.
\end{abstract}

\maketitle

\setcounter{tocdepth}{1}
\tableofcontents

\section*{Introduction}

\subsection*{Poset cohomology and operads}

This article finds its roots in a surprising connection, discovered by Joyal \cite{joyal}, between the notion of a set partition and that of a Lie algebra. Recall that a \emph{partition} of a finite set $S$ is a set of non-empty disjoint subsets of $S$, called \emph{blocks}, whose union equals $S$. They form a poset (partially ordered set) denoted by $\Pi(S)$, where by definition $\pi\leq \pi'$ if a partition $\pi$ can be obtained from another partition $\pi'$ by merging blocks. Joyal's discovery, which builds on earlier work of Hanlon and Stanley in combinatorics \cite{Han81, Stan82} and Brandt in Lie theory \cite{brandt}, is the following isomorphism:
\begin{equation}\label{eq: intro partitions and lie}
h^{|S|-1}(\Pi(S)) \simeq \operatorname{Lie}(S)\otimes\mathrm{sgn}_S.
\end{equation}
The left-hand side is the only non-zero cohomology group of the poset $\Pi(S)$ (see \S\ref{Poset cohomology} for our conventions on poset cohomology). The right-hand side is the subspace of the free Lie algebra with generators indexed by $S$ spanned by iterated Lie brackets involving each generator exactly once---twisted by the sign representation of the symmetric group $\mathfrak{S}_S$ in order to make \eqref{eq: intro partitions and lie} an $\mathfrak{S}_S$-equivariant isomorphism.

In the modern language of algebraic operads (see, e.g., \cite{lodayvallette}), one can recast \eqref{eq: intro partitions and lie} as an isomorphism 
\begin{equation}\label{eq: intro partitions and lie second}
h^\bullet(\Pi) \simeq \Lambda^{-1}\operatorname{Lie},
\end{equation}
between the graded linear species (functor from the category of finite sets and bijections to that of graded vector spaces) $S\mapsto h^\bullet(\Pi(S))$ and the desuspension of the operad encoding Lie algebras. This operadic point of view was studied by Fresse \cite{fressepartitionposets} who explained that \eqref{eq: intro partitions and lie second} is a consequence of Koszul duality between the operad $\mathrm{Lie}$ and the operad $\mathrm{Com}$ encoding commutative algebras. Later, Vallette \cite{vallettepartitionposets} generalized that approach to the poset $\Pi^{\mathcal{Q}}$ of partitions decorated by a set operad $\mathcal{Q}$ satisfying some assumptions, leading to an isomorphism
\begin{equation}\label{eq: intro decorated partitions koszul}
h^\bullet(\Pi^{\mathcal{Q}}) \simeq \Lambda^{-1} \mathcal{Q}^!,
\end{equation}
where $\mathcal{Q}^!$ is the Koszul dual operad. (The case of the operad $\mathcal{Q}=\mathrm{Com}$ recovers \eqref{eq: intro partitions and lie second}.) Finally, Oger \cite{ogerhomologyhypertree} proved an intriguing identification, conjectured by Chapoton \cite{chapotonhyperarbres}, between the operad encoding pre-Lie algebras and the cohomology of the posets of hypertrees $\mathrm{HT}(S)$ \cite{mccammondmeier, jensenMcCammondMeier} (also known as ``Whitehead posets'' \cite{McCulloughMiller, BradyMccammondMeierMiller}). In the notation of the present article, it reads
\begin{equation}\label{eq: intro HT prelie}
\widecheck{h}^\bullet(\mathrm{HT})\simeq \Lambda^{-1}\mathrm{PreLie},
\end{equation}
where $\hbottom{\bullet}$ is a variant of poset cohomology (see \S\ref{Poset cohomology}). Our project started when we tried to understand the structure of the isomorphism \eqref{eq: intro HT prelie}, and led to more surprise: the linear species $\mathrm{PreLie}$ does not naturally appear in \eqref{eq: intro HT prelie} as an \emph{operad}, but as a left operadic module over the operad $\mathrm{PostLie}$ encoding post-Lie algebras \cite{vallettepartitionposets}, which itself appears in an isomorphism
\begin{equation}\label{eq: intro HT postlie}
h^\bullet(\mathrm{HT})\simeq \Lambda^{-1}\mathrm{PostLie}.
\end{equation}
We ended up discovering more instances of this phenomenon, some of which appear in the present article (for \eqref{eq: intro HT prelie} and \eqref{eq: intro HT postlie}, we refer the reader to our forthcoming work \cite{DOD2}). For instance, in \S\ref{sec:metablie} we study the family of ``non-singleton boolean posets'' $\mathrm{NS}(S)$ and exhibit an isomorphism
\begin{equation}\label{eq: intro NS metablie}
h^\bullet(\mathrm{NS}) \simeq \Lambda^{-1}\mathrm{MetabLie}
\end{equation}
involving the operad encoding metabelian Lie algebras. Note that this cannot possibly be a disguised instance of the story of decorated partitions and Koszul duality \eqref{eq: intro decorated partitions koszul} because $\mathrm{MetabLie}$ is not a quadratic operad.

\subsection*{Operadic poset species, their cohomology operads, and operadic modules}

All these examples lead to the question:
\vspace{.1cm}
\begin{center}
Why/How do operadic structures arise in poset cohomology?
\end{center}
\vspace{.1cm}
We answer it by identifying a \emph{structure} on families of posets such as $\Pi$, $\Pi^{\mathcal{Q}}$, $\mathrm{HT}$, and $\mathrm{NS}$ which gives an \emph{a priori} operadic structure on their cohomology, and such that \eqref{eq: intro partitions and lie second}, \eqref{eq: intro decorated partitions koszul}, \eqref{eq: intro HT postlie}, \eqref{eq: intro NS metablie} are isomorphisms of graded operads, and \eqref{eq: intro HT prelie} is an isomorphism of left operadic modules. We give the name \emph{operadic poset species} to this abstract structure. In order to state its definition, let us emphasize the recursive structure of the family of partition posets: lower intervals in $\Pi(S)$ are isomorphic to partition posets, and upper intervals are isomorphic to products of partition posets. Concretely, for a partition $\pi$ (viewed as a set of blocks) of a finite set $S$, we have isomorphisms
\begin{equation}\label{eq: intro recursive structure partition posets}
\Pi_{\leq\pi}(S) \stackrel{\sim}{\To} \Pi(\pi) \quad\quad\quad \mbox{ and } \quad\quad\quad \Pi_{\geq \pi}(S) \stackrel{\sim}{\To} \prod_{T\in \pi}\Pi(T).
\end{equation}
In plain English: the partitions below $\pi$ are obtained by merging the blocks of $\pi$, which amounts to partitioning the set of blocks itself; and the partitions above $\pi$ are obtained by partitioning each block separately. Roughly speaking, an \emph{operadic poset species} is a family of posets $P(S)$, for all finite sets $S$, which lives over the family of partition posets (an object $x\in P(S)$ has an ``underlying partition'' $\pi$), and is equipped with morphisms of posets (not necessarily isomorphisms!) which lift \eqref{eq: intro recursive structure partition posets} and satisfy certain compatibilities. The precise definition roughly goes as follows.

\begin{defi*}[Definition \ref{def: operadic poset species}]
An \emph{operadic poset species} is a poset species $P: S\mapsto P(S)$ equipped with a morphism of poset species $$\diagram{P \ar[d]_-{a} \\ \Pi}$$ and with morphisms of posets
$$\varphi_x : P_{\leq x}(S)\To P(\pi) \quad\quad\quad \mbox{ and } \quad\quad\quad \psi_x:P_{\geq x}(S)\To \prod_{T\in\pi}P(T),$$
for all finite sets $S$ and all $x\in P(S)$, where $\pi$ denotes the underlying partition $a(x)$, which are required to lift \eqref{eq: intro recursive structure partition posets} and satisfy equivariance, unitality, and associativity axioms (see \S\ref{axiom: equivariance}, \ref{axiom: unitality}, \ref{axiom: associativity}).
\end{defi*}

Our main result states that an operadic poset species naturally gives rise to a triple made of a graded operad and graded left and right operadic modules. These three objects correspond to three variants of poset cohomology, denoted by $h^\bullet$ (resp. $\hbottom{\bullet}$, resp. $\htop{\bullet}$), which are computed by the complex of (co)chains $x_0<x_1<\cdots <x_{n-1}<x_n$ in a poset such that $x_0$ is a minimal element and $x_n$ is a maximal element (resp. such that $x_0$ is a minimal element, resp. such that $x_n$ is a maximal element).

\pagebreak

\begin{thm*}[Theorems \ref{thm: main}, \ref{thm: main left}, \ref{thm: main right}]
Let $P$ be an operadic poset species. Then 
\begin{enumerate}[$\triangleright$]
\item the graded linear species $h^\bullet(P) : S\mapsto h^\bullet(P(S))$ is endowed with the structure of a graded operad;
\item the graded linear species $\hbottom{\bullet}(P):S\mapsto \hbottom{\bullet}(P(S))$ is endowed with the structure of a left operadic module over $h^\bullet(P)$;
\item the graded linear species $\htop{\bullet}(P):S\mapsto\htop{\bullet}(P(S))$ is endowed with the structure of a right operadic module over $h^\bullet(P)$.
\end{enumerate}
These structures are functorial in $P$.
\end{thm*}

Since the species of partition posets $\Pi$, with its structure \eqref{eq: intro recursive structure partition posets}, is the terminal object of the category of operadic poset species, we see thanks to \eqref{eq: intro partitions and lie second} that every graded operad $h^\bullet(P)$ appearing in the above theorem comes equipped with a morphism of graded operads
$$\Lambda^{-1}\mathrm{Lie}\stackrel{a^*}{\To} h^\bullet(P).$$
In other words, the operadic suspension $\Lambda h^\bullet(P)$ has the structure of a \emph{Lie-operad}. We note that in many examples (including all the examples described so far in this introduction), the only non-zero cohomology group $h^\bullet(P(S))$ is in degree equal to $|S|-1$. (If $P(S)$ is a graded poset of rank $|S|-1$, this is equivalent to it being Cohen--Macaulay). This property translates into the fact that the operadic suspension $\Lambda h^\bullet(P)$ is concentrated in degree zero. 

The appearance of operadic modules in poset cohomology is very natural in our formalism and, to the best of our knowledge, is a new phenomenon. We note that these modules are sometimes trivial: this is the case for $\hbottom{\bullet}(P)$ if every $P(S)$ has a greatest element, and for $\htop{\bullet}(P)$ if every $P(S)$ has a least element. For instance, both operadic modules are trivial in the classical case $P=\Pi$. By contrast, they are rarely trivial in the case of decorated partitions, as we explain below.

\subsection*{Concatenation of (co)chains}

The constructions behind our main theorem use a non-standard piece of structure in poset cohomology: the possibility to concatenate (co)chains. In the simplest instance, they read
\begin{equation}\label{eq: intro concatenation poset cohomology} h^\bullet(P_{\leq x})\otimes h^\bullet(P_{\geq x}) \To h^\bullet(P),
\end{equation}
for any poset $P$ and any element $x\in P$. Using this, we can explain how the operadic structure arises on the cohomology of an operadic poset species $P$: the operadic composition map
\begin{equation}
    h^\bullet(P(\pi)) \otimes \bigotimes_{T\in \pi} h^\bullet(P(T)) \To h^\bullet(P(S))
\end{equation}
indexed by a partition $\pi\in \Pi(S)$ is defined as the sum over all the elements $x\in P(S)$ with underlying partition $a(x)=\pi$ of the composition of the pullback maps
$$h^\bullet(P(\pi)) \stackrel{\varphi_x^*}{\To} h^\bullet(P_{\leq x}(S)) \qquad \mbox{ and } \qquad \bigotimes_{T\in \pi}h^\bullet(P(T))\to  h^\bullet\left(\prod_{T\in \pi}P(T)\right)\stackrel{\psi_x^*}{\To} h^\bullet(P_{\geq x}(S))$$
with the concatenation map \eqref{eq: intro concatenation poset cohomology}. The operadic module structures are obtained through natural variants of this construction. The proof of our main theorem follows from a careful analysis of the compatibilities that the concatenation maps \eqref{eq: intro concatenation poset cohomology} satisfy.

\subsection*{Examples}
We give many examples of our formalism, summed up in Table \ref{tableRes}. There are certainly many more natural examples, and we hope that the reader will enjoy adding their favorite combinatorial objects to our list. We have collected open questions that arose in the study of examples in \S\ref{sec: open questions}.\medskip

Regarding decorated partitions, we recast Vallette's results \cite{vallettepartitionposets} in our formalism in \S\ref{sec: decorated partitions operads}, calling ``right-decorated partitions'' the objects at play in \emph{op.~cit.}, and develop a dual theory of ``left-decorated partitions''. The operadic modules that arise in those contexts seem to be new, and we study them in \S\ref{sec: decorated partitions modules} (see also \S\ref{sec: open questions}).

 \subsubsection*{Left-decorated partitions}

For a set operad $\mathcal{P}$ and a finite set $S$, a \emph{left-$\mathcal{P}$-decorated partition} of $S$ is  a pair $(\pi, \xi)$ where $\pi$ is a partition of $S$ and $\xi$ is an element in $\mathcal{P}(\pi)$. The posets ${}^{\mathcal{P}}\Pi(S)$ of left-$\mathcal{P}$-decorated partitions form an operadic poset species under the assumption that $\mathcal{P}$ is ``left-basic''. If $\mathcal{P}$ is Koszul, then $\Lambda h^\bullet({}^{\mathcal{P}}\Pi)$ is the Koszul dual operad $\mathcal{P}^!$. The left operadic module $\Lambda\hbottom{\bullet}({}^{\mathcal{P}}\Pi)$ is rarely trivial because ${}^{\mathcal{P}}\Pi(S)$ has many maximal elements, indexed by $\mathcal{P}(S)$. As an example, we compute this module in the case of the operad $\mathcal{P}=\mathrm{As}$ encoding associative algebras.

It can be the case that $\Lambda \hbottom{\bullet}({}^{\mathcal{P}}\Pi)$ is not concentrated in degree zero even when $\mathcal{P}$ is Koszul. In this case, the posets ${}^{\mathcal{P}}\Pi(S)$ are Cohen--Macaulay but the augmented posets ${}^{\mathcal{P}}\Pi^+(S)$ obtained by adding a greatest element are not. This phenomenon arises already for the operad $\mathcal{P}=\mathrm{NAP}$ encoding non-associative permutative algebras, as can readily be seen from an Euler characteristic computation (see \S\ref{subsec: Ldp mod}, and \S\ref{sec: open questions} for more detailed examples and open questions).

\subsubsection*{Right-decorated partitions}

For a set operad $\mathcal{Q}$ and a finite set $S$, a \emph{right-$\mathcal{Q}$-decorated partition} of $S$ is a pair $(\pi, (\xi_T)_{T \in \pi})$ where $\pi$ is a partition of $S$ and $\xi_T$ is an element of $\mathcal{P}(T)$ for every $T\in \pi$. (This is the framework studied by Vallette \cite{vallettepartitionposets}, and even earlier by Mendez and Yang \cite{MendezYang}.) The posets $\Pi^{\mathcal{Q}}$ of right-$\mathcal{Q}$-decorated partitions form an operadic poset species under the assumption that $\mathcal{Q}$ is ``right-basic''. If $\mathcal{Q}$ is Koszul, then $\Lambda h^\bullet(\Pi^{\mathcal{Q}})$ is the Koszul dual operad $\mathcal{Q}^!$. The right operadic module $\Lambda\htop{\bullet}(\Pi^{\mathcal{Q}})$ is rarely trivial because $\Pi^{\mathcal{Q}}(S)$ has many minimal elements, indexed by $\mathcal{Q}(S)$. 

It can be the case that $\Lambda \htop{\bullet}(\Pi^{\mathcal{Q}})$ is not concentrated in degree zero even when $\mathcal{Q}$ is Koszul. In this case, the posets $\Pi^{\mathcal{Q}}(S)$ are Cohen--Macaulay but the augmented posets $\Pi^{\mathcal{Q}}_+(S)$ obtained by adding a least element are not. This phenomenon arises already for the operad $\mathcal{Q}=\mathrm{As}$ encoding associative algebras (see \S\ref{sec: As right operadic module}). We also study the case of the operad $\mathcal{Q}=\mathrm{Perm}$ encoding permutative algebras in \S\ref{sec: perm right operadic module}, for which the right $\mathrm{PreLie}$-module $\Lambda \htop{\bullet}(\Pi^{\mathrm{Perm}})$ is concentrated in degree zero (see \S\ref{sec: open questions} for more detailed examples and open questions).\medskip

In \S\ref{sec: more examples} we study other examples of operadic poset species and their cohomological invariants:

\subsubsection*{Non-singleton boolean posets, and metabelian Lie algebras}
In \S\ref{sec:metablie} we study the operadic poset species $\mathrm{NS}$ of non-singleton boolean posets, simply obtained by removing the singletons from the boolean posets. We prove that these posets are Cohen--Macaulay and that the corresponding cohomology operad is isomorphic to the desuspension of the operad encoding metabelian Lie algebras \eqref{eq: intro NS metablie}.

\subsubsection*{Non-crossing $2$-partitions} 
In \S\ref{sec:NC2P} we study the operadic poset species $\Pi_2$ of non-crossing $2$-partitions. We prove that the operad $h^\bullet(\Pi_2)$ is the desuspension of an operad concentrated in degree zero; we determine the underlying $\mathfrak{S}$-module and prove that the natural morphism $\mathrm{Lie}\to \Lambda h^\bullet(\Pi_2)$ factors through $\mathrm{PreLie}$. We have not tried to compute the non-trivial left operadic module $\hbottom{\bullet}(\Pi_2)$.

\subsubsection*{Multilabeled trees, Kontsevich's operad of trees, and pre-Lie algebras}
In \S\ref{sec: MLT} we study the operadic poset species $\mathrm{MLT}$ of multilabeled trees and prove that the corresponding operad is Kontsevich's operad of trees, the suboperad of Kontsevich's operad of graphs spanned by trees. The rooted variant $\mathrm{MLRT}$ produces the $\mathrm{PreLie}$ operad by \cite{ChapotonLivernet}.

\renewcommand{\arraystretch}{1.5}
\begin{table}[h!]
    \centering
    \addtolength{\leftskip} {-.3cm}
    \begin{tabular}{|c|c|c|c|} \hline
     Operadic poset species $P$   & \begin{tabular}{@{}c@{}}Operad \\ $\Lambda h^\bullet(P)$ \end{tabular} & \begin{tabular}{@{}c@{}}Left module \\ $\Lambda \hbottom{\bullet}(P)$ \end{tabular}  & \begin{tabular}{@{}c@{}}Right module \\ $\Lambda \htop{\bullet}(P)$ \end{tabular}  \\
     \hline
      $\Pi$ (partitions, \S\ref{subsec: partitions and lie}) & $\operatorname{Lie} $ & - & - \\ \hline
      $^{\mathcal{P}}\Pi$ (left-decorated partitions, \S\ref{sec: Ldp}) & $\mathcal{P}^!$ if Koszul & \S \ref{subsec: Ldp mod} & - \\
      example: ${}^{\operatorname{As}}\Pi$ (\S\ref{sec: example left operads}) & $\operatorname{As}$ & \S \ref{subsubsec: left As} & - \\ \hline
        $\Pi^{\mathcal{Q}}$ (right-decorated partitions, \S\ref{sec: Rdp}) & $\mathcal{Q}^!$ if Koszul & -& \S \ref{subsec: Rdp mod} \\
        example: $\Pi^{\operatorname{As}}$ (\S\ref{sec: examples right operads}) & $\operatorname{As}$ & -& \S\ref{sec: As right operadic module} \\
        example: $\Pi^{\operatorname{Perm}}$ (\S\ref{sec: examples right operads}) & $\mathrm{PreLie}$ & -& \S\ref{sec: perm right operadic module}
        \\ \hline 
        ${}^{\mathcal{P}}\Pi^{\mathcal{Q}}$ (bi-decorated partitions, \S\ref{sec: Bdp}) & \S\ref{sec: Bdp} & \S\ref{sec: Bdp} & \S\ref{sec: Bdp} \\ \hline
$\mathrm{NS}$ (non-singleton boolean posets, \S\ref{sec:metablie}) & $\mathrm{MetabLie}$ & - & - \\ \hline
$\Pi_2$ (non-crossing $2$-partitions, \S\ref{sec:NC2P}) & \S\ref{sec:NC2P}  
& \S\ref{sec:NC2P} & - \\ \hline
$\mathrm{MLT}$ (multilabeled trees, \S\ref{sec: MLT}) & $\Lambda \mathrm{Tree}$ & \S\ref{sec: MLT} & - \\
$\mathrm{MLRT}$ (multilabeled rooted trees, \S\ref{sec: MLRT})& $\mathrm{PreLie}$ &  \S\ref{sec: MLRT} & - \\ \hline
$\mathrm{HT}$ (hypertrees, \cite{DOD2}) & $\mathrm{PostLie}$ & $\mathrm{PreLie}$ & - \\ \hline
    \end{tabular}
    \vspace{.2cm}
    \caption{Examples of operadic posets species and the corresponding operads and operadic modules (after operadic suspension). A dash indicates a trivial operadic module.}
    \label{tableRes}
\end{table} 

\subsection*{Open questions, variants, and relations to other works}

Many of the examples in Table \ref{tableRes} have not been  described explicitly, even though in some cases we were able to prove some of their properties. These open questions look like fun exercises that we hope some readers will be happy to try and solve. More open questions can be found in \S\ref{sec: open questions}.

A natural variant of our formalism is to work in the setting of non-symmetric operads, replacing partitions of a finite set with partitions of a linearly ordered finite set into intervals -- or similar objects for other types of operads. The cyclic/anticyclic case is also particularly relevant: in \cite{ogerhomologyhypertree} it is proved that the isomorphism \eqref{eq: intro HT prelie} is compatible with the action of the symmetric group $\mathfrak{S}_{n+1}$ in arity $n$.

We expect that our tools could produce ``bigger'' operads on Whitney cohomology groups of posets, defined as the direct sums of the cohomology groups of all intervals, or some variant. For instance, the case of partitions would produce the Gerstenhaber operad, which contains $\Lambda^{-1}\mathrm{Lie}$. These operads would typically have differentials induced by the connecting maps $b$ studied in \cite{dupontjuteau}.

Finally, our constructions should fit naturally into (a variant of) Coron's type of operad encoded by posets, where concatenation of chains, and more general operations, play a central role \cite{coron}.

\subsection*{Contents}

In \S\ref{Poset cohomology} we introduce the variants of poset cohomology and our main tool, the concatenation of (co)chains. The definition of operadic poset species and the construction of the corresponding operads and operadic modules is carried out in \S\ref{sec:main}. In \S\ref{sec: decorated partitions operads} and \S\ref{sec: decorated partitions modules} we study the examples provided by decorated partitions, and \S\ref{sec: more examples} contains more examples of our formalism. In \S\ref{sec: open questions} we collect some open questions that arose from the study of examples.

\subsection*{Convention and notation}

We fix a commutative ring with unit $\KK$.  All posets are assumed to be finite. All operads are implicitly assumed to be reduced, and all finite sets to be non-empty.

\subsection*{Acknowledgments}
We have benefited from helpful discussions with Frédéric Chapoton, Basile Coron, Bishal Deb, Vladimir Dotsenko, Matthieu Josuat-Vergès, Guillaume Laplante-Anfossi, Anibal Medina-Mardones, Bruce E. Sagan, Sheila Sundaram, and Bruno Vallette. We were supported by the ANR project ANR-20-CE40-0016 HighAGT. B. D.-O. was moreover supported by ANR projects ANR-20-CE40-0007 CARPLO and ANR-20-CE48-0010 SSS.

\section{Poset cohomology} \label{Poset cohomology}

    We review four natural constructions of poset cohomology and explain some of their properties. The only non-standard tool is the concatenation operation on (co)chains, which will be crucial in the construction of operads in the rest of the article.

    \subsection{Four variants of poset cohomology}

    \subsubsection{The classical case}
    Let $P$ be a poset and let $\Delta(P)$ denote its simplicial complex, whose vertex set is $P$ and whose $n$-simplices are the strict $n$-chains
    \begin{equation}\label{eq: chain}\gamma:\; x_0<x_1<\cdots <x_{n-1}<x_n
    \end{equation}
    in $P$. We let $C^\bullet(P)$ denote the simplicial cochain complex of $\Delta(P)$ with coefficients in $\KK$. We denote by $[\gamma]\in C^n(P)$ the (dual) basis element corresponding to an $n$-chain $\gamma$. The differential $d:C^n(P)\to C^{n+1}(P)$ is given by the formula
    \begin{equation*}\begin{split}
    d[\gamma] &= \sum_{y<x_0} [y<x_0<x_1<\cdots <x_{n-1}<x_n] \\ &+ \sum_{i=1}^{n} (-1)^{i} \sum_{x_{i-1}<y<x_{i}}[x_0<x_1<\cdots<x_{i-1}<y<x_{i}<\cdots <x_{n-1}<x_n] \\
    &+ (-1)^{n+1} \sum_{x_n<y}^{\phantom{-}}[x_0<x_1<\cdots <x_{n-1}<x_n<y].
    \end{split}\end{equation*}
    We let $H^\bullet(P)$ denote the cohomology of $C^\bullet(P)$, which is simply the cohomology of the simplicial complex $\Delta(P)$ with coefficients in $\KK$. We will in fact be more interested in variants of this construction arising from conditions on the first and/or last element of a chain. 

    \subsubsection{A first variant}

    A \emph{minimal element} (resp. \emph{maximal element}) of a poset $P$ is an element $x\in P$ such that there exists no element $y\in P$ with $y<x$ (resp. $y>x$). We let $\min(P)$ (resp. $\max(P)$) denote the set of minimal (resp. maximal) elements of $P$, and let
    $$\chains{\bullet}(P)\subset C^\bullet(P)$$
    denote the subcomplex spanned by the basis elements $[\gamma]$ for $\gamma$ a chain \eqref{eq: chain} such that $x_0\in\min(P)$ and $x_n\in \max(P)$. The formula for the differential in $\chains{\bullet}(P)$ simplifies as:
    \begin{equation}\label{eq: formula differential min max}
    d[\gamma] = \sum_{i=1}^{n} (-1)^{i} \sum_{x_{i-1}<y<x_{i}}[x_0<x_1<\cdots<x_{i-1}<y<x_{i}<\cdots <x_{n-1}<x_n].
    \end{equation}
    We let $\h{\bullet}(P)$ denote the cohomology of $\chains{\bullet}(P)$. It is easily computed in low degree:
    $$\h{0}(P) = \bigoplus_{x\,\in\, \min(P)\cap \max(P)}\KK[x] \hspace{1cm} \mbox{ and } \hspace{1cm} \h{1}(P) = \bigoplus_{\substack{x\,\in\, \min(P),\, y\,\in\, \max(P) ,\,  x \lessdot y}} \KK[x<y],$$
    where $\lessdot$ denotes the covering relation in $P$.

    These groups are best interpreted in terms of relative cohomology, as follows. Let $\Delta'(P)\subset \Delta(P)$ denote the simplicial subcomplex formed by the $n$-chains \eqref{eq: chain} which are such that $x_0\notin \min(P)$ or $x_n\notin \max(P)$. Then by definition,
    $$\chains{\bullet}(P)= \ker\left( C^\bullet(\Delta(P)) \To C^\bullet(\Delta'(P)) \right)$$
    is the relative cochain complex of the pair $(\Delta(P),\Delta'(P))$ and therefore $\h{\bullet}(P)$ is the relative cohomology of this pair with coefficients in $\KK$. 
    
    The point of view of \emph{relative} cohomology is the most natural one in our opinion. However, in case some readers suffer from an allergy to relative cohomology, we note that, by excision, $\h{\bullet}(P)$ is isomorphic to the reduced cohomology of the topological quotient $|\Delta(P)|/|\Delta'(P)|$. For $n\geq 2$, the cohomology group $\h{n}(P)$ can also be related to the classical (reduced) cohomology groups of the maximal intervals in $P$:
    $$\h{n}(P) \simeq \bigoplus_{\substack{x\,\in\, \min(P) ,\, y\,\in\, \max(P) ,\, x<y}} \widetilde{H}^{n-2}(P_{(x,y)}),$$ 
    where the isomorphism is induced by $[x<x_1<\cdots <x_{n-1}<y]\mapsto (-1)^n[x_1<\cdots <x_{n-1}]$.
    In particular, if $P$ has a least element $\hat{0}$ and a greatest element $\hat{1}$ then 
    $$\h{n}(P)\simeq \widetilde{H}^{n-2}(P\setminus \{\hat{0},\hat{1}\}).$$ 
    This interpretation has the major drawback of introducing an unnatural degree shift (besides being valid only for $n\geq 2$).

    \begin{rem}
    The relationship between the two topological interpretations is that $|\Delta(P)|/|\Delta'(P)|$ is homeomorphic to the double suspension of $|\Delta(P\setminus \{\hat{0},\hat{1}\})|$.
    \end{rem}

    \subsubsection{Two other variants}

    We let 
    $$\chainsbottom{\bullet}(P) \subset C^\bullet(P)$$
    denote the subcomplex spanned by the basis elements $[\gamma]$ for $\gamma$ a chain \eqref{eq: chain} such that $x_0\in \min(P)$. We let $\hbottom{\bullet}(P)$ denote the cohomology of $\chainsbottom{\bullet}(P)$. It is naturally interpreted as the relative cohomology of the pair $(\Delta(P), \widecheck{\Delta}'(P))$ where $\widecheck{\Delta}'(P)\subset \Delta(P)$ denotes the simplicial subcomplex formed by the $n$-chains \eqref{eq: chain} which are such that $x_0\notin \min(P)$. For an interpretation in terms of classical (reduced) cohomology, we have an isomorphism, for $n\geq 1$,
    \begin{equation} \label{Relhtop}
    \widecheck{h}^n(P) \simeq \bigoplus_{x\,\in\, \min(P)}\widetilde{H}^{n-1}(P_{>x}),    
    \end{equation}
    induced by $[x<x_1<\cdots <x_{n-1}<x_n]\mapsto (-1)^n[x_1<\cdots <x_{n-1}<x_n]$. Again, this has the drawback of introducing an unnatural degree shift.
    
    We let
    $$\chainstop{\bullet}(P)\subset C^\bullet(P)$$
    denote the subcomplex spanned by the basis elements $[\gamma]$ for $\gamma$ a chain \eqref{eq: chain} such that $x_n\in\max(P)$. 
    We let $\htop{\bullet}(P)$ denote the cohomology of $\chainstop{\bullet}(P)$. It is naturally interpreted as the relative cohomology of the pair $(\Delta(P), \widehat{\Delta}'(P))$ where $\widehat{\Delta}'(P)\subset \Delta(P)$ denotes the simplicial subcomplex formed by the $n$-chains \eqref{eq: chain} which are such that $x_n\notin \max(P)$. For an interpretation in terms of classical (reduced) cohomology, we have an isomorphism, for $n\geq 1$,
    \begin{equation}\label{Relhbot}
    \widehat{h}^n(P) \simeq \bigoplus_{y\,\in\, \max(P)}\widetilde{H}^{n-1}(P_{<y}),
    \end{equation}
    induced by $[x_0<x_1<\cdots <x_{n-1}<y]\mapsto [x_0<x_1<\cdots <x_{n-1}]$. Again, this has the drawback of introducing an unnatural degree shift.
    
    \begin{rem}
    The obvious inclusions of cochain complexes induce linear maps as in the following diagram.
    $$\xymatrixcolsep{.5cm}\xymatrixrowsep{.5cm}\diagram{
    & \htop{\bullet}(P) \ar[rd] & \\
    \h{\bullet}(P) \ar[ru]\ar[rd]&& H^\bullet(P) \\
    &\hbottom{\bullet}(P) \ar[ru]&
    }$$
    They are isomorphisms in degree zero but generally not in higher degree.
    \end{rem}
    
 \subsection{Functoriality}
    
        Let $f:P\to Q$ be a morphism of posets (i.e. for all $x\in P$, $x\leq y \implies f(x)\leq f(y)$). It induces a morphism of complexes $f^*:C^\bullet(Q)\to C^\bullet(P)$,
        given by the formula
        $$f^*[y_0<y_1<\cdots <y_{n-1}<y_n] = \sum_{f(x_i)=y_i}[x_0<x_1<\cdots <x_{n-1}<x_n],$$
        and therefore a map at the level of cohomology that we also denote by $f^*:H^\bullet(Q)\to H^\bullet(P)$. The latter is simply the pullback in cohomology induced by the simplicial map $\Delta(P)\to \Delta(Q)$. Further assumptions on $f$ are needed in order to have pullback maps for the other variants of poset cohomology.

        \begin{defi}\label{defi: min max compatible}
        A poset morphism $f:P\to Q$ is \emph{min-compatible} (resp. \emph{max-compatible}) if $f^{-1}(\min(Q))\subset \min(P)$ (resp. if $f^{-1}(\max(Q))\subset \max(P)$), and \emph{min-max-compatible} if it is both min-compatible and max-compatible.
        \end{defi}

        \begin{enumerate}[$\triangleright$]
        \item If $f$ is min-max-compatible then we have an induced morphism of complexes $f^*:\chains{\bullet}(Q)\to \chains{\bullet}(P)$ and therefore also in cohomology, $f^*:\h{\bullet}(Q)\to \h{\bullet}(P)$.
        \item If $f$ is min-compatible then we have an induced morphism of complexes $f^*:\chainsbottom{\bullet}(Q)\to \chainsbottom{\bullet}(P)$ and therefore also in cohomology, $f^*:\hbottom{\bullet}(Q)\to \hbottom{\bullet}(P)$.
        \item If $f$ is max-compatible then we have an induced morphism of complexes $f^*:\chainstop{\bullet}(Q)\to \chainstop{\bullet}(P)$ and therefore also in cohomology, $f^*:\htop{\bullet}(Q)\to \htop{\bullet}(P)$.
        \end{enumerate}
        These pullback maps can also be explained by the functoriality of relative cohomology.

        \begin{rem}\label{rem: min-max-compatible}
        Min-max-compatibility is a strong constraint and it might be desirable, in certain situations, to be able to consider functoriality of the cohomology functors $h^\bullet$, $\hbottom{\bullet}$, $\htop{\bullet}$ for more general morphisms of posets, at the cost of losing the functoriality of the cochain level functors $c^\bullet$, $\chainsbottom{\bullet}$, $\chainstop{\bullet}$. We will not need such a generalization in this article.
        \end{rem}

    \subsection{K\"{u}nneth morphisms}
    
        For two posets $P$, $Q$, there is a quasi-isomorphism of complexes
        $$C^\bullet(P)\otimes C^\bullet(Q) \To C^\bullet(P\times Q)$$
        defined by the formula
        $$[x_0<\cdots <x_m]\otimes [y_0<\cdots <y_n] \mapsto [(x_0,y_0)<\cdots < (x_m,y_0)<(x_m,y_1)<\cdots <(x_m,y_n)].$$
        It induces a K\"{u}nneth morphism at the level of cohomology,
        $$H^\bullet(P)\otimes H^\bullet(Q)\To H^\bullet(P\times Q),$$
        which is an isomorphism if $\KK$ is a field, or more generally if all cohomology groups $H^n(P)$ (or $H^n(Q)$) are free $\KK$-modules.
        Similar K\"{u}nneth morphisms exist at the level of the other variants of poset cohomology:
        $$\h{\bullet}(P)\otimes \h{\bullet}(Q)\To\h{\bullet}(P\times Q) \; ,\quad \hbottom{\bullet}(P)\otimes \hbottom{\bullet}(Q)\To \hbottom{\bullet}(P\times Q) \; , \quad \htop{\bullet}(P)\otimes \htop{\bullet}(Q)\To \htop{\bullet}(P\times Q).$$
        They are functorial and satisfy the obvious symmetry and associativity properties.

    \subsection{Concatenation}
    
        The only non-standard piece of structure on poset cohomology that we will use is the possibility to concatenate (co)chains.
        
        \subsubsection{Definition}
        
        For a poset $P$ and an element $x\in P$ we define a morphism of complexes
        $$\mu_x:c^\bullet(P_{\leq x})\otimes c^\bullet(P_{\geq x})\To c^\bullet(P)$$
        by the formula
        \begin{equation*}
        \begin{split}
        \mu_x([x_0<\cdots <x_{m-1}<x_m=x]\otimes[x=x_m & < x_{m+1}<\cdots <x_{m+n}]) \\
        &=  [x_0<\cdots <x_{m-1} <x_m<x_{m+1}<\cdots <x_{m+n}].
        \end{split}
        \end{equation*}
        One easily checks that it is a morphism of complexes by looking at the formula for the differential \eqref{eq: formula differential min max}. Therefore we get an induced morphism on cohomology:
        $$\mu_x:h^\bullet(P_{\leq x})\otimes h^\bullet(P_{\geq x})\To h^\bullet(P),$$
        called a \emph{concatenation morphism}.

        \subsubsection{Compatibilities}

        Concatenation morphisms satisfy the following basic properties.

        \begin{prop}[Associativity]
        Let $P$ be a poset and let $x,x'\in P$ with $x\leq x'$. The following diagram commutes. 
        $$\xymatrixcolsep{2cm}\xymatrixrowsep{2cm}\diagram{
        h^\bullet(P_{\leq x})\otimes h^\bullet(P_{[x,x']})\otimes h^\bullet(P_{\geq x'}) \ar[r]^-{\mu_x\otimes\mathrm{id}}\ar[d]_-{\mathrm{id}\otimes \mu_{x'}} & h^\bullet(P_{\leq x'})\otimes h^\bullet(P_{\geq x'}) \ar[d]^-{\mu_{x'}} \\
        h^\bullet(P_{\leq x})\otimes h^\bullet(P_{\geq x}) \ar[r]_-{\mu_x} & h^\bullet(P)
         }$$
        \end{prop}

        \begin{proof}
        By direct inspection.
        \end{proof}

        In the context of the proposition we will denote by 
        \begin{equation}\label{eq: definition mu double}
        \mu_{x,x'}:h^\bullet(P_{\leq x})\otimes h^\bullet(P_{[x,x']})\otimes h^\bullet(P_{\geq x'}) \To h^\bullet(P)
        \end{equation}
        the morphism $\mu_{x'}\circ(\mu_x\otimes\mathrm{id})=\mu_x\circ (\mathrm{id}\otimes \mu_{x'})$.

        \begin{prop}[Compatibility with pullback]\label{prop: concatenation and pullback}
        Let $f:P\to Q$ be a morphism of posets and $y\in Q$. For $x\in f^{-1}(\{y\})$ we let $f_{\leq x}:P_{\leq x}\to Q_{\leq y}$ and $f_{\geq x}:P_{\geq x}\to Q_{\geq y}$ denote the morphisms of posets induced by $f$. Assume that $f$ and all the $f_{\leq x}$ and $f_{\geq x}$ are min-max-compatible. 
        Then the following diagram commutes.
        $$\xymatrixcolsep{2cm}\xymatrixrowsep{2cm}\diagram{
        h^\bullet(Q_{\leq y})\otimes h^\bullet(Q_{\geq y}) \ar[r]^-{\mu_y} \ar[d]_-{\bigoplus_x f_{\leq x}^*\otimes f_{\geq x}^*} & h^\bullet(Q) \ar[d]^-{f^*} \\
        \displaystyle\bigoplus_{x\in f^{-1}(\{y\})} h^\bullet(P_{\leq x})\otimes h^\bullet(P_{\geq x}) \ar[r]_-{\bigoplus_x \mu_x}& h^\bullet(P)
         }$$
        \end{prop}

        \begin{proof}
        By direct inspection.
        \end{proof}

        \begin{prop}[Compatibility with K\"{u}nneth]\label{prop: concatenation and kunneth}
        Let $P,Q$ be two posets and $x\in P, y\in Q$. The following diagram commutes, where the vertical arrows are induced by K\"{u}nneth morphisms.
        $$\xymatrixcolsep{2cm}\xymatrixrowsep{2cm}\diagram{
        h^\bullet(P_{\leq x})\otimes h^\bullet(P_{\geq x})\otimes h^\bullet(Q_{\leq y})\otimes h^\bullet(Q_{\geq y}) \ar[r]^-{\mu_x\otimes \mu_y} \ar[d] & h^\bullet(P)\otimes h^\bullet(Q) \ar[d] \\
        h^\bullet(P_{\leq x}\times Q_{\leq y})\otimes h^\bullet(P_{\geq x}\times Q_{\geq y}) \ar[r]_-{\mu_{(x,y)}} & h^\bullet(P\times Q)
        }
        $$
        \end{prop}

        \begin{proof}
        We can assume that $P$ has a greatest element $\hat{1}$ and $Q$ has a least element $\hat{0}$. One checks that the composition involving $\mu_x\otimes \mu_y$ equals the composition
        $$\xymatrixcolsep{.6cm}\diagram{
        h^\bullet(P_{\leq x})\otimes h^\bullet(P_{\geq x})\otimes h^\bullet(Q_{\leq y})\otimes h^\bullet(Q_{\geq y}) \ar[r] & h^\bullet(P_{\leq x})\otimes h^\bullet(P_{\geq x}\times Q_{\leq y}) \otimes h^\bullet(Q_{\geq y}) \ar[r] & h^\bullet(P\times Q)
        }
        $$
        where the first map is induced by a K\"{u}nneth morphism and the second map is the double concatenation morphism \eqref{eq: definition mu double} for the elements $(x,\hat{0})$ and $(\hat{1},y)$. On the other hand, one checks that the composition involving $\mu_{(x,y)}$ equals the composition
        $$\xymatrixcolsep{.6cm}\diagram{
        h^\bullet(P_{\leq x})\otimes h^\bullet(P_{\geq x})\otimes h^\bullet(Q_{\leq y})\otimes h^\bullet(Q_{\geq y}) \ar[d]_-{\sim} && \\
        h^\bullet(P_{\leq x})\otimes h^\bullet(Q_{\leq y})\otimes h^\bullet(P_{\geq x})\otimes h^\bullet(Q_{\geq y}) \ar[r] & h^\bullet(P_{\leq x})\otimes h^\bullet(Q_{\leq y}\times P_{\geq x}) \otimes h^\bullet(Q_{\geq y}) \ar[r] & h^\bullet(Q\times P) \ar[d]^-{\sim}\\
        && h^\bullet(P\times Q)
        }
        $$
        where the first horizontal map is induced by a K\"{u}nneth morphism and the second horizontal map is a double concatenation morphism \eqref{eq: definition mu double} for the elements $(\hat{0},x)$ and $(y,\hat{1})$. The commutativity of the diagram therefore follows from the symmetry of K\"{u}nneth morphisms.
        \end{proof}

        \subsubsection{Variants}

        We will also use the variants
        $$\widecheck{\mu}_x:h^\bullet(P_{\leq x})\otimes \hbottom{\bullet}(P_{\geq x}) \To \hbottom{\bullet}(P) \qquad \mbox{ and } \qquad \widehat{\mu}_x:\htop{\bullet}(P_{\leq x})\otimes h^\bullet(P_{\geq x}) \To \htop{\bullet}(P),$$
        which are given by the same formulas and satisfy similar compatibilities. For the sake of completeness we also mention the variant 
        $$\htop{\bullet}(P_{\leq x})\otimes \hbottom{\bullet}(P_{\geq x})\To H^\bullet(P),$$
        which we will not use.

\section{Operadic poset species}\label{sec:main}

    We set up a formalism of ``families of posets with an operadic structure'', and construct operadic structures on their cohomology.

    \subsection{Poset species, and partition posets}
    
       A \emph{poset species} is a functor from the category of finite sets and bijections to the category of posets. We will write $P:S\mapsto P(S)$ for such a poset species, and for a bijection $f:S\to S'$ between finite sets, $P(f):P(S)\to P(S')$ for the induced isomorphism of posets. We will use the classical shorthand notation
       $$P(n) \defas P(\{1,\ldots,n\}).$$
       The fundamental example is the poset species of partitions 
        $$\Pi:S\mapsto \Pi(S).$$
        Recall that a partition of $S$ is a set $\pi$ of non-empty subsets of $S$, called the \emph{blocks} of $\pi$, which are pairwise disjoint and whose union equals $S$, and that $\Pi(S)$ denotes the poset of all partitions of $S$, where we take the convention that $\alpha\leq \beta$ if $\alpha$ is obtained by merging blocks of $\beta$. We warn the reader that the opposite convention is often found in the literature.

        Let us reinterpret the order relation on partitions while introducing some notation. For a partition $\pi$ of a finite set $S$, and a subset $T\subset S$, we let 
        $$\pi_{|T} \defas \pi\cap 2^T,$$
        whose elements are the blocks of $\pi$ contained in $T$. For $\alpha,\beta\in\Pi(S)$ we have the equivalent definitions:
        $$\alpha\leq \beta \quad \Longleftrightarrow \quad \{\beta_{|T},\, T\in\alpha\} \mbox{ is a partition of } \beta$$
        $$\alpha\leq \beta \quad \Longleftrightarrow \quad \forall\, T\in \alpha, \, \beta_{|T} \mbox{ is a partition of } T.$$
        This leads to the following fundamental structure of the partition posets. For every partition $\pi\in \Pi(S)$ we have isomorphisms of posets
        \begin{equation}\label{eq: phi and psi for partitions}
        \varphi_\pi: \Pi_{\leq \pi}(S)  \stackrel{\sim}{\To} \Pi(\pi) \hspace{2cm} \mbox{and} \hspace{2cm} \psi_\pi: \Pi_{\geq\pi}(S) \stackrel{\sim}{\To} \prod_{T\in \pi} \Pi(T)
        \end{equation}
        defined by $\alpha\mapsto \{\pi_{|T},T\in\alpha\}$ and $\beta\mapsto (\beta_{|T})_{T\in \pi}$ respectively.

        \begin{example}
        Let $S=\{a,b,c,d,e,f,g\}$ be a $7$-element set and let $\pi=\{T_1,T_2,T_3\}$ with $T_1=\{a,b,c\}$, $T_2=\{d,e\}$, $T_3=\{f,g\}$. We identify $\pi$ with the set $\{1,2,3\}$.
        The isomorphism
        $$\varphi_\pi:\Pi_{\leq \pi}(S) \stackrel{\sim}{\To} \Pi(\{1,2,3\})$$
        sends, for instance, the partition $\{T_1\cup T_2, T_3\}=\{\{a,b,c,d,e\},\{f,g\}\}$ to the partition $\{\{1,2\},\{3\}\}$ of $\{1,2,3\}$. The isomorphism
        $$\psi_\pi:\Pi_{\geq \pi}(S)\stackrel{\sim}{\To} \Pi(T_1)\times \Pi(T_2)\times \Pi(T_3)$$
        sends, for instance, the partition $\{\{a\},\{b,c\},\{d\},\{e\},\{f,g\}\}$ to $(\{\{a\},\{b,c\}\}, \{\{d\},\{e\}\}, \{\{f,g\}\})$.
        \end{example}

        For $\pi,\pi'\in\Pi(S)$ with $\pi\leq \pi'$, it will be convenient to write
        $$\pi/\pi' = \varphi_{\pi'}(\pi) = \{\pi'_{|T}, T \in \pi\} \;\; \in\Pi(\pi').$$
        Indeed, viewing the partition $\pi'$ as an equivalence relation on $S$ and its underlying set of blocks as the quotient $S/\pi'$, we can view $\pi/\pi'$ as the partition induced by $\pi$ on that quotient.
        
        The compatibilities between the maps $\varphi$ and $\psi$ are summarized in the following proposition, which will later be upgraded to the \emph{definition} of operadic poset species. 

        \begin{prop}
        Let $S$ be a finite set, and let $\pi,\pi'\in\Pi(S)$ with $\pi\leq \pi'$.
        \begin{enumerate}[$\triangleright$]
        \item (Composition of $\varphi$'s.) The isomorphism $\varphi_{\pi'}$ induces an isomorphism, denoted by the same symbol, between $\Pi_{\leq \pi}(S)$ and $\Pi_{\leq \pi/\pi'}(\pi')$, which fits into the following commutative diagram.
        \begin{equation}\label{eq: composition of phis for partitions}
        \begin{gathered}
        \diagram{
        \Pi_{\leq \pi}(S) \ar[rr]^{\varphi_\pi} \ar[rd]_{\varphi_{\pi'}} && \Pi(\pi) \\
        & \Pi_{\leq \pi/\pi'}(\pi') \ar[ru]_{\varphi_{\pi/\pi'}}&
        }
        \end{gathered}
        \end{equation}
        \item (Composition of $\psi$'s.) The isomorphism $\psi_\pi$ induces an isomorphism, denoted by the same symbol, between $\Pi_{\geq \pi'}(S)$ and the product of the $\Pi_{\geq \pi'_{|T}}(T)$ for $T\in\pi$, which fits into the following commutative diagram.
        \begin{equation}\label{eq: composition of psis for partitions}
        \begin{gathered}
        \diagram{
        \Pi_{\geq \pi'}(S) \ar[rr]^{\psi_{\pi'}} \ar[rd]_{\psi_\pi} && \displaystyle\prod_{U\in\pi'}\Pi(U) \\
        &\displaystyle\prod_{T\in \pi}\Pi_{\geq {\pi}'_{|T}}(T) \ar[ru]_{\left(\psi_{{\pi}'_{|T}}\right)_{T\in \pi}}&
        }
        \end{gathered}
        \end{equation}
        \item (Composition of $\varphi$'s and $\psi'$s.) The isomorphism $\varphi_{\pi'}$ induces an isomorphism, denoted by the same symbol, between $\Pi_{[\pi,\pi']}(S)$ and $\Pi_{\geq \pi/\pi'}(\pi')$. The isomorphism $\psi_\pi$ induces an isomorphism, that we denote by the same symbol, between $\Pi_{[\pi,\pi']}(S)$ and the product of the $\Pi_{\leq \pi'_{|T}}(T)$ for $T\in\pi$. They fit into the following commutative diagram.
        \begin{equation}\label{eq: composition of phis and psis for partitions}
        \begin{gathered}
        \xymatrixcolsep{2cm}\xymatrixrowsep{1.5cm}\diagram{
        \Pi_{[\pi,\pi']}(S) \ar[r]^-{\varphi_{\pi'}} \ar[d]_-{\psi_\pi} & \Pi_{\geq {\pi/\pi'}}(\pi') \ar[d]^-{\psi_{{\pi/\pi'}}}\\
        \displaystyle\prod_{T\in \pi}\Pi_{\leq {\pi}'_{|T}}(T) \ar[r]_-{\left(\varphi_{{\pi}'_{|T}}\right)_{T\in \pi}}& \displaystyle\prod_{T\in \pi}\Pi(\pi'_{|T})
        }
        \end{gathered}
        \end{equation}
        \end{enumerate}
        \end{prop}

        \begin{proof}
            \begin{enumerate}[$\triangleright$]
                \item (Composition of $\varphi$'s.)
                The first claim is obvious. For $\alpha\in \Pi_{\leq \pi}(S)$ we have the equalities
                $$(\varphi_{\pi/\pi'} \circ \varphi_{\pi'})(\alpha) = \varphi_{\pi/\pi'}(\alpha/\pi') = (\alpha/\pi')/(\pi/\pi') = \alpha/\pi = \varphi_\pi(\alpha),$$
                which give the commutativity of \eqref{eq: composition of phis for partitions}.
                \item (Composition of $\psi$'s.) The first claim is obvious. For $\beta\in \Pi_{\pi'}(S)$ we have the equalities
                $$\left(\left(\psi_{{\pi'}_{|T}} \right)_{T \in \pi} \circ \psi_{\pi}\right)(\beta) = \left(\psi_{\pi'_{|T}}(\beta_{|T})\right)_{T\in\pi} = \left((\beta_{|T})_{|U}\right)_{U\in\pi'} = \left(\beta_{|U}\right)_{U\in\pi'}=\psi_{\pi'}(\beta),$$
                which give the commutativity of \eqref{eq: composition of psis for partitions}.
                \item (Composition of $\varphi$'s and $\psi$'s.) The first claim is obvious. For $\gamma\in\Pi_{[\pi,\pi']}(S)$ we have the equalities
                $$ \left(\left(\varphi_{{\pi'}_{|T}} \right)_{T \in \pi} \circ \psi_{\pi}\right)(\gamma) = \left(\varphi_{\pi'_{|T}}(\gamma_{|T})\right)_{T\in\pi} = \left(\gamma_{|T}/\pi'_{|T}\right)_{T\in\pi}$$
                and
                $$\left(\psi_{\pi/\pi'}  \circ \varphi_{\pi'}\right)(\gamma) = \psi_{\pi/\pi'}(\gamma/\pi') = \left((\gamma/\pi')_{|T}\right)_{T\in\pi/\pi'}.$$
                The natural identification between the sets (of blocks of) $\pi$ and $\pi/\pi'$ is compatible with restriction in the sense that  $\left( \gamma / \pi'\right)_{|T}= {\gamma_{|T}}/\pi'_{|T}$ for all $T\in\pi$, and the commutativity of \eqref{eq: composition of phis and psis for partitions} follows.
                \end{enumerate}
        \end{proof}
        
    \subsection{Operadic posets species}
    
        Let $P$ be a poset species equipped with a morphism of poset species
        $$\diagram{P \ar[d]_-{a} \\ \Pi}$$
        that we think of as a ``structure morphism''. We will refer to $a(x)$ as the \emph{underlying partition} of $x$. We require that for every finite set $S$ the morphism of posets $a:P(S)\to \Pi(S)$ satisfies
        \begin{equation}\label{eq: axiom a min max compatible strong}
        a^{-1}(\min(\Pi(S)))=\min(P(S)) \quad \hbox{ and } \quad a^{-1}(\max(\Pi(S)))=\max(P(S)).
        \end{equation}
        This assumption is strictly stronger that min-max-compatibility (Definition \ref{defi: min max compatible}) and will ensure that other morphisms of posets considered below are min-max-compatible.

        For an element $x\in P(S)$ with underlying partition $a(x)=\pi$, we consider morphisms of posets
        $$\varphi_x:P_{\leq x}(S)\To P(\pi) \quad\quad\quad \mbox{ and } \quad\quad\quad \psi_x:P_{\geq x}(S)\To \prod_{T\in\pi}P(T).$$
        We require the following diagrams to commute.
        \begin{equation}\label{eq: commutative squares phi psi a}
        \begin{aligned}
        \xymatrixcolsep{1.5cm}\xymatrixrowsep{1.5cm}\diagram{
        P_{\leq x}(S) \ar[r]^-{\varphi_x} \ar[d]_{a} & P(\pi) \ar[d]^{a}\\
        \Pi_{\leq \pi}(S) \ar[r]^-{\sim}_-{\varphi_\pi} & \Pi(\pi)
        } \hspace{3cm}
        \xymatrixcolsep{1.5cm}\xymatrixrowsep{1cm}\diagram{
        P_{\geq x}(S) \ar[r]^-{\psi_x} \ar[d]_{a} & \displaystyle\prod_{T\in \pi}P(T) \ar[d]^-{a}\\
        \Pi_{\geq \pi}(S) \ar[r]^-{\sim}_-{\psi_\pi} & \displaystyle\prod_{T\in \pi}\Pi(T)
        }
        \end{aligned}
        \end{equation}
        
        \begin{defi}\label{def: operadic poset species}
        We say that $P$ equipped with $a$ and the $\varphi_x, \psi_x$ is an \emph{operadic poset species} if the axioms of equivariance, unitality, and associativity below are satisfied.
        \end{defi}

        \subsubsection{Equivariance}\label{axiom: equivariance}
        
        Let $f:S\to S'$ be a bijection, let $x\in P(S)$, and let $x'=P(f)(x)$, with underlying partitions $\pi=a(x)$ and $\pi'=a(x')$. Since $a$ is a morphism of species, the blocks of $\pi'$ are the $T'=f(T)$ for $T$ a block of $\pi$. We let $f:\pi\to \pi'$ denote the bijection $T\mapsto f(T)$, and for a block $T\in \pi$ we let $f_{|T}:T\to f(T)$ denote the natural bijection. We require the following diagrams to commute. 
        $$\xymatrixcolsep{1.5cm}\xymatrixrowsep{1.5cm}\diagram{
        P_{\leq x}(S) \ar[r]^{\varphi_x} \ar[d]_{P(f)} & P(\pi) \ar[d]^{P(f)}\\
        P_{\leq x'}(S') \ar[r]_{\varphi_{x'}} & P(\pi')
        }\hspace{2cm}\xymatrixcolsep{1.5cm}\xymatrixrowsep{1.5cm}\diagram{
        P_{\geq x}(S) \ar[r]^{\psi_x} \ar[d]_{P(f)} & \displaystyle \prod_{T\in\pi}P(T) \ar[d]^{(P(f_{|T}))_{T\in\pi}} \\
        P_{\geq x'}(S') \ar[r]_{\psi_{x'}} & \displaystyle\prod_{T'\in \pi'}P(T')
        } $$
        
        \subsubsection{Unitality}\label{axiom: unitality}

        For a singleton $\{s\}$, we require $P(\{s\})$ to be a singleton: $P(\{s\})=\{1_s\}$.

        \begin{rem}
        This axiom could easily be relaxed by requiring the existence of a distinguished element $1_s \in P(\{s\})$ for every singleton $\{s\}$, such that for $f:\{s\}\to \{t\}$ we have $P(f)(1_s)=1_t$, and such that those elements behave well with respect to the morphisms $\varphi_x$ for $x$ with underlying partition $a(x)=\hat{0}$ and $\psi_x$ for $x$ with underlying partition $a(x)=\hat{1}$.
        \end{rem}
                    
        \subsubsection{Associativity}\label{axiom: associativity}

        Let $S$ be a finite set and let $x\leq x'$ be elements of $P(S)$, with underlying partitions $\pi=a(x)\leq \pi'=a(x')$. To simplify notation we identify $\pi$ with $\varphi_{\pi'}(\pi)$ and therefore view $\pi$ as the partition of $\pi'$ whose blocks are the $\pi'_{|T}\defas\pi'\cap 2^T$, for all blocks $T\in \pi$. 

        \begin{enumerate}[$\triangleright$]
        \item (Composition of $\varphi$'s.) We let $$\underline{x}\defas\varphi_{x'}(x) \;\in P(\pi').$$ 
        Since $\varphi_{x'}$ is a morphism of posets it sends $P_{\leq x}(S)$ to $P_{\leq \underline{x}}(\pi')$, and we denote the corresponding restriction by the same symbol:
        $$\varphi_{x'}:P_{\leq x}(S)\To P_{\leq \underline{x}}(\pi').$$
        We require the following diagram to commute.
        \begin{equation}\label{eq: axiom composition of phis}
        \begin{gathered}
        \diagram{
        P_{\leq x}(S) \ar[rr]^{\varphi_x} \ar[rd]_{\varphi_{x'}} && P(\pi) \\
        & P_{\leq \underline{x}}(\pi') \ar[ru]_{\varphi_{\underline{x}}}&
        }
        \end{gathered}
        \end{equation}
        \item (Composition of $\psi$'s.) We let 
        $$(x'_{|T})_{T\in \pi}\defas\psi_x(x') \;\in \prod_{T\in\pi}P(T).$$ 
        Since $\psi_x$ is a morphism of posets it sends $P_{\geq x'}(S)$ to the product of the $P_{\geq x'_{|T}}(T)$ for $T\in\pi$, and we denote the corresponding restriction by the same symbol:
        $$\psi_{x}:P_{\geq x'}(S) \To \prod_{T\in\pi} P_{\geq x'_{|T}}(T).$$
        We require the following diagram to commute.
        \begin{equation}\label{eq: axiom composition of psis}
        \begin{gathered}
        \diagram{
        P_{\geq x'}(S) \ar[rr]^{\psi_{x'}} \ar[rd]_{\psi_x} && \displaystyle\prod_{U\in\pi'}P(U) \\
        &\displaystyle\prod_{T\in \pi}P_{\geq x'_{|T}}(T) \ar[ru]_{\left(\psi_{x'_{|T}}\right)_{T\in \pi}}&
        }
        \end{gathered}
        \end{equation}
        \item (Composition of $\varphi$'s and $\psi$'s.) Since $\varphi_{x'}$ is a morphism of posets it sends $P_{[x,x']}(S)$ to $P_{\geq \underline{x}}(\pi')$ and we denote the corresponding restriction by the same symbol:
        $$\varphi_{x'}: P_{[x,x']}(S) \To P_{\geq \underline{x}}(\pi').$$
        Similarly, since $\psi_x$ is a morphism of posets it sends $P_{[x,x']}(S)$ to the product of the $P_{\leq x'_{|T}}(T)$ for $T\in\pi$, and we denote the corresponding restriction by the same symbol:
        $$\psi_x:P_{[x,x']}(S) \To \prod_{T\in \pi}P_{\leq x'_{|T}}(T).$$
        We require the following diagram to commute.
        \begin{equation}\label{eq: axiom composition of phis and psis}
        \begin{gathered}
        \xymatrixcolsep{2cm}\xymatrixrowsep{1.5cm}\diagram{
        P_{[x,x']}(S) \ar[r]^-{\varphi_{x'}} \ar[d]_-{\psi_x} & P_{\geq \underline{x}}(\pi') \ar[d]^-{\psi_{\underline{x}}}\\
        \displaystyle\prod_{T\in \pi}P_{\leq x'_{|T}}(T) \ar[r]_-{\left(\varphi_{x'_{|T}}\right)_{T\in \pi}}& \displaystyle\prod_{T\in \pi}P(\pi'_{|T})
        }
        \end{gathered}
        \end{equation}
        The morphism corresponding to the diagonal of this commutative square will be denoted by
        \begin{equation}\label{eq: definition xi}
        \xi_{x,x'}: P_{[x,x']}(S)\To \prod_{T\in\pi} P(\pi'_{|T}).
        \end{equation}
        \end{enumerate}

        \begin{rem}\label{rem: strong min max compatibility for all}
        The axiom \eqref{eq: axiom a min max compatible strong} implies that all morphisms $\varphi_x$, $\psi_x$, along with the morphisms $\varphi_x'$ from \eqref{eq: axiom composition of phis}, $\psi_x$ from \eqref{eq: axiom composition of psis}, and $\xi_{x,x'}$ from \eqref{eq: definition xi}, satisfy the same strong min-max-compatibility.
        \end{rem}
        
    \subsection{Algebraic operads in the language of partitions}

        We refer the reader to \cite{lodayvallette} for basic definitions and properties of algebraic operads. A \emph{linear species} is a functor from the category of finite sets and bijections to the category of $\KK$-modules. The category of linear species has a monoidal structure $\circ$ defined by the formula
        $$(\mathrm{F}\circ \mathrm{G})(S)=\bigoplus_{\pi\in \Pi(S)} \left(\mathrm{F}(\pi)\otimes \bigotimes_{T\in \pi}\mathrm{G}(T)\right).$$
        Its unit $\mathrm{I}$ is the linear species that sends $S$ to $0$ if $|S|\neq 1$ and to $\KK=\KK S$ if $|S|=1$. 
        
        Recall that an \emph{algebraic operad} (or \emph{operad} for simplicity) is by definition an algebra in the monoidal category of linear species. In other words, it is given by a linear species $\mathrm{F}$ and morphisms of linear species
        $$\rho:\mathrm{F}\circ \mathrm{F}\to \mathrm{F} \quad\quad\quad \mbox{ and } \quad\quad\quad \eta:\mathrm{I}\to \mathrm{F}$$
        satisfying the usual axioms. Concretely, an operad is given by a linear species $\mathrm{F}$ and a collection of linear \emph{operadic composition} morphisms
        \begin{equation}\label{eq:operadic-composition}
        \rho_\pi\colon \mathrm{F}(\pi)\otimes \bigotimes_{T\in\pi}\mathrm{F}(T) \To \mathrm{F}(S)
        \end{equation}
        for all finite sets $S$ and all partitions $\pi\in\Pi(S)$, which are required to satisfy the axioms of equivariance, unitality, and associativity spelled out in the next paragraphs.

        \begin{rem}
        The operadic composition morphisms can all be obtained from the \emph{partial} operadic composition morphisms defined as follows. For disjoint finite sets $A$ and $B$, let $\pi$ denote the partition of $A\sqcup B$ whose blocks are $B$ and the singletons $\{a\}$ for $a\in A$. We identify the underlying set of $\pi$, or equivalently the quotient of $A\sqcup B$ by $\pi$, with the set $A\sqcup\{*\}$. Inserting the identities $1_a$, for $a\in A$, in the composition morphism $\rho_\pi$ gives rise to the partial operadic composition morphism
        \begin{equation}\label{eq:partial-operadic-composition}
        \circ_*\colon \diagram{
        \mathrm{F}(A\sqcup\{*\}) \otimes \mathrm{F}(B)  \To \mathrm{F}(A\sqcup B).
        }
        \end{equation}
        \end{rem} 
        
        \begin{rem}
        We will also consider \emph{(differential) graded} operads, which are algebras in the monoidal category of (differential) graded linear species. The monoidal structure on (differential) graded linear species is based on the symmetric monoidal structure on (differential) graded $\KK$-modules, for which we use the sign-graded symmetry convention. Concretely, for (differential) graded $\KK$-modules $V, W$, the symmetry isomorphism $V\otimes W\stackrel{\sim}{\to} W\otimes V$ is given, for homogeneneous elements $v\in V^a$ and $w\in W^b$, by the formula $v\otimes w\mapsto (-1)^{ab}w\otimes v$.
        \end{rem}
        
        \subsubsection{Equivariance} 
        
        Let $f:S\to S'$ be a bijection, and $\pi\in \Pi(S)$. We require the following diagram to commute, where the notation is the same as in \S\ref{axiom: equivariance}.
        $$\xymatrixcolsep{2cm}\xymatrixrowsep{1.5cm}\diagram{
        \mathrm{F}(\pi)\otimes\displaystyle\bigotimes_{T\in\pi}\mathrm{F}(T) \ar^-{\rho_\pi}[r] \ar[d]_-{\mathrm{F}(f)\otimes \bigotimes_{T\in\pi}\mathrm{F}(f_{|T})} & \mathrm{F}(S) \ar^-{\mathrm{F}(f)}[d]  \\
        \mathrm{F}(\pi')\otimes\displaystyle\bigotimes_{T'\in\pi'}\mathrm{F}(T') \ar_-{\rho_{\pi'}}[r] & \mathrm{F}(S') 
        }$$
        
        \subsubsection{Unitality} 
        
        There exists, for every one-element set $\{s\}$, a distinguished element 
        $$1_s \in \mathrm{F}(\{s\})$$
        such that for $f:\{s\}\to \{t\}$ we have $\mathrm{F}(f)(1_s)=1_t$, and such that the following conditions hold for every finite set $S$.
        \begin{enumerate}[$\triangleright$]
        \item If $\hat{0}$ denotes the partition of $S$ with only one block, then the composition
        $$\xymatrixcolsep{2cm}\diagram{
        \mathrm{F}(S) \ar[r]^-{1_S\otimes\mathrm{id}} & \mathrm{F}(\{S\})\otimes \mathrm{F}(S) = \mathrm{F}(\hat{0})\otimes \mathrm{F}(S) \ar[r]^-{\rho_{\hat{0}}} & \mathrm{F}(S)
        }$$
        is the identity.
        \item If $\hat{1}$ denotes the partition of $S$ whose blocks are singletons, then the composition
        $$\xymatrixcolsep{2cm}\diagram{
        \mathrm{F}(S) \ar[r]^-{\mathrm{id}\otimes\bigotimes_{s\in S}1_s} & \mathrm{F}(S)\otimes \displaystyle\bigotimes_{s\in S}\mathrm{F}(\{s\}) \simeq \mathrm{F}(\hat{1})\otimes \displaystyle\bigotimes_{s\in S}\mathrm{F}(\{s\}) \ar[r]^-{\rho_{\hat{1}}} & \mathrm{F}(S)
        }$$
        is the identity.
        \end{enumerate}

        \subsubsection{Associativity}

        Let $S$ be a finite set, and let $\pi,\pi'\in\Pi(S)$ with $\pi\leq \pi'$. The following diagram is required to commute.
        $$\xymatrixcolsep{3cm}\xymatrixrowsep{1.5cm}\diagram{
        \mathrm{F}(\pi)\otimes \displaystyle\bigotimes_{T\in\pi}\mathrm{F}(\pi'_{|T}) \otimes \displaystyle\bigotimes_{U\in\pi'}\mathrm{F}(U) \ar[r]^-{\mathrm{id}\otimes\bigotimes_{T\in\pi}\rho_{\pi'_{|T}}} \ar[d]_-{\rho_\pi\otimes\mathrm{id}} & \mathrm{F}(\pi)\otimes \displaystyle\bigotimes_{T\in\pi}\mathrm{F}(T) \ar[d]^-{\rho_\pi} \\
        \mathrm{F}(\pi')\otimes\displaystyle\bigotimes_{U\in\pi'}\mathrm{F}(U) \ar[r]_-{\rho_{\pi'}} & \mathrm{F}(S)
        }$$
        Note that to make sense of the upper horizontal arrow we make the identification
        $$\mathrm{F}(\pi)\otimes \displaystyle\bigotimes_{T\in\pi}\mathrm{F}(\pi'_{|T}) \otimes \displaystyle\bigotimes_{U\in\pi'}\mathrm{F}(U) = \mathrm{F}(\pi) \otimes \bigotimes_{T\in\pi}\left( \mathrm{F}(\pi'_{|T})\otimes\bigotimes_{U\in\pi'_{|T}}\mathrm{F}(U) \right).$$

     \subsection{The operad structure in cohomology}

        Let $P$ be an operadic poset species in the sense of Definition \ref{def: operadic poset species}. The functor $S\mapsto h^\bullet(P(S))$ is a graded linear species that we denote by $h^\bullet(P)$. Let $S$ be a finite set and $\pi\in\Pi(S)$ be a partition of $S$. We obtain a morphism
        $$\rho_\pi: h^{\bullet}(P(\pi)) \otimes \bigotimes_{T\in\pi} h^{\bullet}(P(T)) \To h^{\bullet}(P(S))$$
        by composing the K\"{u}nneth morphism
        $$\xymatrixcolsep{1.4cm}\diagram{ h^{\bullet}(P(\pi))\otimes \displaystyle\bigotimes_{T\in \pi}h^{\bullet}(P(T)) \ar[r]^-{\operatorname{id}\otimes\, \kappa} & h^{\bullet}(P(\pi))\otimes h^{\bullet}\left(\displaystyle\prod_{T\in \pi}P(T)\right)}$$
         and the sum over all elements $x\in P(S)$ with $a(x)=\pi$ of the compositions
         $$\xymatrixcolsep{2cm}\diagram{h^{\bullet}(P(\pi))\otimes h^{\bullet}\left(\displaystyle\prod_{T\in \pi}P(T)\right) \ar[r]^-{\varphi_x^*\,\otimes\,\psi_x^*} &  h^{\bullet}(P_{\leq x}(S))\otimes h^{\bullet}(P_{\geq x}(S)) \ar[r]^-{\mu_x} & h^\bullet(P(S))}.$$
       
        \begin{thm}\label{thm: main}
        The morphisms $\rho_\pi$ give $\h{\bullet}(P)$ the structure of a graded operad of $\KK$-modules.
        \end{thm}

        The proof of this theorem will be given below in \S\ref{subsec: proof of main theorem}.

        \begin{rem}
        The operadic composition maps $\rho_\pi$ are naturally defined at cochain level, i.e. they are induced by morphisms
        \begin{equation}\label{eq: operadic composition at cochain level}
        c^\bullet(P(\pi)) \otimes \bigotimes_{T\in\pi} c^\bullet(P(T)) \To c^\bullet(P(S)).
        \end{equation}
        However, those do not define the structure of a differential graded operad on $c^\bullet(P)$. The reason is simply that K\"{u}nneth morphisms are neither commutative nor associative ``on the nose''. One should therefore think of $c^\bullet(P)$ as a differential graded operad \emph{up to homotopy}. In the sequel \cite{DOD2} we will explain how De Concini and Procesi's building sets \cite{deConciniProcesi, feichtnerkozlov} can be used to produce strict operadic models in the spirit of \cite{coron}, and use them to study the operadic poset species of hypertrees.
        \end{rem}

        \begin{rem}\label{rem: partial operadic composition poset species}
        The \emph{partial} operadic composition morphisms \eqref{eq:partial-operadic-composition} of $h^\bullet(P)$ are easier to compute because they do not require the K\"{u}nneth morphism. For disjoint finite sets $A$ and $B$, let $\pi$ denote the partition of $A\sqcup B$ whose blocks are $B$ and the singletons $\{a\}$ for $a\in A$. We identify the underlying set of $\pi$, or equivalently the quotient of $A\sqcup B$ by $\pi$, with the set $A\sqcup\{*\}$. For any element $x\in P(A\sqcup B)$ whose underlying partition is $a(x)=\pi$, we therefore have morphisms of posets
        $$\varphi_x:P_{\leq x}(A\sqcup B) \To P(A\sqcup \{*\}) \qquad \mbox{ and } \qquad \psi_x:P_{\geq x}(A\sqcup B)\To P(B),$$
        where for the latter we have used the fact that $P(\{a\})$ is a singleton for every $a\in A$. The partial composition morphism \eqref{eq:partial-operadic-composition} for $h^\bullet(P)$ is obtained as the sum over all such elements $x$ of the compositions
        $$\xymatrixcolsep{1.5cm}\diagram{h^{\bullet}(P(A\sqcup \{*\}))\otimes h^{\bullet}(P(B)) \ar[r]^-{\varphi_x^*\,\otimes\,\psi_x^*} &  h^{\bullet}(P_{\leq x}(A\sqcup B))\otimes h^{\bullet}(P_{\geq x}(A\sqcup B)) \ar[r]^-{\mu_x} & h^\bullet(P(A\sqcup B))}.$$
        \end{rem}

        \begin{exple}\label{ex: operadic composition partitions} 
        We illustrate the construction of the operad $h^\bullet(P)$ in the case $P=\Pi$ consisting of the partition posets. In this case the map $a$ is the identity and the maps $\varphi_\pi$ and $\psi_\pi$ are the isomorphisms \eqref{eq: phi and psi for partitions}. We will see below (Proposition \ref{prop: partitions and Lie}) that the operad $h^\bullet(\Pi)$ is the (operadic desuspension of the) Lie operad. Let us use Remark \ref{rem: partial operadic composition poset species} to compute the first non-trivial partial operadic composition,
        $$\circ_*\colon h^1(\Pi(\{1,*\}))\otimes h^1(\Pi(\{2,3\})) \To h^2(\Pi(\{1,2,3\})).$$
        It corresponds to the partition $\pi=1|23$ of $\{1,2,3\}$, for which we have
        $$\varphi_\pi^*([1*<1|*]) = [123<1|23] \qquad \mbox{ and } \qquad \varphi_\pi^*([23<2|3]) = [1|23<1|2|3].$$
        We therefore get:
        $$[1*<1|*] \circ_* [23<2|3] = [123 < 1|23 < 1|2|3].$$
        \end{exple}

\subsection{Variant: operadic modules}

        By using the variant $\widecheck{\mu}_x$ of the concatenation morphisms, we obtain morphisms
        $$\widecheck{\rho}_\pi:\h{\bullet}(P(\pi)) \otimes \bigotimes_{T\in\pi} \hbottom{\bullet}(P(T)) \To \hbottom{\bullet}(P(S)).$$
        
        \begin{thm}\label{thm: main left}
        The morphisms $\widecheck{\rho}_\pi$ give $\hbottom{\bullet}(P)$ the structure of a left operadic module over $\h{\bullet}(P)$.
        \end{thm}

        By using the variant $\widehat{\mu}_x$ of the concatenation morphisms, we obtain morphisms
        $$\widehat{\rho}_\pi:\htop{\bullet}(P(\pi)) \otimes \bigotimes_{T\in\pi} \h{\bullet}(P(T)) \To \htop{\bullet}(P(S)).$$

        \begin{thm}\label{thm: main right}
        The morphisms $\widehat{\rho}_\pi$ give $\htop{\bullet}(P)$ the structure of a right operadic module over $\h{\bullet}(P)$.
        \end{thm}

        Theorems \ref{thm: main left} and \ref{thm: main right} are proved by adapting the proof of Theorem \ref{thm: main} in the obvious way.

        \begin{rem}
        The left (resp. right) operadic module $\hbottom{\bullet}(P)$ (resp. $\htop{\bullet}(P)$) is by definition a quotient of $h^\bullet(P)$, i.e. is a \emph{cyclic} left (resp. right) operadic module.
        \end{rem}

\subsection{Functoriality and Lie-operad structure}\label{sec: functoriality}

    Let $P$ and $Q$ be two operadic poset species. There is an obvious notion of a morphism of operadic poset species $f:P\to Q$. It amounts to a morphism of poset species which commutes with the structure morphisms $a$ and with the morphisms $\varphi$ and $\psi$. In particular, the commutativity of the following diagram implies that $f$ is min-max-compatible (Definition \ref{defi: min max compatible}) because of \eqref{eq: axiom a min max compatible strong}.
    $$\diagram{
    P \ar[rr]^-{f} \ar[rd]_-{a} && Q \ar[ld]^-{a}\\
    &\Pi&
    }$$
    By using the functoriality of K\"{u}nneth morphisms and concatenation morphisms (Proposition \ref{prop: concatenation and pullback}) one easily sees that the pullback morphisms
    $$f^*:h^\bullet(Q)\to h^\bullet(P) , \qquad  f^*:\hbottom{\bullet}(Q)\to \hbottom{\bullet}(P), \qquad f^*:\htop{\bullet}(Q)\to\htop{\bullet}(P)$$
    give rise to morphisms of graded operads and compatible morphisms of graded left and right operadic modules.
    
    The category of operadic poset species has a terminal object, which is $\Pi$, the poset species of partitions \eqref{eq: phi and psi for partitions}. Proposition \ref{prop: partitions and Lie} below therefore implies that for every operadic poset species $P$, the cohomology operad $h^\bullet(P)$ comes equipped with morphism of graded operads
    $$a^*:\Lambda^{-1}\operatorname{Lie} \to h^\bullet(P)$$
    from the operadic desuspension of the $\operatorname{Lie}$ operad. Upon operadic suspension, we obtain a morphism of graded operads
    $$\operatorname{Lie} \to \Lambda h^\bullet(P)$$
    which makes $\Lambda h^\bullet(P)$ into a (graded) \emph{Lie-operad}.

\subsection{Products}

    Let $P_1$ and $P_2$ be two operadic poset species and form the following fiber product (cartesian square).
    $$\diagram{
    P_1 \underset{\Pi}{\times} P_2 \ar[r]\ar[d] & P_1 \ar[d]^{a} \\
    P_2 \ar[r]_{a} & \Pi
    }$$
    Concretely, an element of $(P_1\underset{\Pi}{\times} P_2)(S)$ is a pair $(x_1,x_2)\in P_1(S)\times P_2(S)$ such that $a(x_1)=a(x_2)$. We view $(P_1\underset{\Pi}{\times} P_2)(S)$ as a subposet of the product poset $P_1(S)\times P_2(S)$.

    \begin{prop} \label{prop: fiber product}
    The fiber product $P_1\underset{\Pi}{\times} P_2$ is naturally equipped with the structure of an operadic poset species, that we simply call the \emph{product} of $P_1$ and $P_2$.
    \end{prop}
    
    \begin{example}
    The product of the operadic poset species of left-decorated partitions ${}^{\mathcal{P}}\Pi$, defined in \S\ref{sec: Ldp}, and that of right-decorated partitions $\Pi^{\mathcal{Q}}$, defined in \S\ref{sec: Rdp}, is the operadic poset species of bi-decorated partitions ${}^{\mathcal{P}}\Pi^{\mathcal{Q}}$ defined in \S\ref{sec: Bdp}.
    \end{example}    

    \begin{proof}[Proof of Proposition \ref{prop: fiber product}]
    The natural morphism of posets $a:P_1\underset{\Pi}{\times} P_2\to \Pi \; , \; (x_1,x_2)\mapsto a(x_1)=a(x_2)$ is strictly increasing.
    For $x_1\in P_1(S)$ and $x_2\in P_2(S)$ such that $a(x_1)=a(x_2)=\pi$, we define a morphism of posets
    $$\varphi_{(x_1,x_2)}:(P_1\underset{\Pi}{\times} P_2)_{\leq (x_1,x_2)}(S) = (P_1)_{\leq x_1}(S) \underset{\Pi_{\leq \pi}(S)}{\times} (P_2)_{\leq x_2}(S) \To P_1(\pi)\underset{\Pi(\pi)}{\times} P_2(\pi).$$
    by the obvious formula $\varphi_{(x_1,x_2)}(y_1,y_2)=(\varphi_{x_1}(y_1),\varphi_{x_2}(y_2))$. The fact that it is well-defined is a consequence of the commutativity of the first diagram in \eqref{eq: commutative squares phi psi a} for $P_1$ and $P_2$. Similarly, we define a morphism of posets 
    $$\psi_{(x_1,x_2)}: ( P_1\underset{\Pi}{\times} P_2 )_{\geq (x_1,x_2)} (S) = (P_1)_{\geq x_1}(S) \underset{\Pi_{\geq \pi}(S)}{\times} (P_2)_{\geq x_2}(S)\To \prod_{T \in \pi} P_1(T) \underset{\Pi(T)}{\times} P_2(T).$$
    The commutativity of \eqref{eq: commutative squares phi psi a} for $P_1\underset{\Pi}{\times} P_2$ is clear, as well as the equivariance and unitality axioms. The commutativity axioms \eqref{eq: axiom composition of phis}, \eqref{eq: axiom composition of psis}, \eqref{eq: axiom composition of phis and psis} can be checked componentwise.
    \end{proof}

\subsection{Proof of Theorem \ref{thm: main}}\label{subsec: proof of main theorem}
    The equivariance axiom easily follows from the functoriality of K\"{u}nneth morphisms and composition morphisms (Proposition \ref{prop: concatenation and pullback}). The unitality axiom is clear. For the associativity axiom, let $S$ be a finite set and $\pi,\pi'\in\Pi(S)$ with $\pi\leq \pi'$. We let
        $$A=\rho_{\pi'}\circ (\rho_\pi\otimes\mathrm{id}) \qquad \mbox{ and } \qquad B=\rho_\pi\circ \left(\mathrm{id}\otimes\bigotimes_{T\in\pi}\rho_{\pi'_{|T}}\right).$$
        We break the proof of the equality $A=B$ into two parts:
        $$A=\rho_{\pi,\pi'} \quad \mbox{ and } \qquad B=\rho_{\pi,\pi'}.$$
        Here we are using as an intermediary the morphism (``symmetrical double operadic composition'')
        $$\rho_{\pi,\pi'}:h(P(\pi))\otimes\bigotimes_{T\in \pi}h(P(\pi'_{|T}))\otimes \bigotimes_{U\in \pi'}h(P(U)) \To h(P(S))$$
        obtained by composing the K\"{u}nneth morphism
        \begin{equation}\label{eq: double kunneth in proof}
        \xymatrixcolsep{1.35cm}
        \diagram{
        h(P(\pi))\otimes\displaystyle\bigotimes_{T\in \pi}h(P(\pi'_{|T}))\otimes \displaystyle\bigotimes_{U\in \pi'}h(P(U)) \ar[r]^-{\mathrm{id}\otimes\kappa\otimes\kappa} & h(P(\pi))\otimes h\left(\displaystyle\prod_{T\in \pi}P(\pi'_{|T})\right) \otimes h\left(\displaystyle\prod_{U\in\pi'}P(U)\right)
        }
        \end{equation}
        and the sum over pairs $(x,x')$ of elements of $P(S)$ with $x\leq x'$, $a(x)=\pi$, $a(x')=\pi'$, of the compositions of the morphisms
        $$\xymatrixcolsep{2.3cm}
        \diagram{
        h(P(\pi))\otimes h\left(\displaystyle\prod_{T\in \pi}P(\pi'_{|T})\right) \otimes h\left(\displaystyle\prod_{U\in\pi'}P(U)\right) \ar[r]^-{\varphi_x^*\otimes \xi_{x,x'}^*\otimes \psi_{x'}^*} & h(P_{\leq x}(S))\otimes h(P_{[x,x']}(S))\otimes h(P_{\geq x'}(S))
        }$$
        and
        $$\xymatrixcolsep{1.5cm}
        \diagram{h(P_{\leq x}(S))\otimes h(P_{[x,x']}(S))\otimes h(P_{\geq x'}(S)) \ar[r]^-{\mu_{x,x'}} & h(P(S)).
        }$$
        \medskip
        
        \subsubsection{Proof of $A=\rho_{\pi,\pi'}$}

        Note that $A$ is a sum of morphisms $A_{y,x'}$ indexed by pairs $(y,x')$ with $y\in P(\pi')$ such that $a(y)=\pi$ and $x'\in P(S)$ such that $a(x')=\pi'$. Concretely, $A_{y,x'}$ is obtained by composing the K\"{u}nneth morphism \eqref{eq: double kunneth in proof} and the morphisms in the following diagram.
        $$
        \xymatrixcolsep{1.5cm}\xymatrixrowsep{2cm}\diagram{
        h(P(\pi))\otimes h\left(\displaystyle\prod_{T\in \pi}P(\pi'_{|T})\right) \otimes h\left(\displaystyle\prod_{U\in\pi'}P(U)\right) \ar[d]_-{\varphi_y^*\otimes \psi_y^*\otimes\psi_{x'}^*} && \\
        h(P_{\leq y}(\pi'))\otimes h(P_{\geq y}(\pi'))\otimes h(P_{\geq x'}(S)) \ar[r]^-{\mu_y\otimes\mathrm{id}} & h(P(\pi'))\otimes h(P_{\geq x'}(S)) \ar[d]_-{\varphi_{x'}^*\otimes\mathrm{id}} & \\
        & h(P_{\leq x'}(S))\otimes h(P_{\geq x'}(S)) \ar[r]^-{\mu_{x'}} & h(P(S))
        }
        $$
        By applying Proposition \ref{prop: concatenation and pullback} to $\varphi_{x'}:P_{\leq x'}(S)\to P(\pi')$ and $y$ we get the following commutative diagram. (Note that the assumptions of Proposition \ref{prop: concatenation and pullback} are indeed satisfied by Remark \ref{rem: strong min max compatibility for all}.)
        $$\xymatrixcolsep{3cm}\xymatrixrowsep{2cm}\diagram{
        h(P_{\leq y}(\pi'))\otimes h(P_{\geq y}(\pi')) \ar[r]^-{\mu_y} \ar[d]_-{\bigoplus_x \varphi_{x'}^*\otimes\varphi_{x'}^*} & h(P(\pi')) \ar[d]^-{\varphi_{x'}^*} \\
        \displaystyle\bigoplus_{x\in P_{\leq x'}(S), \, \varphi_{x'}(x)=y} h(P_{\leq x}(S))\otimes h(P_{[x,x']}(S)) \ar[r]_-{\bigoplus_x \mu_x}   & h(P_{\leq x'}(S))
        }$$
        It allows to rewrite $A$ as a sum of morphisms $A_{x,x'}$ indexed by pairs $(x,x')$ of elements of $P(S)$ with $x\leq x'$, $a(x)=\pi$, and $a(x')=\pi'$. More specifically, $A_{x,x'}$ is obtained by composing the K\"{u}nneth morphism \eqref{eq: double kunneth in proof} and the morphisms in the following diagram.
        $$
        \xymatrixcolsep{1.5cm}\xymatrixrowsep{2cm}\diagram{
        h(P(\pi))\otimes h\left(\displaystyle\prod_{T\in \pi}P(\pi'_{|T})\right) \otimes h\left(\displaystyle\prod_{U\in\pi'}P(U)\right) \ar[d]_-{\varphi_{\underline{x}}^*\otimes \psi_{\underline{x}}^*\otimes\psi_{x'}^*}  && \\
        h(P_{\leq\underline{x}}(\pi'))\otimes h(P_{\geq \underline{x}}(\pi'))\otimes h(P_{\geq x'}(S)) \ar[d]_-{\varphi_{x'}^*\otimes \varphi_{x'}^*\otimes\mathrm{id}} &  & \\
        h(P_{\leq x}(S))\otimes h(P_{[x,x']}(S))\otimes h(P_{\geq x'}(S)) \ar[r]^-{\mu_x\otimes \mathrm{id}}& h(P_{\leq x'}(S))\otimes h(P_{\geq x'}(S)) \ar[r]^-{\mu_{x'}} & h(P(S))
        }
        $$
        The vertical composition in this diagram equals $\varphi_x^*\otimes \xi_{x,x'}^*\otimes \psi_{x'}^*$ by the commutativity of \eqref{eq: axiom composition of phis} and the definition of $\xi_{x,x'}$ \eqref{eq: definition xi}, while the horizontal composition equals $\mu_{x,x'}$ by definition \eqref{eq: definition mu double}. This concludes the proof of the equality $A= \rho_{\pi,\pi'}$.
        \medskip
                
        \subsubsection{Proof of $B=\rho_{\pi,\pi'}$} 
        
        Note that $B$ is a sum of morphisms $B_{x,y'}$ indexed by pairs $(x,y')$ with $x\in P(S)$ such that $a(x)=\pi$ and $y'=(y'_T)_{T\in \pi'}\in \prod_{T\in \pi}P(T)$ such that $a(y'_T)=\pi'_{|T}$ for all $T\in \pi'$. Concretely, $B_{x,y'}$ is obtained by composing the K\"{u}nneth morphism \eqref{eq: double kunneth in proof} and the morphisms in the following diagram, where we use the notation $\varphi_{y'}=(\varphi_{y'_T})_{T\in\pi}$ and $\psi_{y'}=(\psi_{y'_T})_{T\in\pi}$.
        $$
        \xymatrixcolsep{0.5cm}\xymatrixrowsep{2cm}\diagram{
        h(P(\pi))\otimes h\left(\displaystyle\prod_{T\in \pi}P(\pi'_{|T})\right) \otimes h\left(\displaystyle\prod_{U\in\pi'}P(U)\right) \ar[d]_-{\varphi_x^*\otimes\varphi_{y'}^*\otimes\psi_{y'}^*} &&& \\
        h(P_{\leq x}(S))\otimes h\left(\displaystyle\prod_{T\in\pi}P_{\leq y'_T}(T)\right)\otimes h\left(\displaystyle\prod_{T\in\pi}P_{\geq y'_T}(T)\right) \ar[rr]^-{\mathrm{id}\otimes \mu_{y'}} && h(P_{\leq x}(S))\otimes h\left(\displaystyle\prod_{T\in \pi}P(T)\right) \ar[d]_-{\mathrm{id}\otimes \psi_x^*} & \\
        && h(P_{\leq x}(S))\otimes h(P_{\geq x}(S)) \ar[r]^-{\mu_x}& h(P(S))
        }
        $$
        To be precise, this description of $B_{x,y'}$ follows from the compatibility of concatenation morphisms with K\"{u}nneth morphisms (Proposition \ref{prop: concatenation and kunneth}) and the functoriality of K\"{u}nneth morphisms.
        By applying Proposition \ref{prop: concatenation and pullback} to $\psi_x:P_{\geq x}(S)\to \prod_{T\in \pi}P(T)$ and $y'$ we get the following commutative diagram. (Note that the assumptions of Proposition \ref{prop: concatenation and pullback} are indeed satisfied by Remark \ref{rem: strong min max compatibility for all}.)
        $$\xymatrixcolsep{3cm}\xymatrixrowsep{2cm}\diagram{
        h\left(\displaystyle\prod_{T\in \pi}P_{\leq y'_T}(T)\right)\otimes h\left(\displaystyle\prod_{T\in\pi}P_{\geq y'_T}(T)\right) \ar[r]^-{\mu_{y'}} \ar[d]_-{\bigoplus_{x'}\psi_{x'}^*\otimes \psi_{x'}^*} & h\left(\displaystyle\prod_{T\in \pi}P(T)\right) \ar[d]^-{\psi_x^*} \\
        \displaystyle\bigoplus_{x'\in P_{\geq x}(S),\psi_x(x')=y'}h(P_{[x,x']}(S))\otimes h(P_{\geq x'}(S)) \ar[r]_-{\bigoplus_{x'}\mu_{x'}} & h(P_{\geq x}(S))
        }$$
        It allows to rewrite $B$ as a sum of morphisms $B_{x,x'}$ indexed by pairs $(x,x')$ of elements of $P(S)$ with $x\leq x'$, $a(x)=\pi$, $a(x')=\pi'$. More specifically, $B_{x,x'}$ is obtained by composing the K\"{u}nneth morphism \eqref{eq: double kunneth in proof} and the morphisms in the following diagram.
        $$
        \xymatrixcolsep{1.1cm}\xymatrixrowsep{2cm}\diagram{
        h(P(\pi))\otimes h\left(\displaystyle\prod_{T\in \pi}P(\pi'_{|T})\right) \otimes h\left(\displaystyle\prod_{U\in\pi'}P(U)\right) \ar[d]_-{\varphi_x^*\otimes\varphi_{\underline{x}'}^*\otimes\psi_{\underline{x}'}^*} && \\
        h(P_{\leq x}(S))\otimes h\left(\displaystyle\prod_{T\in\pi}P_{\leq x'_{|T}}(T)\right)\otimes h\left(\displaystyle\prod_{T\in\pi}P_{\geq x'_{|T}}(T)\right)\ar[d]_-{\mathrm{id}\otimes \psi_{x'}^*\otimes \psi_{x'}^*}  &  & \\
        h(P_{\leq x}(S))\otimes h(P_{[x,x']}(S))\otimes h(P_{\geq x'}(S)) \ar[r]^-{\mathrm{id}\otimes \mu_{x'}} & h(P_{\leq x}(S))\otimes h(P_{\geq x}(S)) \ar[r]^-{\mu_x}& h(P(S))
        }
        $$
        The vertical composition in this diagram equals $\varphi_x^*\otimes \xi_{x,x'}^*\otimes \psi_{x'}^*$ by the commutativity of \eqref{eq: axiom composition of psis} and the definition of $\xi_{x,x'}$ \eqref{eq: definition xi}, while the horizontal composition equals $\mu_{x,x'}$ by definition \eqref{eq: definition mu double}. This concludes the proof of the equality $B=\rho_{\pi,\pi'}$.

\section{Decorated partitions and their cohomology operads}\label{sec: decorated partitions operads}

    After Fresse's analysis \cite{fressepartitionposets} of the case of the operads $\mathrm{Com}$ and $\mathrm{Lie}$, Vallette \cite{vallettepartitionposets} interpreted the bar construction of certain operads as a complex that computes the cohomology of posets of decorated partitions. The latter had already been considered by Mendez and Yang \cite[Lemma 3.1]{MendezYang} without the operadic language. We recast this framework in our formalism, calling the objects at play \emph{right-decorated partitions}, first developing the dual case of \emph{left-decorated partitions}.

    \subsection{Partitions and the Lie operad}\label{subsec: partitions and lie}

        We start by reviewing the classical case of (undecorated) partitions: the partition posets $\Pi(S)$ form an operadic poset species $\Pi$, whose cohomology $h^\bullet(\Pi)$ is naturally isomorphic to the (desuspension of the) operad encoding Lie algebras (Proposition \ref{prop: partitions and Lie} below).
        
        Recall (from, e.g. \cite[\S 7.2.2]{lodayvallette}) the suspension of graded vector spaces $(sV)^\bullet=V^{\bullet-1}$, and the operadic suspension
        $$(\Lambda\mathcal{P})(n)=s^{1-n}\mathrm{sgn}_n\otimes\mathcal{P}(n),$$
        for $\mathcal{P}$ a differential graded operad. It is such that $\Lambda\mathcal{P}$-algebras are suspensions of $\mathcal{P}$-algebras. Its inverse operation, the operadic desuspension, is defined as
        $$(\Lambda^{-1}\mathcal{P})(n)=s^{n-1}\mathrm{sgn}_n\otimes\mathcal{P}(n).$$
        The following result is classical, and was first remarked by Joyal \cite{joyal} who compared the symmetric group representations appearing in the cohomology of partition posets \cite{Han81, Stan82} on the one hand and in free Lie algebras on the other \cite{brandt}.
        However, the point of view is slightly different here: Theorem \ref{thm: main} puts an \emph{a priori} operadic structure on $h^\bullet(\Pi)$, and the next proposition \emph{computes} it.
        
        \begin{prop}\label{prop: partitions and Lie}
        There exists a unique isomorphism of graded operads
        $$\Phi: \Lambda^{-1}\operatorname{Lie} \stackrel{\sim}{\To} h^\bullet(\Pi)$$
        that sends the Lie bracket to the generator $[12<1|2]$ of $h^1(\Pi(2))$. 
        \end{prop}
        
        \begin{proof}
        Note that the symmetric group $\mathfrak{S}_2$ acts trivially on the element $[12<1|2]$. We next check that the Jacobi identity is satisfied in $h^2(\Pi(3))$. Recalling the computation of $\Phi([1,[2,3]])$ from Example \ref{ex: operadic composition partitions}, we have:
        \begin{equation*}
\Phi([1,[2,3]]+ [2,[3,1]] + [3,[1,2]]) =  [123<1|23<1|2|3] + [123<2|13<1|2|3] + [123<3|12<1|2|3],   
        \end{equation*}
        which is the class of the coboundary $d\,[1|2|3<123]$, and therefore zero. It is well-known on the one hand that the $\KK$-module $\operatorname{Lie}(n)$ is free of rank $(n-1)!$, and on the other hand that $h^i(\Pi(n))=0$ for $i\neq n-1$ and $h^{n-1}(\Pi(n))$ is free of rank $(n-1)!$ by \cite{bjornershellable}. Therefore, it remains to prove that the morphism $\mathrm{sgn}_n\otimes\,\operatorname{Lie}(n)\to h^{n-1}(\Pi(n))$ induced by $\Phi$ is surjective. This is easily done by the same argument as in \cite[Lemma 3.1, Definition 3.2, Theorem 3.11]{hanlonwachs}.
        \end{proof}
        
    \subsection{Left-decorated partitions} \label{sec: Ldp}

        All set operads are implicitly assumed to be finite in each arity.

        \subsubsection{Definition}

        \begin{defi}
        Let $\mathcal{P}$ be a set operad satisfying $\mathcal{P}(0)=\varnothing$ and $\mathcal{P}(1)=\{*\}$. A \emph{left-$\mathcal{P}$-decorated partition} of a finite set $S$ is a pair $(\pi,\xi)$ where $\pi$ is a partition of $S$ and $\xi\in \mathcal{P}(\pi)$. We denote by ${}^{\mathcal{P}}\Pi(S)$ the set of left-$\mathcal{P}$-decorated partitions of $S$.
        \end{defi}

        \begin{rem}
        The terminology comes from the fact that ${}^{\mathcal{P}}\Pi$ equals the composition of species $\mathcal{P}\circ\mathbb{E}^+$, with $\mathcal{P}$ on the left, where $\mathbb{E}^+$ denotes the species of non-empty sets.
        \end{rem}

         We define an order relation on left-$\mathcal{P}$-decorated partitions by saying that
         $$(\alpha,\xi) \leq (\beta,\eta)$$
        if $\alpha$ is obtained by merging blocks of $\beta$ and $\eta$ is obtained from $\xi$ by operadic composition. More formally, this means that $\alpha\leq \beta$ as partitions and there exists a collection of elements $\mu_T\in \mathcal{P}(\beta_{|T})$, for $T\in\alpha$, such that $\eta=\xi\circ(\mu_T)_{T\in\alpha}$. 
        One checks that this gives each ${}^{\mathcal{P}}\Pi(S)$ the structure of a poset, and $S\mapsto {}^{\mathcal{P}}\Pi(S)$ the structure of a poset species. There is a morphism of poset species
        $$\diagram{{}^{\mathcal{P}}\Pi\;\; \ar[d]_a \\ \Pi}$$
        which forgets the decorations: $a(\pi,\xi)=\pi$.

    \subsubsection{Left-decorated partitions as an operadic poset species}

        In order to make left-decorated partitions into an operadic poset species, the following assumption on the decorating operad $\mathcal{P}$ is natural.

        \begin{defi}
        A set operad $\mathcal{P}$ satisfying $\mathcal{P}(0)=\varnothing$ and $\mathcal{P}(1)=\{*\}$ is said to be \emph{left-basic} if for every finite set $S$ and every left-$\mathcal{P}$-decorated partition $(\pi,\xi)$ of $S$, the operadic composition map
        \begin{equation}\label{eq: axiom left basic}
        \prod_{T\in\pi}\mathcal{P}(T)\To \mathcal{P}(S) \; , \; (\mu_T)_{T\in\pi} \mapsto \xi\circ (\mu_T)_{T\in\pi}
        \end{equation}
        is injective.
        \end{defi} 

        Let now $\mathcal{P}$ be a left-basic set operad. Let $(\pi,\xi)$ be a left-$\mathcal{P}$-decorated partition of a finite set $S$. The morphism of posets
        $$\varphi_{(\pi,\xi)}:{}^{\mathcal{P}}\Pi_{\leq (\pi,\xi)}(S)\To {}^{\mathcal{P}}\Pi(\pi)$$
        is easy to define and does not use the fact that $\mathcal{P}$ is left-basic. It sends a left-$\mathcal{P}$-decorated partition $(\alpha,\mu)\leq (\pi,\xi)$, to $(\varphi_\pi(\alpha)=\alpha/\pi,\widetilde{\mu})$ where the decoration $\widetilde{\mu}\in\mathcal{P}(\alpha/\pi)$ is induced by $\mu\in\mathcal{P}(\alpha)$ via the bijection $\alpha\simeq \alpha/\pi,T\leftrightarrow \pi_{|T}$. One can note that $\varphi_{(\pi,\xi)}$ induces an isomorphism of posets between ${}^{\mathcal{P}}\Pi_{\leq (\pi,\xi)}(S)$ and the interval ${}^{\mathcal{P}}\Pi_{\leq (\hat{1},\xi)}(\pi)$. We now define a morphism of posets
        \begin{equation} \label{eq:DefPsiLeftDec}
            \psi_{(\pi,\xi)}:{}^{\mathcal{P}}\Pi_{\geq (\pi,\xi)}(S)\To \prod_{T\in \pi} {}^{\mathcal{P}}\Pi(T).
        \end{equation}
        An element of ${}^{\mathcal{P}}\Pi_{\geq (\pi,\xi)}(S)$ is a left-$\mathcal{P}$-decorated partition $(\beta, \eta)$ where $\beta\geq \pi$ and such that there exist elements $\mu_T\in \mathcal{P}(\beta_{|T})$, for $T\in\pi$, such that $\eta=\xi\circ (\mu_T)_{T\in\pi}$. Since $\mathcal{P}$ is left-basic, $(\mu_T)_{T\in\pi}$ is unique and we can define the image of $(\beta,\eta)$ by $\psi_{(\pi,\xi)}$ to consist of the partitions $\beta_{|T}$, for $T\in\pi$, decorated by the $\mu_T$. One can note that $\psi_{(\pi,\xi)}$ is an isomorphism of posets.

        \begin{prop}
        For $\mathcal{P}$ a left-basic set operad, this structure makes ${}^{\mathcal{P}}\Pi$ into an operadic poset species.
        \end{prop}

        \begin{proof}
        The structure morphisms $a$ satisfy \eqref{eq: axiom a min max compatible strong}. The commutativity of the diagrams \eqref{eq: commutative squares phi psi a} is obvious, as are the equivariance and unitality axioms. We are left with proving the associativity axioms. In each case we only need to trace the decorations since the diagrams obviously commute at the level of the underlying partitions. Let $(\pi,\xi)\leq (\pi',\xi')$ be left-$\mathcal{P}$-decorated partitions of a finite set $S$.
        \begin{enumerate}[$\triangleright$]
        \item (Composition of $\varphi$'s.) This is clear since for every partition $\alpha\leq \pi$ the composition of the natural bijections $\alpha\simeq \varphi_{\pi'}(\alpha)\simeq \varphi_{\pi/\pi'}(\varphi_{\pi'}(\alpha))$ is the natural bijection $\alpha\simeq \varphi_\pi(\alpha)$.
        \item (Composition of $\psi$'s.) Let us write $\xi'=\xi\circ(\mu_T)_{T\in\pi}$. For $(\beta,\nu)\geq (\pi',\xi')$, we can write $\nu=\xi\circ (\sigma_T)_{T\in\pi}$, and then, for every $T\in\pi$, $\sigma_T=\mu_T\circ(\sigma')_{U\in\pi'_{|T}}$. The bottom composition in \eqref{eq: axiom composition of psis} sends $(\beta,\nu)$ to the collection of the $\beta_{|U}$, for $U\in\pi'$, decorated by the $\sigma'_U$. By the associativity of operadic composition we have
        $$\xi'\circ(\sigma'_U)_{U\in\pi} = \left(\xi\circ (\mu_T)_{T\in\pi}\right)\circ(\sigma'_U)_{U\in\pi} = \xi\circ\left(\mu_T\circ (\sigma'_U)_{U\in\pi'_{|T}}\right)_{T\in\pi} =\xi\circ(\sigma_T)_{T\in\pi} = \nu, $$
        and therefore the horizontal arrow in \eqref{eq: axiom composition of psis} sends $(\beta,\nu)$ to the collection of the $\beta_{|U}$, for $U\in\pi'$, decorated by the $\sigma'_U$.
        \item (Composition of $\varphi$'s and $\psi$'s.) We consider \eqref{eq: axiom composition of phis and psis} applied to some left-$\mathcal{P}$-decorated partition $(\beta,\nu)$ with $(\pi,\xi)\leq (\beta,\nu)\leq (\pi',\xi')$. The commutativity easily follows from the equivariance of the operadic composition 
        $$\mathcal{P}(\pi)\times \prod_{T\in\pi}\mathcal{P}(\beta_{|T})\to \mathcal{P}(\beta)$$
        with respect to the natural bijection $\beta\simeq \varphi_{\pi'}(\beta)$.
        \end{enumerate}
        \end{proof}

        \subsubsection{Relation to the cobar construction and Koszul duality}

        We now explain how to understand the cohomology operad $h^\bullet({}^{\mathcal{P}}\Pi)$ in the setting of operadic Koszul duality, following Fresse \cite{fressepartitionposets} (in the case of the trivial operad) and adapting Vallette's study of the case of right-decorated partitions \cite{vallettepartitionposets}, which we will treat below. 

        Let $\mathcal{C}$ be a linear cooperad concentrated in degree $0$ and such that $\mathcal{C}(0)=0$ and $\mathcal{C}(1)=\KK 1$. We recall the definition of its \emph{cobar construction} $\Omega(\mathcal{C})$, which is a differential graded operad, see, e.g. \cite[\S 6.5]{lodayvallette} for more details. For every finite set $S$, the dg $\KK$-module $\Omega(\mathcal{C})(S)$ is spanned by reduced rooted trees with leaves labeled by $S$ and nodes labeled by elements of $\mathcal{C}$, where the grading is the number of nodes, and the differential is induced by the operadic cocomposition of $\mathcal{C}$. The operadic structure of $\Omega(\mathcal{C})$ is simply induced by the grafting operation on trees.
        
        We will also use a variant, the \emph{cobar construction with levels}, defined in the same way with \emph{leveled trees}. (It is also known in the literature as the \emph{simplicial cobar construction}.) In a leveled tree, each node is assigned an integer called its \emph{level}, in such a way that the levels are strictly increasing on each path from the root to a leaf, and such that the set of levels is an interval $\{1,\ldots,n\}$ for some $n$. For every finite set $S$, the dg $\KK$-module $\Omega_{\mathrm{level}}(\mathcal{C})(S)$ is spanned by reduced rooted leveled trees with leaves labeled by $S$ and nodes labeled by elements of $\mathcal{C}$, where the grading is the number of levels (it may be strictly smaller than the number of nodes), and the differential is induced by the operadic cocomposition of $\mathcal{C}$. The grafting operation on leveled trees does not make $\Omega_{\mathrm{level}}(\mathcal{C})$ into an operad because the associativity axiom is not satisfied. The \emph{delevelization morphism}
        \begin{equation}\label{eq: delevelization morphism}
        \Omega_{\mathrm{level}}(\mathcal{C}) \To \Omega(\mathcal{C}),
        \end{equation}
        sends a (decorated) leveled tree to the underlying (decorated) tree if the nodes have distinct levels, and to $0$ otherwise. It is a quasi-isomorphism in each arity by the dual statement of \cite[Theorem 4.1.8]{fressepartitionposets}. Furthermore, it commutes with the grafting operations.

        The following proposition is the ``left-decorated'' version of the main result of \cite{vallettepartitionposets}, with the slight difference that in our formalism, there is an \emph{a priori} operadic structure on $h^\bullet({}^{\mathcal{P}}\Pi)$. Here we borrow the terminology from \cite{vallettepartitionposets} and say that a graded poset $P$ is \emph{Cohen--Macaulay} (over $\KK$) if for every $x\leq y$ in $P$ we have $h_n([x,y])=0$ for $n$ different from the rank of $[x,y]$.

        \begin{prop}\label{prop: left Koszul}
        Let $\mathcal{P}$ be a left-basic set operad. 
        \begin{enumerate}[1)]
        \item We have an isomorphism of complexes, for every finite set $S$,
        $$\chains{\bullet}({}^{\mathcal{P}}\Pi(S)) \simeq \Omega_{\mathrm{level}}(\KK\mathcal{P}^\vee)(S),$$
        which induces an isomorphism of graded operads
        $$\h{\bullet}({}^{\mathcal{P}}\Pi) \simeq \operatorname{H}^\bullet(\Omega(\KK\mathcal{P}^\vee)).$$
        \item Assume that $\mathcal{P}$ is quadratic, and that the ring of coefficients $\KK$ is hereditary (e.g. $\KK=\mathbb{Z}$ or a field). Then $\KK\mathcal{P}$ is Koszul if and only if the maximal intervals $[\hat{0},(\hat{1},\xi)]$ of ${}^{\mathcal{P}}\Pi(S)$ are Cohen--Macaulay for all finite sets $S$ and all $\xi\in\mathcal{P}(S)$. In this case, if $(\KK\mathcal{P})^!$ denotes the Koszul dual operad, we have an isomorphism of graded operads
        $$\h{\bullet}({}^{\mathcal{P}}\Pi) \simeq \Lambda^{-1}(\KK\mathcal{P})^!.$$
        \end{enumerate}
        \end{prop} 

        \begin{proof}
        \begin{enumerate}[1)]
        \item For every finite set $S$ we define a linear map
        \begin{equation}\label{eq: from leveled trees to chains in proof}
        f_S:\; \Omega_{\mathrm{level}}(\KK\mathcal{P}^\vee)(S)\To c^\bullet({}^{\mathcal{P}}\Pi(S)) \; , \; t\mapsto [(\pi_0,\xi_0)<(\pi_1,\xi_1)<\cdots <(\pi_n,\xi_n)]
        \end{equation}
        A basis of $\Omega_{\mathrm{level}}(\KK\mathcal{P}^\vee)(S)$ in degree $n$ is given by leveled (reduced rooted) trees $t$ with levels in $\{1,\ldots,n\}$ whose leaves are labeled by $S$ and whose nodes are labeled by elements of $\mathcal{P}$.
        For each $i=0,\ldots,n$, cutting the tree $t$ between the levels $i$ and $i+1$ produces a forest $f_i$ above the cut and a tree $t_i$ below the cut. The connected components of $f_i$ define a partition of $S$ that we denote by $\pi_i$. We decorate it with the element $\xi_i\in \mathcal{P}(\pi_i)$ obtained by composing all the decorations of $t_i$. By convention, $(\pi_0,\xi_0)$ is the partition with only one block decorated by the unit element in $\mathcal{P}(1)$, and $(\pi_n,\xi_n)$ is the partition whose blocks are singletons, decorated by the composition of all the decorations in $t$. One checks that the linear map \eqref{eq: from leveled trees to chains in proof} thus defined is an isomorphism of complexes for every finite set $S$ (with an appropriate sign convention for the differential on the cobar construction with levels). Thanks to the quasi-isomorphism \eqref{eq: delevelization morphism}, we therefore get an isomorphism of graded linear species
        $$h^\bullet({}^{\mathcal{P}}\Pi) \simeq \operatorname{H}^\bullet(\Omega_{\mathrm{level}}(\KK\mathcal{P}^\vee)) \simeq \operatorname{H}^\bullet(\Omega(\KK\mathcal{P}^\vee)).$$
        We are now left with proving that this is an isomorphism of graded operads. Let $A$ and $B$ be disjoint finite sets, and let $\pi$ denote the partition of $A\sqcup B$ whose blocks are $B$ and the singletons $\{a\}$ for $a\in A$. We need to prove the commutativity of the following diagram, for disjoint finite sets $A$ and $B$, where the top horizontal arrow is the grafting operation on leveled trees, and the bottom horizontal arrow is \eqref{eq: operadic composition at cochain level}.
        $$\xymatrixrowsep{1.5cm}\diagram{
        \Omega_{\mathrm{level}}(\KK\mathcal{P}^\vee)(A\sqcup\{*\}) \otimes \Omega_{\mathrm{level}}(\KK\mathcal{P}^\vee)(B) \ar[r]^-{\circ_*} \ar[d]_-{f_{A\sqcup\{*\}}\otimes f_B} & \Omega_{\mathrm{level}}(\KK\mathcal{P}^\vee)(A\sqcup B) \ar[d]^-{f_{A\sqcup B}} \\
        c^\bullet({}^{\mathcal{P}}\Pi(A\sqcup\{*\})) \otimes c^\bullet({}^{\mathcal{P}}\Pi(B)) \ar[r]^-{\circ_*} & c^\bullet({}^{\mathcal{P}}\Pi(A\sqcup B))
        }$$
        Let us write $c=f_{A\sqcup \{*\}}(t)$ and $c'=f_B(t')$, and let $\xi\in \mathcal{P}(A\sqcup\{*\})$ denote the composition of all the decorations in $t$, which decorates the maximal element of $c$. We view $\xi$ as a left-$\mathcal{P}$-decoration of $\pi$. Then by definition, the only non-zero term in $c\circ_* c'$ is
        $$\mu_{(\pi,\xi)}(\varphi_{(\pi,\xi)}^*(c)\otimes \psi_{(\pi,\xi)}^*(c')).$$
        \begin{enumerate}[$\triangleright$]
        \item Applying $\varphi_{(\pi,\xi)}^*$ to $c$ replaces $*$ with $B$ in every partition appearing in $c$, while keeping the decorations.
        \item Applying $\psi_{(\pi,\xi)}^*$ to $c'$ replaces each partition $\pi'_i\in \Pi(B)$ appearing in $c'$ with the partition $\{\{a\},a\in A\}\cup \pi'_i$, and replaces the corresponding decoration $\xi'_i\in\mathcal{P}(\pi'_i)$ with $\xi\circ_*\xi'_i \in\mathcal{P}(A\sqcup \pi'_i)$.
        \end{enumerate}
        It is therefore clear that the chain obtained by concatenating $\varphi_{(\pi,\xi)}^*(c)$ and $\psi_{(\pi,\xi)}^*(c')$ is $f_{A\sqcup B}(t\circ_* t')$. This concludes the proof of the commutativity of the diagram above.
        \item This follows from the same proof as \cite[Theorem 9]{vallettepartitionposets}.
        \end{enumerate}
        \end{proof}

       \subsubsection{Examples of left-decorated partition posets and their cohomology operads}\label{sec: example left operads}

       \subsubsection*{The operad $\mathrm{Com}$.}
        Consider $\mathcal{P}=\mathrm{Com}$, the set operad encoding (associative and) commutative algebras, which satisfies $\mathrm{Com}(S)=\{*\}$ for every non-empty finite set $S$. It is left-basic. A left-$\mathrm{Com}$-decorated partition is just a partition, and ${}^{\mathrm{Com}}\Pi=\Pi$. In this case, Proposition \ref{prop: left Koszul} is the classical Proposition \ref{prop: partitions and Lie} because $(\KK\mathrm{Com})^!\simeq \operatorname{Lie}$.

        \subsubsection*{The operad $\mathrm{As}$.}

        Consider $\mathcal{P}=\mathrm{As}$, the set operad encoding associative algebras, for which $\mathrm{As}(S)$ is the set of linear orders on $S$ for every non-empty finite set $S$. It is left-basic. A left-$\mathrm{As}$-decorated partition is a partition equipped with a linear order on the set of its blocks. In the poset ${}^{\mathrm{As}}\Pi(S)$, one only merges blocks which are consecutive for the linear order on the set of blocks. Since $\KK\mathrm{As}$ is Koszul self-dual, one gets the following description of the cohomology of the operadic poset species of partitions equipped with a linear order on the set of blocks:
        $$h^\bullet({}^{\mathrm{As}}\Pi) \simeq \Lambda^{-1}\KK\mathrm{As}.$$

        \subsubsection*{The Umbrella Pine operad $\mathrm{UP}$.}
        We introduce a simple example of a set operad which is left-basic but not right-basic in the sense of Definition \ref{def:rightBasic} below. Let us consider the set operad $\mathrm{UP}$ generated by two binary generators $a\dashv b$ and $a\vdash b$ subject to the following relations:
        \begin{align*}
            \left( a \dashv b \right) \vdash c & = \left( a \dashv b \right) \dashv c &
            \left( a \vdash b \right) \vdash c & = \left( a \vdash b \right) \dashv c \\
            a \vdash \left( b \dashv c \right) &=a \dashv \left( b \dashv c \right) &
            a \vdash \left( b \vdash c \right) &=a \dashv \left( b \vdash c \right)
        \end{align*}
        Let us call \emph{highest} internal node of a tree an internal node whose children are leaves. Then $\mathrm{UP}(S)$ can be identified with the set of binary trees whose leaves are labeled by $S$ and whose highest internal nodes are decorated by $\dashv$ or $\vdash$. (These trees with decorations only at the highest internal nodes remind us of the umbrella pines which can be found in the Mediterranean region, hence the name.) Indeed, the relations are effectively forgetting the labeling of an internal node which is not a highest internal node. The sets $\mathrm{UP}(n)$ are enumerated by \cite[A025227]{oeis}, and by $2^k \times$\cite[A091894]{oeis} when refined by the number $k$ of highest internal nodes.

        The operad $\mathrm{UP}$ is left basic: a tree $t\in \mathrm{UP}(S)$ has at most one preimage by the operadic composition map \eqref{eq: axiom left basic}, namely the forest obtained by removing the tree $\xi$ at the root of $t$. However, it is not right-basic in the sense of Definition \ref{def:rightBasic} below: the map $\mathrm{UP}(2) \rightarrow \mathrm{UP}(3)$ given by $\mu\mapsto \mu(1\vdash 2, 3)$ is not injective as $\vdash$ and $\dashv$ have the same image. 
        
        Using the methods of Gröbner bases and rewriting of \cite{DotsenkoKhoroshkin}, specifically using \cite[Theorem 8.3.1]{lodayvallette}, one can easily prove that $\KK\mathrm{UP}$ is Koszul, as follows. We order the generators of $\mathrm{UP}$ by $\vdash \,\geq\, \dashv$, and induce an order on trees in the free operad generated by $\vdash$ and $\dashv$ by the lexicographical order  \cite[Proposition 3.5]{Hoffbeck}. All the relations can then be viewed as rewriting rules from left to right, and one checks that the associated critical pairs are confluent, and deduces that $\KK\mathrm{UP}$ is a Koszul operad. Its Koszul dual $(\KK\mathrm{UP})^!$ is presented by two binary generators $\dashv$ and $\vdash$ with relations: 
        \begin{align*}
            \left( a \dashv b \right) \vdash c & = - \left( a \dashv b \right) \dashv c &
            \left( a \vdash b \right) \vdash c & = - \left( a \vdash b \right) \dashv c \\
            a \vdash \left( b \dashv c \right) & = - \, a \dashv \left( b \dashv c \right) &
            a \vdash \left( b \vdash c \right) & = - \, a \dashv \left( b \vdash c \right)
        \end{align*}
        Therefore, the posets of left-$\mathrm{UP}$-decorated partitions are Cohen--Macaulay and we have an isomorphism of graded operads: 
        $$h^\bullet({}^{\mathrm{UP}}\Pi)\simeq \Lambda^{-1}(\KK\mathrm{UP})^!.$$
        
        \subsubsection*{The operad $\mathrm{NAC}_2$}

        We introduce the set operad $\mathrm{NAC}_2$ encoding ``non-associative $2$-step commutative'' algebras. It has one binary generator $ab$ with only relation
        $$(ab)c=c(ab).$$
        In other words, is is obtained from the operad $\mathrm{UP}$ by quotienting it by the relation $a \vdash b = b \dashv a$. For a finite set $S$, the set $\mathrm{NAC}_2(S)$ consists of planar binary trees with leaves labeled by $S$ considered modulo swapping of the planar order at all the internal nodes except for the highest ones. For the same reason as for $\mathrm{UP}$, the operad $\mathrm{NAC}_2$ is left-basic but not right-basic in the sense of Definition \ref{def:rightBasic}. 
        
        The numbers $u_n=|\mathrm{NAC}_2(n)|$ satisfy $u_1=1$, $u_2=2$, and the recurrence relation
        $$u_n=\frac{1}{2}\sum_{k=1}^{n-1}\binom{n}{k}u_ku_{n-k} \quad (n\geq 3).$$
        From there, one readily sees that the exponential generating series of $\mathrm{NAC}_2$ equals
        $$\sum_{n\geq 1} u_n\frac{x^n}{n!} = 1-\sqrt{1-2x-x^2}.$$
        This corresponds to \cite[A182037]{oeis}.

        As in the case of the operad $\mathrm{UP}$, one can prove that $\KK\mathrm{NAC}_2$ is Koszul by using the methods of Gröbner bases and rewriting. By \cite[Proposition 7.6.8]{lodayvallette} one sees that $\KK\mathrm{NAC}_2$ is Koszul self-dual: $(\KK\mathrm{NAC}_2)^! = \KK\mathrm{NAC}_2$. One therefore gets the following description of the cohomology of the operadic poset species of left-$\mathrm{NAC}_2$-decorated partitions:
        $$h^\bullet({}^{\mathrm{NAC}_2}\Pi) \simeq \Lambda^{-1}\KK\mathrm{NAC}_2.$$

    \subsection{Right-decorated partitions} \label{sec: Rdp}

        We now switch from left-decorated to right-decorated partitions, which is the setting of \cite{vallettepartitionposets} and \cite{MendezYang}. Recall that all set operads are implicitly assumed to be finite in each arity.

        \subsubsection{Definitions and examples}

        \begin{defi}
        Let $\mathcal{Q}$ be a set operad satisfying $\mathcal{Q}(0)=\varnothing$ and $\mathcal{Q}(1)=\{*\}$. A right-$\mathcal{Q}$-decorated partition of a finite set $S$ is a pair $(\pi,\xi)$ where $\pi$ is a partition of $S$ and $\xi$ is a collection of elements $\xi_T\in \mathcal{Q}(T)$ for $T\in\pi$. We denote by $\Pi^{\mathcal{Q}}(S)$ the set of right-$\mathcal{Q}$-decorated partitions of $S$. 
        \end{defi} 

        \begin{rem}
        The terminology comes from the fact that $\Pi^{\mathcal{Q}}$ equals the composition of species $\mathbb{E}^+\circ\mathcal{Q}$, with $\mathcal{Q}$ on the right, where $\mathbb{E}^+$ denotes the species of non-empty sets.
        \end{rem}
        
        We define an order relation on right-$\mathcal{Q}$-decorated partitions by saying that
        $$(\alpha,\eta) \leq (\beta,\xi)$$
        if $\alpha$ is obtained by merging blocks of $\beta$ and the $\eta_A$'s are obtained accordingly as operadic compositions of the $\xi_B$'s. More formally, this means that $\alpha\leq \beta$ as partitions and for every $A\in \alpha$ there exists $\nu_A\in \mathcal{Q}(\beta_{|A})$ such that $\eta_A=\nu_A\circ(\xi_B)_{B\in \beta_{|A}}$.
        One checks that this gives each $\Pi^{\mathcal{Q}}(S)$ the structure of a poset, and $S\mapsto \Pi^{\mathcal{Q}}(S)$ the structure of a poset species. 
        There is a morphism of poset species
        $$\diagram{\;\;\Pi^{\mathcal{Q}} \ar[d]_a \\ \Pi}$$
        which forgets the decorations: $a(\pi,\xi)=\pi$.

        \subsubsection{Right-decorated partitions as an operadic poset species}

        In order to make right-decorated partitions into an operadic poset species, the following assumption on the decorating operad $\mathcal{Q}$ is natural.

        \begin{defi} \label{def:rightBasic}
        A set operad $\mathcal{Q}$ satisfying $\mathcal{Q}(0)=\varnothing$ and $\mathcal{Q}(1)=\{*\}$ is said to be \emph{right-basic} if for every finite set $S$ and every right-$\mathcal{Q}$-decorated partition $(\pi,\xi)$ of $S$, the composition map
        $$\mathcal{Q}(\pi)\To \mathcal{Q}(S) \; , \; \nu \mapsto \nu\circ (\xi_T)_{T\in\pi}$$
        is injective.
        \end{defi}

        \begin{rem}
        Right-basic operads are called \emph{set-basic} in \cite{vallettepartitionposets}. They correspond to the notion of $c$-monoid in \cite{MendezYang}.
        \end{rem}

        Let now $\mathcal{Q}$ be a right-basic set operad. Let $(\pi,\xi)$ be a right-$\mathcal{Q}$-decorated partition of a finite set $S$. We first define a morphism of posets
        $$\varphi_{(\pi,\xi)}:\Pi^{\mathcal{Q}}_{\leq (\pi,\xi)}(S)\To \Pi^{\mathcal{Q}}(\pi).$$
        An element of $\Pi^{\mathcal{Q}}_{\leq (\pi,\xi)}(S)$ is a right-$\mathcal{Q}$-decorated partition $(\alpha, \eta)$ where $\alpha\leq \pi$ and for every $A\in \alpha$ there exists $\nu_A\in\mathcal{Q}(\pi_{|A})$ such that $\eta_A=\nu_A\circ(\xi_T)_{T\in\pi_{|A}}$. Since $\mathcal{Q}$ is right-basic, each $\nu_A$ is unique and we can define the image of $(\alpha,\eta)$ by $\varphi_{(\pi,\xi)}$ to be the partition $\varphi_\pi(\alpha)=\{\pi_{|A},A\in\alpha\}$ decorated by the $\nu_A$. One can note that $\varphi_{(\pi,\xi)}$ is an isomorphism of posets. Defining the morphism of posets
        $$\psi_{(\pi,\xi)}:\Pi^{\mathcal{Q}}_{\geq (\pi,\xi)}(S)\To \prod_{T\in \pi} \Pi^{\mathcal{Q}}(T)$$
        is easier and does not use the fact that $\mathcal{Q}$ is right-basic: we define $\psi_{(\pi,\xi)}(\beta,\nu)$ to be the family of partitions $\beta_{|T}$, for $T\in \pi$, with the natural decorations induced by $\nu$. One can note that $\psi_{(\pi,\xi)}$ induces an isomorphism of posets between $\Pi^{\mathcal{Q}}_{\geq (\pi,\xi)}(S)$ and the product of the intervals $\Pi^{\mathcal{Q}}_{\geq (\hat{0},\xi_T)}(T)$, for $T\in\pi$.

        \begin{prop}
        For $\mathcal{Q}$ a right-basic set operad, this structure makes $\Pi^{\mathcal{Q}}$ into an operadic poset species.
        \end{prop}

        \begin{proof}
        The structure morphisms $a$ satisfy \eqref{eq: axiom a min max compatible strong}. The commutativity of the diagrams \eqref{eq: commutative squares phi psi a} is obvious, as are the equivariance and unitality axioms. We are left with proving the associativity axioms. In each case we only need to trace the decorations since the diagrams obviously commute at the level of the underlying partitions. Let $(\pi,\xi)\leq (\pi',\xi')$ be right-$\mathcal{Q}$-decorated partitions of a finite set $S$.
        \begin{enumerate}[$\triangleright$]
        \item (Composition of $\varphi$'s.) Let us write $\xi_T=\nu_T\circ (\xi'_U)_{U\in\pi'_{|T}}$ for all $T\in\pi$, where $\nu_T\in\mathcal{Q}(\pi'_{|T})$. For $(\alpha,\eta)\leq (\pi,\xi)$ we can write $\eta_A=\mu_A\circ(\xi_T)_{T\in\pi_{|A}}$ for all $A\in\alpha$, where $\mu_A\in\mathcal{Q}(\pi_{|A})$. The horizontal arrow $\varphi_{(\pi,\xi)}$ in \eqref{eq: axiom composition of phis} sends $(\alpha,\eta)$ to $\varphi_\pi(\alpha)=\{\pi_{|A},A\in\alpha\}$ with decorations $\mu_A$. By the associativity of operadic composition we have, for all $A\in\alpha$, the equality
        $$\eta_A=\mu_A\circ\left(\nu_T\circ (\xi'_U)_{U\in\pi'_{|T}}\right)_{T\in\pi_{|A}} = \left(\mu_A\circ (\nu_T)_{T\in \pi_{|A}}\right)\circ (\xi'_{U})_{U\in \pi'_{|A}}.$$
        Therefore, in the bottom composition of \eqref{eq: axiom composition of phis}, the first arrow sends $(\alpha,\eta)$ to $\varphi_{\pi'}(\alpha)=\{\pi'_{|A},A\in\alpha\}$ with decorations 
        $$\mu_A\circ (\nu_T)_{T\in \pi_{|A}} \; \in \mathcal{Q}(\pi'_{|A}).$$
        Note that the object denoted by $\underline{x}$ in \eqref{eq: axiom composition of phis} is $\varphi_{\pi'}(\pi)=\{\pi'_{|T},T\in\pi\}$ with decorations the $\nu_T\in\mathcal{Q}(\pi'_{|T})$, and thus by definition the bottom composition of \eqref{eq: axiom composition of phis} sends $(\alpha,\eta)$ to $\varphi_\pi(\alpha)$ with decorations $\mu_A$.
        \item (Composition of $\psi$'s.) By direct inspection.
        \item (Composition of $\varphi$'s and $\psi$'s.) By direct inspection.
        \end{enumerate}
        \end{proof}

    \subsubsection{Relation to the cobar construction and Koszul duality}

        The following result is the right-decorated version of Proposition \ref{prop: left Koszul}, and was originally proved by Vallette \cite{vallettepartitionposets}, with the slight difference that in our formalism, there is an \emph{a priori} operadic structure on $h^\bullet(\Pi^{\mathcal{Q}})$.

        \begin{prop}\label{prop: right koszul}
        Let $\mathcal{Q}$ be a right-basic set operad.
        \begin{enumerate}[1)]
        \item There is an isomorphism of complexes, for every finite set $S$,
        $$\chains{\bullet}(\Pi^{\mathcal{Q}}(S)) \simeq \Omega_{\mathrm{level}}(\KK\mathcal{Q}^\vee)(S),$$
        which induces an isomorphism of graded operads
        $$\h{\bullet}(\Pi^{\mathcal{Q}}) \simeq \operatorname{H}^\bullet(\Omega(\KK\mathcal{Q}^\vee)).$$
        \item Assume that $\mathcal{Q}$ is quadratic, and that the ring of coefficients $\mathbb{K}$ is hereditary (e.g. $\KK=\mathbb{Z}$ or a field). Then $\KK\mathcal{Q}$ is Koszul if and only if the maximal intervals $[(\hat{0},\xi),\hat{1}]$ of $\Pi^{\mathcal{Q}}(S)$ are Cohen--Macaulay for all finite sets $S$ and all $\xi\in\mathcal{Q}(S)$. In this case, if $(\KK\mathcal{Q})^!$ denotes the Koszul dual operad, we have an isomorphism of graded operads
        $$\h{\bullet}(\Pi^{\mathcal{Q}}) \simeq \Lambda^{-1}(\KK\mathcal{Q})^!.$$
        \end{enumerate}
        \end{prop}

        \begin{proof}
        \begin{enumerate}[1)]
        \item For every finite set $S$ we define a linear map
        \begin{equation}\label{eq: from leveled trees to chains in proof right}
        g_S:\; \Omega_{\mathrm{level}}(\KK\mathcal{Q}^\vee)(S)\To c^\bullet(\Pi^{\mathcal{Q}}(S)) \; , \; t\mapsto [(\pi_0,\xi_0)<(\pi_1,\xi_1)<\cdots <(\pi_n,\xi_n)]
        \end{equation}
        A basis of $\Omega_{\mathrm{level}}(\KK\mathcal{Q}^\vee)(S)$ in degree $n$ is given by leveled (reduced rooted) trees $t$ with levels in $\{1,\ldots,n\}$ whose leaves are labeled by $S$ and whose nodes are labeled by elements of $\mathcal{Q}$.
        For each $i=0,\ldots,n$, cutting the tree $t$ between the levels $i$ and $i+1$ produces a forest $f_i$ above the cut and a tree $t_i$ below the cut. The connected components of $f_i$ define a partition of $S$ that we denote by $\pi_i$. We decorate a block $T$ of $\pi_i$ with the element $\xi_T\in \mathcal{Q}(T)$ obtained by composing all the decorations of the corresponding connected component of $f_i$. By convention, $(\pi_0,\xi_0)$ is the partition with only one block decorated by the composition of all the decorations in $t$, and $(\pi_n,\xi_n)$ is the partition whose blocks are singletons, all decorated by the unit element in $\mathcal{Q}(1)$. One checks that the linear map \eqref{eq: from leveled trees to chains in proof right} thus defined is an isomorphism of complexes for every finite set $S$ (with an appropriate sign convention for the differential on the cobar construction with levels). Thanks to the quasi-isomorphism \eqref{eq: delevelization morphism}, we therefore get an isomorphism of graded linear species
        $$h^\bullet(\Pi^{\mathcal{Q}}) \simeq \operatorname{H}^\bullet(\Omega_{\mathrm{level}}(\KK\mathcal{Q}^\vee)) \simeq \operatorname{H}^\bullet(\Omega(\KK\mathcal{Q}^\vee)).$$
        We are now left with proving that this is an isomorphism of graded operads. 
        Let $A$ and $B$ be disjoint finite sets, and let $\pi$ denote the partition of $A\sqcup B$ whose blocks are $B$ and the singletons $\{a\}$ for $a\in A$. We need to prove the commutativity of the following diagram, for disjoint finite sets $A$ and $B$, where the top horizontal arrow is the grafting operation on leveled trees, and the bottom horizontal arrow is \eqref{eq: operadic composition at cochain level}.
        $$\xymatrixrowsep{1.5cm}\diagram{
        \Omega_{\mathrm{level}}(\KK\mathcal{Q}^\vee)(A\sqcup\{*\}) \otimes \Omega_{\mathrm{level}}(\KK\mathcal{Q}^\vee)(B) \ar[r]^-{\circ_*} \ar[d]_-{g_{A\sqcup\{*\}}\otimes g_B} & \Omega_{\mathrm{level}}(\KK\mathcal{Q}^\vee)(A\sqcup B) \ar[d]^-{g_{A\sqcup B}} \\
        c^\bullet(\Pi^{\mathcal{Q}}(A\sqcup\{*\})) \otimes c^\bullet(\Pi^{\mathcal{Q}}(B)) \ar[r]^-{\circ_*} & c^\bullet(\Pi^{\mathcal{Q}}(A\sqcup B))
        }$$
        Let us write $c=g_{A\sqcup \{*\}}(t)$ and $c'=g_B(t')$, and let $\xi\in \mathcal{Q}(B)$ denote the composition of all the decorations in $t'$, which decorates the minimal element of $c'$. We view $\xi$ as a right-$\mathcal{Q}$-decoration of $\pi$. Then by definition, the only non-zero term in $c\circ_* c'$ is
        $$\mu_{(\pi,\xi)}(\varphi_{(\pi,\xi)}^*(c)\otimes \psi_{(\pi,\xi)}^*(c')).$$
        \begin{enumerate}[$\triangleright$]
        \item Applying $\varphi_{(\pi,\xi)}^*$ to $c$ replaces $*$ with $B$ in every partition appearing in $c$, and replaces the corresponding decoration $\xi_i\in \mathcal{Q}(T\sqcup\{*\})$ with $\xi_i\cup_*\xi\in \mathcal{Q}(T\sqcup B)$. 
        \item Applying $\psi_{(\pi,\xi)}^*$ to $c'$ replaces each partition $\pi'_i\in \Pi(B)$ appearing in $c'$ with the partition $\{\{a\},a\in A\}\cup \pi'_i$, while keeping the decorations.
        \end{enumerate}
        It is therefore clear that the chain obtained by concatenating $\varphi_{(\pi,\xi)}^*(c)$ and $\psi_{(\pi,\xi)}^*(c')$ is $g_{A\sqcup B}(t\circ_* t')$. This concludes the proof of the commutativity of the diagram above.
        \item This is \cite[Theorem 9]{vallettepartitionposets}.
        \end{enumerate}
        \end{proof}

        \subsubsection{Examples of right-decorated partition posets and their cohomology operads}\label{sec: examples right operads}

        The following examples are taken from \cite{vallettepartitionposets}, which also contains more.

        \subsubsection*{The operad $\mathrm{Com}$}

        Consider $\mathcal{Q}=\mathrm{Com}$, the set operad encoding (associative and) commutative algebras, which satisfies $\mathrm{Com}(S)=\{*\}$ for every non-empty finite set $S$. It is right-basic. A right-$\mathrm{Com}$-decorated partition is just a partition, and $\Pi^{\mathrm{Com}}=\Pi$. In this case, Proposition \ref{prop: right koszul} is the classical Proposition \ref{prop: partitions and Lie} because $(\KK\mathrm{Com})^!\simeq \mathrm{Lie}$. Another special case of interest is $\mathcal{Q}=\mathrm{Com}_k$, the operad generated by a totally commutative and associative operation in arity $k+1$, for which one recovers the result of \cite{hanlonwachs}.

        \subsubsection*{The operad $\mathrm{As}$}

        Consider $\mathcal{Q}=\mathrm{As}$, the set operad encoding associative algebras, for which $\mathrm{As}(S)$ is the set of linear orders on $S$ for every non-empty finite set $S$. It is right-basic. A right-$\mathrm{As}$-decorated partition is a partition for which each block is equipped with a total order. In the poset $\Pi^{\mathrm{As}}$, one can only split a block into intervals for the linear order. Since $\KK\mathrm{As}$ is Koszul self-dual, one gets the following description of the cohomology of the operadic poset species of partitions into blocks with linear orders:
        $$h^\bullet(\Pi^{\mathrm{As}}) \simeq \Lambda^{-1}\mathrm{As}.$$
        
        \begin{rem}
        Right-$\mathrm{As}$-decorated partitions are called ``ordered partitions'' in \cite{vallettepartitionposets}, but we will not use that terminology since it could equally well apply to \emph{left}-$\mathrm{As}$-decorated partitions (see \S\ref{sec: example left operads} above).
        \end{rem}

        \subsubsection*{The operad $\mathrm{Perm}$}
        Consider the operad $\mathcal{Q}=\mathrm{Perm}$ encoding permutative algebras, introduced by Chapoton \cite{chapotonendofoncteur}. It is generated by a binary operation $*$ which is associative and satisfies the identity $x*y*z=x*z*y$. We have the description $\mathrm{Perm}(S)=S$, where an element $s\in S$ corresponds to the $\mathrm{Perm}$ operation where the variable decorated by $s$ is on the left. The operad $\mathrm{Perm}$ is right-basic but not left-basic. A right-$\mathrm{Perm}$-decorated partition is a partition in which each block has a distinguished element. In the poset $\Pi^{\mathrm{Perm}}$, one merges blocks by keeping the distinguished element of one of the merged blocks. Right-$\mathrm{Perm}$-decorated partitions are called ``pointed partitions'' in \cite{vallettepartitionposets}.

        By the results of \cite{vallettepartitionposets, chapotonvallette}, the operad $\KK\mathrm{Perm}$ is Koszul with Koszul dual the operad $\mathrm{PreLie}$ encoding pre-Lie algebras \cite{ChapotonLivernet}. We therefore recover the following description of the cohomology of the operadic poset species of pointed partitions:
        $$h^\bullet(\Pi^{\mathrm{Perm}}) \simeq \Lambda^{-1}\mathrm{PreLie}.$$

        \subsection{Bi-decorated partitions} \label{sec: Bdp}

        \begin{defi}
        Let $\mathcal{P}$ and $\mathcal{Q}$ be two set operads. A $(\mathcal{P},\mathcal{Q})$-decorated partition of a set $S$ is a triple $(\pi,{}^{\mathcal{P}}\xi, \xi^{\mathcal{Q}})$ where $\pi$ is a partition of $S$, ${}^{\mathcal{P}}\xi$ is an element of $\mathcal{P}(\pi)$, and $\xi^{\mathcal{Q}}$ is a collection of elements $\xi^{\mathcal{Q}}_T\in\mathcal{Q}(T)$ for $T\in\pi$. We denote by ${}^{\mathcal{P}}\Pi^{\mathcal{Q}}(S)$ the set of $(\mathcal{P},\mathcal{Q})$-decorated partitions of $S$.
        \end{defi}

        \begin{rem} 
        We have the equality of species: ${}^{\mathcal{P}}\Pi^{\mathcal{Q}}=\mathcal{P} \circ \mathcal{Q}$. 
        \end{rem}
        
        The following diagram of forgetful morphisms is cartesian.
        $$\diagram{
        {}^{\mathcal{P}}\Pi^{\mathcal{Q}} \ar[r]\ar[d] & {}^{\mathcal{P}}\Pi \ar[d] \\
        \Pi^{\mathcal{Q}} \ar[r] & \Pi
        }$$
        In other words, we have the equality
        $${}^{\mathcal{P}}\Pi^{\mathcal{Q}} = {}^{\mathcal{P}}\Pi\underset{\Pi}{\times} \Pi^{\mathcal{Q}}.$$
        and Proposition \ref{prop: fiber product} implies the following result.

        \begin{prop}
        Let $\mathcal{P}$ be a left-basic set operad, and $\mathcal{Q}$ be a right-basic set operad. Then there is a unique structure of an operadic poset species on ${}^{\mathcal{P}}\Pi^{\mathcal{Q}}$ for which the forgetful morphisms to ${}^{\mathcal{P}}\Pi$ and $\Pi^{\mathcal{Q}}$ are morphisms of operadic poset species.
        \end{prop}

        We leave it to the reader to explore interesting examples of operads $h^\bullet({}^{\mathcal{P}}\Pi^{\mathcal{Q}})$, and their left and right operadic modules $\hbottom{\bullet}({}^{\mathcal{P}}\Pi^{\mathcal{Q}})$ and $\htop{\bullet}({}^{\mathcal{P}}\Pi^{\mathcal{Q}})$.

\section{Decorated partitions revisited: operadic modules}\label{sec: decorated partitions modules}

    We now explore the operadic modules appearing in the variants $\hbottom{\bullet}$ (left operadic module) and $\htop{\bullet}$ (right operadic module) of the cohomology of posets of decorated partitions.

    \subsection{Left-decorated partitions and left operadic modules}\label{subsec: Ldp mod}
    
        \subsubsection{Interpretation as cobar construction with coefficients}
        
        The poset ${}^{\mathcal{P}}\Pi(S)$ of left-$\mathcal{P}$-decorated partitions of $S$ has a least element but potentially many maximal elements, which are in bijection with $\mathcal{P}(S)$. This means that the right operadic module $\htop{\bullet}({}^{\mathcal{P}}\Pi)$ is trivial, but the left operadic module $\hbottom{\bullet}({}^\mathcal{P}\Pi)$ may be interesting. It can be expressed in terms of the cobar construction with coefficients, as follows. 
        
        Recall \cite{fressepartitionposets} that for a linear cooperad $\mathcal{C}$ as in the previous paragraph, and a left $\mathcal{C}$-comodule $L$, we have the cobar construction with coefficients and its variant with levels:
        $$\Omega(\mathcal{C},L) = \Omega(\mathcal{C})\circ L \qquad \mbox{ and } \qquad \Omega_{\mathrm{level}}(\mathcal{C},L) = \Omega_{\mathrm{level}}(\mathcal{C})\circ L.$$ Note that $\Omega(\mathcal{C},L)$ is a left operadic $\Omega(\mathcal{C})$-module. We have the delevelization quasi-isomorphism
        $$\Omega_{\mathrm{level}}(\mathcal{C},L) \stackrel{\sim}{\To} \Omega(\mathcal{C},L).$$

        Like any set operad, $\mathcal{P}$ has a canonical left operadic module $L_{\mathcal{P}}$, defined by
        $$L_{\mathcal{P}}(S)=\{*\}$$
        for every non-empty finite set $S$, and $L_{\mathcal{P}}(\varnothing)=\varnothing$. Therefore, the linear cooperad $\KK\mathcal{P}^\vee$ has a left operadic comodule $\KK L_{\mathcal{P}}^\vee$. The cobar construction $\Omega(\KK\mathcal{P}^\vee,\KK L_{\mathcal{P}}^\vee)$ has a basis consisting of reduced rooted trees with nodes labeled by elements of $\mathcal{P}$ and leaves labeled by non-empty subsets of $S$ which form a partition of $S$.

        \begin{prop}\label{prop: left operadic module decorated partitions}
        Let $\mathcal{P}$ be a left-basic set operad. We have an isomorphism of complexes
        $$\chainsbottom{\bullet}({}^{\mathcal{P}}\Pi(S)) \simeq \Omega_{\mathrm{level}}(\KK\mathcal{P}^\vee, \KK L_{\mathcal{P}}^\vee)(S)$$
        which induces an isomorphism 
        $$\hbottom{\bullet}({}^{\mathcal{P}}\Pi) \simeq \operatorname{H}^\bullet(\Omega(\KK\mathcal{P}^\vee,\KK L_{\mathcal{P}}^\vee)),$$
        compatible with the left operadic module structures over $h^\bullet({}^{\mathcal{P}}\Pi) \simeq \operatorname{H}^\bullet(\Omega(\KK\mathcal{P}^\vee))$.
        \end{prop} 

        \begin{proof}
        It follows from a straighforward adaptation of the proof of Proposition \ref{prop: left Koszul}.
        \end{proof}

    \subsubsection{($\mathfrak{S}$-equivariant) Euler characteristics} 
    
        We do not know of any operadic tools which would allow to compute the cohomology of the cobar construction appearing in Proposition \ref{prop: left operadic module decorated partitions}, in the spirit of Koszul duality (Proposition \ref{prop: left Koszul}). However, when $\KK\mathcal{P}$ is Koszul, we can give formulas for the corresponding ($\mathfrak{S}$-equivariant) Euler characteristics, as follows. 
        
        We work with the ring of coefficients $\KK=\QQ$, and the ring $R_\QQ \defas \QQ[[p_1,p_2,p_3,\ldots]]$ of polynomials in countably many formal variables $p_n$, for $n\geq 1$, viewed as the (completed) ring of symmetric functions. For a linear species $\mathrm{F}$  such that $\mathrm{F}(n)$ is finite dimensional for all $n$, we define its \emph{cycle index series}
        $$\mathrm{Z}_{\mathrm{F}} = \sum_{n\geq 0} \frac{1}{n!} \sum_{\sigma\in\mathfrak{S}_n} \mathrm{Tr}(\sigma|\mathrm{F}(n)) \, p_{\lambda(\sigma)} \;\;\in R_\QQ,$$
        where $p_\lambda(\sigma)\defas p_{\lambda_1}\cdots p_{\lambda_r}$ if $\sigma$ is the product of disjoint cycles of lengths $\lambda_1,\ldots, \lambda_r$. For two elements $\mathrm{Z}=\mathrm{Z}(p_1,p_2,p_3,\ldots)$ and $\mathrm{Z}'=\mathrm{Z}'(p_1,p_2,p_3,\ldots)$ of $R_\QQ$ such that $\mathrm{Z}'(0,0,0,\ldots)=0$, their \emph{plethysm} is:
        $$\mathrm{Z}\circ \mathrm{Z}' \defas \mathrm{Z}(\mathrm{Z}'(p_1,p_2,p_3,\ldots), \mathrm{Z}'(p_2,p_4,p_6, \ldots), \mathrm{Z}'(p_3,p_6,p_9,\ldots),\ldots).$$
        For instance, $p_m\circ p_n=p_{mn}$ for all $m,n\geq 0$, and when $\mathrm{G}(0)=0$, we have $\mathrm{Z}_{\mathrm{F}\circ \mathrm{G}}=\mathrm{Z}_{\mathrm{F}}\circ \mathrm{Z}_{\mathrm{G}}$. For an element $\mathrm{Z}\in R_\QQ$, its \emph{suspension} is 
        $$\Sigma \mathrm{Z} = -\mathrm{Z}(-p_1,-p_2,\ldots).$$
        
        \begin{prop}\label{prop: cycle index series left operadic module}
        Let $\mathcal{P}$ be a left-basic set operad such that $\QQ\mathcal{P}$ is Koszul, with Koszul dual $(\QQ\mathcal{P})^!$. Then we have the equality in the ring $R_\QQ$:
            $$\sum_{k\geq 0}(-1)^k \mathrm{Z}_{\hbottom{k}({}^{\mathcal{P}}\Pi)}  = (\Sigma \mathrm{Z}_{(\QQ\mathcal{P})^!} ) \circ \left( \exp\Bigg(\sum_{n\geq 1}\frac{p_n}{n}\Bigg) -1 \right). $$
        \end{prop}

        Note that the sum on the left-hand side makes sense since we have $\hbottom{k}({}^{\mathcal{P}}\Pi(n))=0$ for $k\geq n$.

        \begin{proof}
        By Proposition \ref{prop: left operadic module decorated partitions}, $\hbottom{\bullet}({}^{\mathcal{P}}\Pi)$ is isomorphic to the cohomology of the dg species $\Omega(\QQ\mathcal{P}^\vee)\circ \QQ L_{\mathcal{P}}^\vee$. This implies the equality in the ring $R_\QQ$:
        $$\sum_{k\geq 0}(-1)^k \mathrm{Z}_{\hbottom{k}({}^{\mathcal{P}}\Pi)}  = \left( \sum_{k\geq 0} (-1)^k \mathrm{Z}_{\mathrm{H}^k(\Omega(\QQ\mathcal{P}^\vee))} \right) \circ \mathrm{Z}_{\QQ L_{\mathcal{P}}^\vee}. $$
        Since $\QQ\mathcal{P}$ is Koszul, the only non-zero cohomology group of $\Omega(\QQ\mathcal{P}^\vee)(n)$ is the $\mathrm{H}^{n-1}$, which is isomorphic to $(\QQ\mathcal{P})^!(n)\otimes \mathrm{sgn}_n$. By using the standard formula for the signature of a cycle, we readily deduce:
        $$\sum_{k\geq 0} (-1)^k \mathrm{Z}_{\mathrm{H}^k(\Omega(\QQ\mathcal{P}^\vee))} = \Sigma \mathrm{Z}_{(\QQ\mathcal{P})^!}.$$
        Finally, $\QQ L_{\mathcal{P}}^\vee(n)$ is the $1$-dimensional trivial representation of $\mathfrak{S}_n$ in every arity $n\geq 1$, and zero in arity $n=0$.
        A classical computation shows that its cycle index series is 
        $$\mathrm{Z}_{\QQ L_{\mathcal{P}}^\vee}=\exp\Bigg(\sum_{n\geq 1}\frac{p_n}{n}\Bigg)-1.$$ 
        The equality follows.
        \end{proof}

        As a special case of Proposition \ref{prop: cycle index series left operadic module} we now derive a formula for the alternating sums of the dimensions of the cohomology groups $\hbottom{k}({}^{\mathcal{P}}\Pi(n))$, also known as M\"{o}bius numbers. For a poset $P$, we let
        $$\widecheck{\mu}(P) = \sum_{k\geq 0} (-1)^k \mathrm{dim}(\hbottom{k}(P)).$$ 
        For a linear species $F$, we introduce its \emph{exponential generating series}
        $$\mathrm{C}_{\mathrm{F}}(x) = \sum_{n\geq 0} \mathrm{dim}(\mathrm{F}(n))\frac{x^n}{n!}.$$
        When $\mathrm{G}(0)=0$, we note that $\mathrm{C}_{\mathrm{F}\circ \mathrm{G}}$ is the usual composition $\mathrm{C}_{\mathrm{F}}\circ \mathrm{C}_{\mathrm{G}}$.

        \begin{prop}\label{prop: generating series left operadic module}
        Let $\mathcal{P}$ be a left-basic set operad such that $\QQ\mathcal{P}$ is Koszul, with Koszul dual $(\QQ\mathcal{P})^!$. Then we have the equality of formal power series:
        \begin{equation}\label{eq: generating series mobius numbers left operadic module}
        \sum_{n\geq 1} \widecheck{\mu}({}^{\mathcal{P}}\Pi(n)) \frac{x^n}{n!} = -\mathrm{C}_{(\QQ\mathcal{P})^!}(1-\exp(x)).    
        \end{equation}
        \end{prop}

      \begin{proof}
        This follows from Proposition \ref{prop: cycle index series left operadic module} by noting that $\mathrm{C}_{\mathrm{F}}(x) = \mathrm{Z}_{\mathrm{F}}|_{p_1=x, p_2=p_3=\cdots =0}$.
        \end{proof}
        
        The sequences $\widecheck{\mu}({}^{\mathcal{P}}\Pi(n))$ for some examples are presented in Table \ref{tab:checkmu}. All the operads $\mathcal{P}$ in this table are left-basic with $\QQ\mathcal{P}$ Koszul, so that Proposition \ref{prop: generating series left operadic module} applies. We see that in the cases $\mathcal{P}=\mathrm{NAP},\mathrm{2as},\mathrm{Dipt}$, the operad $\Lambda\hbottom{\bullet}({}^{\mathcal{P}}\Pi)$ is not concentrated in degree zero. We develop some of these examples in \S\ref{sec: open questions}.

  \begin{table}[h!]
           \centering
           \begin{tabular}{|P{25mm}|c|P{90mm}|}
           \hline
$\mathcal{P} $ &      $-\mathrm{C}_{(\QQ\mathcal{P})^!}(-x)$ 
& $\widecheck{\mu}({}^{\mathcal{P}} \Pi(n)),  n\geq 1$  
\\ \hline
            \begin{minipage}{20mm} \vspace{2pt} As \cite{PeirceAs} \vspace{2pt}\end{minipage}& $\frac{x}{1+x}$  & $1,-1,1,-1,1,-1,1,-1,1,-1$ \\ \hline
           \begin{minipage}{25mm}  \vspace{2pt} $\mathrm{Com}_2$ \cite{DotsenkoKhoroshkinCom2}  \vspace{2pt}\end{minipage}& $\sum_{n \geq 1} (-n)^{n-1} \frac{x^n}{n!}$  & $1, -1, 4, -23, 181, -1812, 22037$ \cite[A383985]{oeis}
 \\ 
             \hline
        \begin{minipage}{25mm} \vspace{2pt} $\mathrm{NAC}_2$ (\ref{sec: example left operads}) \vspace{2pt} \end{minipage} & $-1+\sqrt{1+2x-x^2}$ & $1, -1, 1, -13, 61, -601, 5881, -73333$ \cite[A383986]{oeis} \\ \hline
         \rowcolor{lightgray} \begin{minipage}{20mm}  \vspace{2pt}  NAP \cite{LivernetNAP} \vspace{2pt} \end{minipage}  & $x\exp(-x)$  & $1,-1,-2,1,11,18,-41,-317,-680$  \cite[A101851(-x)]{oeis}\\ \hline
      \rowcolor{lightgray}    \begin{minipage}{20mm}\vspace{2pt} 2as \cite{LodayRonco2as}/    Dipt \cite{LodayRoncoDipt} \vspace{2pt}\end{minipage} & $\frac{x+x^2}{1-x}$  & $1, 5, 25, 149, 1081, 9365, 94585, 1091669$ \cite[A002050]{oeis}
             \\ \hline
             \begin{minipage}{20mm} \vspace{2pt} Dup \cite{LodayDup} \vspace{2pt}\end{minipage} &  $\frac{x}{(1+x)^2}$  & $1, -3, 7, -15, 31, -63, 127, -255, 511$ \cite[A225883]{oeis} \\ \hline
             \begin{minipage}{25mm}  \vspace{2pt} Terp \cite{BurgunderDO}/ TriDup \cite{NovelliThibonTridup} \vspace{2pt}\end{minipage}& $-\frac{1}{1+2x}+\frac{1}{1+x}$  & $1, -5, 25, -149, 1081, -9365, 94585$
            \cite[A002050]{oeis}
            \\ \hline
              \end{tabular}
            \vspace{.2cm}
           \caption{Table of the sequences $\widecheck{\mu}({}^{\mathcal{P}} \Pi(n))$ for some examples of operads $\mathcal{P}$, named as in the \emph{Operadia} database \cite{operadia}. A row is grayed if $(-1)^{n-1}\widecheck{\mu}({}^{\mathcal{P}}\Pi(n))<0$ for some $n$, which implies that $\Lambda \hbottom{\bullet}({}^{\mathcal{P}}\Pi)$ is not concentrated in degree zero.}
           \label{tab:checkmu}
       \end{table}

        \begin{rem}\label{rem: cycle index series via chains left operadic module} 
Proposition \ref{prop: cycle index series left operadic module} can also be recovered by applying the methods of \cite[Proposition 1.21]{ogerhomologyhypertree}, used in \cite[Theorem 3.2]{ogerhomologyhypertree} and extended in \cite[Proposition 1.7]{OgerSemiPointedPartition} to compute $\mathfrak{S}$-equivariant Euler characteristics. For a poset $P$ and an integer $n$, let us denote by $\mathrm{M}_n(P)$ (resp. $\widecheck{\mathrm{M}}_n(P)$) the species of multichains $x_0\leq x_1\leq \cdots \leq x_n$ of elements of $P$ such that $x_0$ is a minimal element and $x_n$ is a maximal element (resp. such that $x_0$ is a minimal element). By \cite[Proposition 1.6]{OgerSemiPointedPartition}, their cycle index series are polynomial in $n$, and are an $\mathfrak{S}$-equivariant version of the well-known \emph{zeta polynomial} introduced by Stanley \cite[p. 201]{StanleyZeta} and further studied by Edelman \cite{EdelmanZeta}. The value of this polynomial for $\widecheck{\mathrm{M}}_n(P)$ at $n=-1$ is the $\mathfrak{S}$-equivariant Euler characteristic appearing in Proposition \ref{prop: cycle index series left operadic module}. One can compute it by using the following relations, where $\mathbb{X}$ is the singleton species and $\mathbb{E}^+$ is the species of non-empty sets:
\begin{align*}
   \mathrm{M}_0({}^{\mathcal{P}}\Pi) &= \mathbb{X} \\
    \mathrm{M}_n({}^{\mathcal{P}}\Pi) & =   \mathrm{M}_{n-1}({}^{\mathcal{P}}\Pi) \circ \mathcal{P}\\
    \widecheck{\mathrm{M}}_n(^{\mathcal{P}}\Pi) &= \mathrm{M}_n({}^{\mathcal{P}}\Pi) \circ \mathbb{E}^+.
\end{align*}
Indeed, if $\QQ\mathcal{P}$ is Koszul, then $\mathrm{Z}_{\QQ\mathcal{P}}$ is invertible for $\circ$ with inverse equal to $\Sigma\mathrm{Z}_{(\QQ\mathcal{P})^!}$, and one finds formally that $\mathrm{Z}_{\widecheck{\mathrm{M}}_{-1}({}^{\mathcal{P}}\Pi)}=\Sigma\mathrm{Z}_{(\QQ\mathcal{P})^!}\circ \mathrm{Z}_{\mathbb{E}^+}$, which is the content of Proposition \ref{prop: cycle index series left operadic module}. 
        \end{rem}

    \subsubsection{A worked example of a left operadic module: left-$\operatorname{As}$-decorated partitions}\label{subsubsec: left As}

    Let us consider the operadic poset species of left-$\operatorname{As}$-decorated partitions ${}^{\operatorname{As}}\Pi$, i.e. partitions equipped with a linear order on the set of blocks (see Figure \ref{figComp}), also known as set compositions, or faces of the permutohedron. 

\begin{figure}[h]
    \centering
    \resizebox{\textwidth}{!}{
    \begin{tikzpicture}
\node (min) at (0,0) {$\{1,2,3\}$};
\coordinate[above=1cm of min] (c1);
\coordinate[above=2cm of c1] (c2);
\node[left=0.5cm of c1] (13-2) {$\{1,3\}\{2\}$};
\node[left=1cm of 13-2] (1-23) {$\{1\}\{2,3\}$};
\node[left=1cm of 1-23] (12-3) {$\{1,2\}\{3\}$};
\node[right=0.5cm of c1] (3-12) {$\{3\}\{1,2\}$};
\node[right=1cm of 3-12] (23-1) {$\{2,3\}\{1\}$};
\node[right=1cm of 23-1] (2-13) {$\{2\}\{1,3\}$};
\node[left=0.5cm of c2] (312) {$\{3\}\{1\}\{2\}$};
\node[left=1cm of 312] (132) {$\{1\}\{3\}\{2\}$};
\node[left=1cm of 132] (123) {$\{1\}\{2\}\{3\}$};
\node[right=0.5cm of c2] (321) {$\{3\}\{2\}\{1\}$};
\node[right=1cm of 321] (231) {$\{2\}\{3\}\{1\}$};
\node[right=1cm of 231] (213) {$\{2\}\{1\}\{3\}$};
\draw (min)--(13-2);
\draw (min)--(1-23);
\draw (min)--(12-3);
\draw (min)--(3-12);
\draw (min)--(23-1);
\draw (min)--(2-13);
\draw (123)--(12-3.north)--(213.south)--(2-13)--(231)--(23-1)--(321)--(3-12)--(312)--(13-2)--(132)--(1-23)--(123);
    \end{tikzpicture}
    }
    \caption{The poset of left-$\operatorname{As}$-decorated partitions on three elements $^{\operatorname{As}}\Pi\left( \{1,2,3\}\right)$.}
    \label{figComp}
\end{figure}
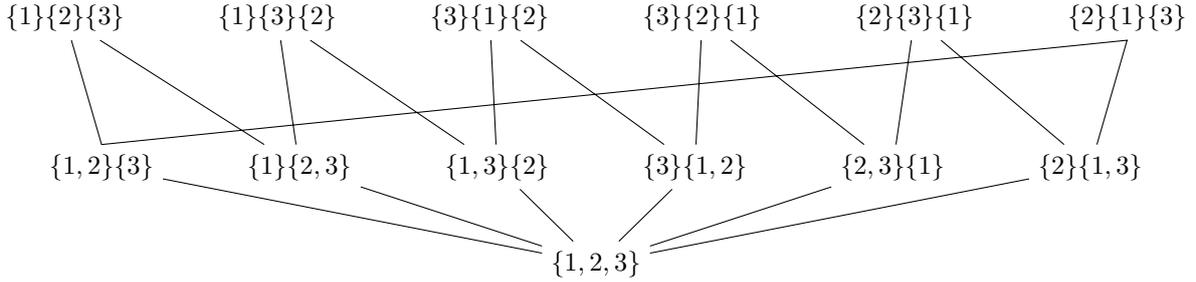

    As explained in \S\ref{sec: example left operads}, the operad structure on its cohomology $h^\bullet({}^{\operatorname{As}}\Pi)$ is the operadic desuspension of the $\operatorname{As}$ operad. We now study its left operadic module $\hbottom{\bullet}({}^{\operatorname{As}}\Pi)$. Up to a shift, it is computed by the cohomology groups of the augmented posets
    ${}^{\operatorname{As}}\Pi^+(S)$
    , obtained by adding a greatest element.    
    We also denote by ${}^{\operatorname{As}}\Pi^+(S)^{\operatorname{op}}$
     the dual poset, obtained by reversing the order.
    The first step is to prove that this poset is Cohen--Macaulay. We will prove the existence of a recursive atom ordering. Let us first recall some topological results from \cite{BjornerWachsLexSh}.

    \begin{defi}\cite[Definition 3.1]{BjornerWachsLexSh} \label{DefRAO}
        A bounded poset $P$ admits a \emph{recursive atom ordering} if it has only two elements $\hat{0}<\hat{1}$ or if it has at least three elements and there is an ordering $a_1, \ldots, a_k$ of its atoms such that:
        \begin{enumerate}
            \item For all $j=1, \ldots, k$, the interval $[a_j, \hat{1}]$ admits a recursive atom ordering in which the atoms of $[a_j, \hat{1}]$ that come first in the ordering are those covering some $a_i$, $i<j$.
            \item For all $i<j$, if $a_i,a_j<y$ then there is $k<j$ and an element $z$ such that $a_k$ and $a_j$ are covered by $z$ and $z \leq y$.
        \end{enumerate}
    \end{defi}

    \begin{prop}\cite[Theorem 3.2]{BjornerWachsLexSh}
        If a poset admits a recursive atom ordering, then it is CL-shellable, hence Cohen--Macaulay.
    \end{prop}

    The concept of total semimodularity will allow us to dispense of the first point of Definition \ref{DefRAO}:

\begin{defi} \cite[\S 5]{BjornerWachsLexSh} \label{DefTotSemiMod}
    A finite poset $P$ is \emph{totally semimodular} if it is bounded and if for any interval $[x,y]$ in $P$ and any $u, v$ in $[x,y]$ that both cover $x$, there exists $z$ in 
    $[x,y]$ such that $z$ covers both $u$ and $v$.
\end{defi}

    \begin{prop} \cite[Theorem 5.1]{BjornerWachsLexSh} \label{BjornerWachsTotSemiMod}
        A graded poset $P$ is totally semimodular if and only if for every interval $[x,y]$ of $P$, every atom ordering of $[x,y]$ is a recursive atom ordering.
    \end{prop}

We can then state the main result of this paragraph:
    \begin{prop}\label{prop: AsPi Cohen Macaulay}
    For every finite set $S$, the dual augmented poset ${}^{\operatorname{As}}\Pi^+(S)^{\operatorname{op}}$ admits a recursive atom ordering, and the poset ${}^{\operatorname{As}}\Pi^+(S)$ is therefore Cohen--Macaulay.
    \end{prop}

    \begin{proof}
    The maximal intervals in ${{}^{\operatorname{As}}\Pi}(S)^{\operatorname{op}}$  are isomorphic to boolean posets, hence totally semi-modular. Therefore, by Proposition \ref{BjornerWachsTotSemiMod}, the first point of Definition \ref{DefRAO} is satisfied for any ordering of the atoms. It remains to define a linear ordering on the set of atoms of  ${}^{\operatorname{As}}\Pi^+(n)^{\operatorname{op}}$, i.e. the maximal elements of ${{}^{\operatorname{As}}\Pi}(n)$, which satisfies the second point of the definition. 
    
    An atom of ${}^{\operatorname{As}}\Pi^+(n)^{\operatorname{op}}$ is of the form $a=\{k_1\}\cdots \{k_n\}$ with $\{k_1,\ldots,k_n\}=\{1,\ldots,n\}$, and we view it as the permutation $\sigma\in\mathfrak{S}_n$ defined by $\sigma(i)=k_i$, or simply as the word $k_1\cdots k_n$.
    The Steinhaus--Johnson--Trotter algorithm \cite{Steinhaus, Johnson, Trotter} provides a linear ordering on $\mathfrak{S}_n$, that we denote by $\prec_{\mathrm{SJT}}$, and which is defined recursively as follows (see \cite[2. end of p. 282]{Johnson}). Suppose that we have constructed the ordering $a_1 \prec_{\mathrm{SJT}}\cdots \prec_{\mathrm{SJT}} a_{n!}$ for the elements of $\mathfrak{S}_n$. For $a\in\mathfrak{S}_n$ and $j\in\{0,\ldots,n\}$, we denote by $a^j\in \mathfrak{S}_{n+1}$ the word obtained by inserting $n+1$ in position $j+1$ in $a$. Then the Steinhaus--Johnson--Trotter ordering of the elements of $\mathfrak{S}_{n+1}$ is defined as the transitive closure of:
    $$\begin{cases}
    a_i^k \prec_{\mathrm{SJT}} a_j^l & \mbox{ if } i<j;\\
    a_{2i+1}^k \prec_{\mathrm{SJT}} a_{2i+1}^l & \mbox{ if } k>l;\\
    a_{2i}^k \prec_{\mathrm{SJT}} a_{2i}^l & \mbox{ if } k<l.
    \end{cases}$$
    For instance, the ordering from left to right of the maximal elements of ${}^{\mathrm{As}}\Pi(3)$ in Figure \ref{figComp} is the Steinhaus--Johnson--Trotter ordering for $\mathfrak{S}_3$. For $a\in\mathfrak{S}_n$ and an integer $k\leq n$, we let $a^{\leq k}\in\mathfrak{S}_k$ denote the word obtained from $a$ by removing all the letter $>k$ in $a$. We have the following remarkable property of the Steinhaus--Johnson--Trotter orderings: for $a,b\in\mathfrak{S}_n$, and for some $k\leq n$,
    \begin{equation}\label{eq: SJT hereditary}
     a \preccurlyeq_{\mathrm{SJT}} b \qquad \Longrightarrow \qquad a^{\leq k} \preccurlyeq_{\mathrm{SJT}} b^{\leq k} .
    \end{equation}

    Let us now consider two maximal elements $a,b$ of ${}^{\mathrm{As}}\Pi(n)$ with $a\prec_{\mathrm{SJT}} b$, and such that there exists an element $y\in {}^{\mathrm{As}}\Pi(n)$ satisfying $a>y$ and $b>y$. We want to find an atom $c\prec_{\mathrm{SJT}}b$ and an element $z\in{}^{\mathrm{As}}\Pi(n)$ which is $\geq y$ and is covered by both $b$ and $c$ in ${}^{\mathrm{As}}\Pi(n)$. There exists a unique integer $1<\ell \leq n$ such that $a^{\leq \ell-1}=b^{\leq \ell-1}$ and $a^{\leq \ell}\neq b^{\leq \ell}$. By \eqref{eq: SJT hereditary}, we have that $a^{\leq \ell}\prec_{\mathrm{SJT}} b^{\leq \ell}$. There are two cases to treat depending on whether the letter $\ell$ appears in a smaller position in $a^{\leq \ell}$ or in $b^{\leq \ell}$.

    \begin{enumerate}[$\triangleright$]
    \item First case: the position of $\ell$ in $a^{\leq \ell}$ is smaller than in $b^{\leq \ell}$. This means that in the Steinhaus--Johnson--Trotter algorithm, we move from $a^{\leq \ell}$ from $b^{\leq \ell}$ by inserting $\ell$ from the left to the right. Let $c$ be the atom obtained from $b$ by swapping $\ell$ and its left neighbor, denoted by $g$. Then we have $a^{\leq \ell}\preccurlyeq_{\mathrm{SJT}} c^{\leq \ell}\prec_{\mathrm{SJT}} b^{\leq \ell}$ if $g\leq \ell$, and $c^{\leq \ell}=b^{\leq \ell}$ if $g>\ell$. Using \eqref{eq: SJT hereditary} we deduce that $c\prec_{\mathrm{SJT}} b$.

    Recall that one goes down in the poset ${}^{\mathrm{As}}\Pi(n)$ by merging consecutive blocks. Let $z\in{}^{\mathrm{As}}\Pi(n)$ denote the element obtained from $b$ or $c$ by merging the blocks $\{g\}$ and $\{\ell\}$. Then clearly $z$ is covered by both $b$ and $c$, and it remains to prove that $z\geq y$ in ${}^{\mathrm{As}}\Pi(n)$. Since $b\geq y$, this is equivalent to proving that $g$ and $\ell$ are in the same block of $y$. Write $a^{\leq \ell}=w_1\ell w_2 w_3$ and $b^{\leq \ell} = w_1w_2\ell w_3$, where $w_1, w_2, w_3$ are words in $\{1,\ldots, \ell-1\}$ and $w_2$ contains at least one letter. Since $a>y$ and $b>y$, all the letters in $w_2$ are in the same block as $\ell$ in $y$. But $g$ is either a letter of $w_2$ (if $g<\ell$) or is between all the letters of $w_2$ and $\ell$ in $b$ (if $g>\ell)$. Therefore, $g$ is in the same block as $\ell$ in $y$, and we are done.
    \vspace{.15cm}
    \item Second case: the position of $\ell$ in $a^{\leq \ell}$ is greater than in $b^{\leq \ell}$. This case is treated similarly: define $c$ to be the atom obtained from $b$ by swapping $\ell$ and its right neighbor instead.
    \end{enumerate}
    \end{proof}

    \begin{prop}\label{prop: AsPi left module}
    The left $\Lambda^{-1}\operatorname{As}$-module $\hbottom{\bullet}({}^{\operatorname{As}}\Pi)$ is isomorphic to the operadic quotient $\Lambda^{-1}\operatorname{Com}$.
    \end{prop}

    Here the surjection $h^\bullet({}^\mathrm{As}\Pi)\twoheadrightarrow \hbottom{\bullet}({}^{\mathrm{As}}\Pi)$, which in general is simply a morphism of left operadic modules over $h^\bullet({}^\mathrm{As}\Pi)$, is here a morphism of operads, even though there is no a priori operadic structure on $\hbottom{\bullet}({}^{\mathrm{As}}\Pi)$. This seems to be a coincidence.

    \begin{proof}
    By an application of Proposition \ref{prop: generating series left operadic module}, we have
    $$\sum_{n\geq 1} \widecheck{\mu}({}^{\mathrm{As}}\Pi(n))\frac{x^n}{n!} = -\mathrm{C}_{\KK\mathrm{As}}(1-\exp(x)) = 1-\exp(-x),$$
    where we have used the equality $\mathrm{C}_{\KK\mathrm{As}}(x)=\frac{x}{1-x}$.
    Therefore, we get $\widecheck{\mu}({}^{\mathrm{As}}\Pi(n))=(-1)^{n-1}$ for all $n\geq 1$. (This also readily follows from the identity $\sum_{k=1}^n (-1)^{k-1} S(n,k)k! = (-1)^n$ involving Stirling numbers of the second kind.)
    Proposition \ref{prop: AsPi Cohen Macaulay} therefore implies that $\hbottom{k}({}^{\mathrm{As}}\Pi(n))=0$ if $k\neq n-1$, and isomorphic to $\KK$ for $k=n-1$.
    We write $m_n$ for the element of $\mathrm{As}(n)$ which corresponds to the multiplication $x_1\cdots x_n$. By looking at the $n=2$ case, one sees that the relation $m_2=m_2^{(12)}$ is satisfied in $\hbottom{1}({}^{\mathrm{As}}\Pi(2))$. By using the left operadic module structure, one deduces that the relation $m_n^\sigma=m_n$ holds in $\hbottom{\bullet}({}^{\mathrm{As}}(n))$ for all $n$ and all $\sigma \in \mathfrak{S}_n$. Therefore, the surjection $h^\bullet({}^\mathrm{As}\Pi)\twoheadrightarrow \hbottom{\bullet}({}^{\mathrm{As}}\Pi)$ factors as $\Lambda^{-1}\mathrm{As}\twoheadrightarrow \Lambda^{-1}\mathrm{Com}\twoheadrightarrow \hbottom{\bullet}({}^{\mathrm{As}}\Pi)$. This last map is an isomorphism for dimension reasons.
    \end{proof}

    \begin{rem}
    Shortly after a first version of the present article appeared on arXiv, Sagan and Sundaram posted an independent study \cite{sagansundaram} of the posets ${}^{\mathrm{As}}\Pi(n)$, in which they constructed an alternative recursive atom ordering. They also explain that the existence of a recursive atom ordering, along with the formula $\widecheck{\mu}({}^{\mathrm{As}}\Pi(n))=(-1)^{n-1}$, is a consequence of the fact that ${}^{\mathrm{As}}\Pi(n)$ is the face lattice of the associahedron \cite[Theorems 3.2 and 4.5]{BjornerWachsLexSh}. The results of \cite{sagansundaram} on the generalizations $\breve{\Omega}_n^{(d)}$, where all blocks have size congruent to $1$ modulo a fixed integer $d$, are ordered versions of \cite{hanlonwachs} and could be interpreted in our formalism as an instance of bi-decorated partitions (\S\ref{sec: Bdp}).
    \end{rem}

\subsection{Right-decorated partitions and right operadic modules}\label{subsec: Rdp mod}

  \subsubsection{Interpretation as cobar construction with coefficients}

        The poset $\Pi^{\mathcal{Q}}(S)$ of right-$\mathcal{Q}$-decorated partitions of $S$ has a greatest element but potentially many minimal elements, which are in bijection with $\mathcal{P}(S)$. This means that the left operadic module $\hbottom{\bullet}({}^{\mathcal{P}}\Pi)$ is trivial, but the right operadic module $\htop{\bullet}({}^\mathcal{P}\Pi)$ may be interesting. It can be expressed in terms of the cobar construction with coefficients, as follows. 
        
        Recall \cite{fressepartitionposets} that for a linear cooperad $\mathcal{C}$ as in the previous paragraph, and a right $\mathcal{C}$-comodule $R$, we have the cobar construction with coefficients and its variant with levels:
        $$\Omega(R,\mathcal{C}) = R\circ \Omega(\mathcal{C}) \qquad \mbox{ and } \qquad \Omega_{\mathrm{level}}(R,\mathcal{C}) = R\circ \Omega_{\mathrm{level}}(\mathcal{C}).$$ Note that $\Omega(R,\mathcal{C})$ is a right operadic $\Omega(\mathcal{C})$-module. We have the delevelization quasi-isomorphism
        $$\Omega_{\mathrm{level}}(R,\mathcal{C}) \stackrel{\sim}{\To} \Omega(R,\mathcal{C}).$$

        Like any set operad, $\mathcal{Q}$ has a canonical right operadic module $R_{\mathcal{Q}}$, defined by
        $$R_{\mathcal{Q}}(S)=\{*\}$$
        for every non-empty finite set $S$, and $R_{\mathcal{Q}}(\varnothing)=\varnothing$. Therefore, the linear cooperad $\KK\mathcal{Q}^\vee$ has a right operadic comodule $\KK R_{\mathcal{Q}}^\vee$. The cobar construction $\Omega(\KK R_{\mathcal{Q}}^\vee, \KK\mathcal{Q}^\vee)$ has a basis consisting of reduced forests with nodes labeled by elements of $\mathcal{Q}$ and leaves labeled by elements of $S$.

        \begin{prop}\label{prop: right operadic module decorated partitions}
        Let $\mathcal{Q}$ be a right-basic set operad. We have an isomorphism of complexes
        $$\chainstop{\bullet}(\Pi^{\mathcal{Q}}(S)) \simeq \Omega_{\mathrm{level}}(\KK R_{\mathcal{Q}}^\vee,\KK\mathcal{Q}^\vee)(S)$$
        which induces an isomorphism 
        $$\htop{\bullet}(\Pi^{\mathcal{Q}}) \simeq \operatorname{H}^\bullet(\Omega(\KK R_{\mathcal{Q}}^\vee, \KK\mathcal{Q}^\vee)),$$
        compatible with the right operadic module structures over $h^\bullet(\Pi^{\mathcal{Q}}) \simeq \operatorname{H}^\bullet(\Omega(\KK\mathcal{Q}^\vee))$.
        \end{prop} 

        \begin{proof}
        It follows from a straighforward adaptation of the proof of Proposition \ref{prop: right koszul}.
        \end{proof} 

    \subsubsection{($\mathfrak{S}$-equivariant) Euler characteristics}

        We do not know of any operadic tools which would allow to compute the cohomology of the cobar construction appearing in Proposition \ref{prop: right operadic module decorated partitions}, in the spirit of Koszul duality (Proposition \ref{prop: right koszul}). However, we have the following analogue of Proposition \ref{prop: cycle index series left operadic module}.

        \begin{prop}\label{prop: cycle index series right operadic module}
        Let $\mathcal{Q}$ be a right-basic set operad such that $\QQ\mathcal{Q}$ is Koszul, with Koszul dual $(\QQ\mathcal{Q})^!$. Then we have the equality in the ring $R_\QQ$:
            $$\sum_{k\geq 0}(-1)^k \mathrm{Z}_{\htop{k}(\Pi^{\mathcal{Q}})}  =  \left( \exp\Bigg(\sum_{n\geq 1}\frac{p_n}{n}\Bigg) -1 \right) \circ (\Sigma \mathrm{Z}_{(\QQ\mathcal{Q})^!} ). $$
        \end{prop}

        For a poset $P$, we let
        $$\widehat{\mu}(P) = \sum_{k\geq 0} (-1)^k \mathrm{dim}(\htop{k}(P)).$$ 
        We have the following analogue of Proposition \ref{prop: generating series left operadic module}.

        \begin{prop}\label{prop: generating series right operadic module}
        Let $\mathcal{Q}$ be a right-basic set operad such that $\QQ\mathcal{Q}$ is Koszul, with Koszul dual $(\QQ\mathcal{Q})^!$. Then we have the equality of formal power series:
        \begin{equation}\label{eq: generating series mobius numbers right operadic module}
        \sum_{n\geq 1} \widehat{\mu}(\Pi^{\mathcal{Q}}(n)) \frac{x^n}{n!} = \exp(-\mathrm{C}_{(\QQ\mathcal{Q})^!}(-x))-1.    
        \end{equation}
        \end{prop}

        The sequences $\widecheck{\mu}(\Pi^{\mathcal{Q}}(n))$ for some examples are presented in Table \ref{tab:hatmu}. All the operads $\mathcal{Q}$ in this table are right-basic with $\QQ\mathcal{Q}$ Koszul, so that Proposition \ref{prop: generating series right operadic module} applies. We see that in the cases $\mathcal{Q}=\mathrm{As}, \mathrm{NAP}, \mathrm{Dias}, \mathrm{2as}, \mathrm{Dipt}, \mathrm{Dup}, \mathrm{Terp},\mathrm{TriDup}$, the operad $\htop{\bullet}(\Pi^{\mathcal{Q}})$ is not concentrated in degree zero. We develop some of these examples in \S\ref{sec: open questions}.
        
          \begin{table}[h!]
           \centering
           \begin{tabular}{|P{25mm}|c|P{90mm}|}
           \hline
$\mathcal{Q} $ &      $-\mathrm{C}_{(\QQ\mathcal{Q})^!}(-x)$
& $\widehat{\mu}(\Pi^{\mathcal{Q}}(n)),  n\geq 1$
\\ \hline
     \rowcolor{lightgray}   \begin{minipage}{20mm}\vspace{2pt}      As \cite{PeirceAs} \vspace{2pt}\end{minipage} & $\frac{x}{1+x}$  & $1, -1, 1, 1, -19, 151, -1091, 7841$ \cite[A111884]{oeis}
             \\ \hline
           \begin{minipage}{25mm} \vspace{2pt}  Perm \cite{ChapotonPerm}/ $\mathrm{Com}_2$ \cite{DotsenkoKhoroshkinCom2}  \vspace{2pt}\end{minipage}  & $\sum_{n \geq 1} (-n)^{n-1} \frac{x^n}{n!}$  & $1, -1, 4, -27, 256, -3125, 46656$ 
     \cite[A177885]{oeis}       
 \\ 
             \hline
    \rowcolor{lightgray}          NAP \cite{LivernetNAP}   & $x\exp(-x)$  & $1, -1, -2, 9, -4, -95, 414, 49, -10088$ \cite[A003725]{oeis}
           \\ \hline
     \rowcolor{lightgray}        Dias \cite{LodayDias} & $\sum_{n \geq 1}\frac{-(2n)!}{(n+1)!} \frac{(-x)^n}{n!}$  & $1, -3, 19, -191, 2661, -47579, 1040047$ \cite[A383990]{oeis}
\\ \hline
             Trias \cite{LodayRoncoTrias} & $\frac{1+3x-\sqrt{1+6x+x^2}}{4x}$  & $1, -5, 49, -743, 15421, -407909$ \cite[A383991]{oeis}
\\ \hline
    \rowcolor{lightgray}    \begin{minipage}{20mm} \vspace{2pt} 2as \cite{LodayRonco2as}/    Dipt \cite{LodayRoncoDipt} \vspace{2pt}\end{minipage}& $\frac{x+x^2}{1-x}$  & $1, 5, 25, 169, 1361, 12781, 136585$ \cite[A112242]{oeis}
             \\ \hline
           \begin{minipage}{25mm} \vspace{2pt}   ComTrias~\cite{vallettepartitionposets} \end{minipage}  & $\log\left(\frac{1+\sqrt{1+4x}}{2}\right)$ & $1, -2, 12, -120, 1680, -30240, 665280$ \cite[A001813]{oeis}
\\ \hline
     \rowcolor{lightgray}         Dup \cite{LodayDup} & $\frac{x}{(1+x)^2}$  & $1, -3, 7, 1, -219, 2581, -22973, 162177$ \cite[A318215]{oeis}
             \\ \hline
   \rowcolor{lightgray}   \begin{minipage}{25mm} \vspace{2pt}        Terp \cite{BurgunderDO}/ TriDup \cite{NovelliThibonTridup} \vspace{2pt}\end{minipage}  & $-\frac{1}{1+2x}+\frac{1}{1+x}$  & $1, -5,
25, -119, 301, 5611, -171275$ \cite[A383993]{oeis} \\
    \hline
              \end{tabular}
            \vspace{.2cm}
           \caption{Table of the sequences $\widehat{\mu}(\Pi^{\mathcal{Q}}(n))$ for some examples of operads $\mathcal{Q}$, named as in the \emph{Operadia} database \cite{operadia}. A row is grayed if $(-1)^{n-1}\widehat{\mu}(\Pi^{\mathcal{Q}}(n))<0$ for some $n$, which implies that $\Lambda \htop{\bullet}(\Pi^{\mathcal{Q}})$ is not concentrated in degree zero.}
           \label{tab:hatmu}
       \end{table}

        \begin{rem}\label{rem: cycle index series via chains right operadic module}
        One can give an alternative proof of Proposition \ref{prop: cycle index series right operadic module} as in the case of left-decorated partitions (Remark \ref{rem: cycle index series via chains left operadic module}). For a poset $P$ and an integer $n$, let us denote by $\widehat{\mathrm{M}}_n(P)$ the species of multichains $x_0\leq x_1\leq \cdots\leq x_n$ for which $x_n$ is a maximal element of $P$. One argues as in Remark \ref{rem: cycle index series via chains left operadic module} by using the following identities between species:
        \begin{align*}
        \mathrm{M}_0(\Pi^{\mathcal{Q}})&=\mathbb{X} \\
        \mathrm{M}_n(\Pi^{\mathcal{Q}})  & = \mathcal{Q} \circ   \mathrm{M}_{n-1}(\Pi^{\mathcal{Q}}) \\
        \widehat{\mathrm{M}}_n(\Pi^{\mathcal{Q}}) &= \mathbb{E}^+ \circ \mathrm{M}_n(\Pi^{\mathcal{Q}}).
        \end{align*}    
        \end{rem}

    \subsubsection{A worked example of a right operadic module: right-$\operatorname{As}$-decorated partitions}\label{sec: As right operadic module}

           Let us consider the operadic poset species of right-$\operatorname{As}$-decorated partitions, i.e. partitions equipped with a linear order on each block (see Figure \ref{figComp2}). 

        \begin{figure}[h]
    \centering
    \begin{tikzpicture}
\node (min) at (0,0) {$\{1,2,3\}$};
\coordinate[below=1cm of min] (c1);
\coordinate[below=2cm of c1] (c2);
\node[left=0.5cm of c2] (231) {$\{231\}$};
\node[left=1cm of 231] (312) {$\{312\}$};
\node[left=1cm of 312] (123) {$\{123\}$};
\node[right=0.5cm of c2] (132) {$\{132\}$};
\node[right=1cm of 132] (321) {$\{321\}$};
\node[right=1cm of 321] (213) {$\{213\}$};
\node[above=1cm of 312] (31-2) {$\{31,2\}$};
\node[above=1cm of 132] (1-32) {$\{1,32\}$};
\node[above=1cm of 123] (12-3) {$\{12,3\}$};
\node[above=1cm of 321] (21-3) {$\{21,3\}$};
\node[above=1cm of 231] (1-23) {$\{1,23\}$};
\node[above=1cm of 213] (13-2) {$\{13,2\}$};
\draw (min)--(31-2);
\draw (min)--(1-32);
\draw (min)--(12-3);
\draw (min)--(1-23);
\draw (min)--(21-3);
\draw (min)--(13-2);
\draw (123.north)--(12-3.south)--(312.north)--(31-2.south)--(231.north)--(1-23.south)--(123.north);
\draw (213.north)--(21-3.south)--(321.north)--(1-32.south)--(132.north)--(13-2.south)--(213.north);
    \end{tikzpicture}
    \caption{The poset of right-$\operatorname{As}$-decorated partitions on three elements $\Pi^{\operatorname{As}}\left( \{1,2,3\}\right)$.}
    \label{figComp2}
\end{figure}
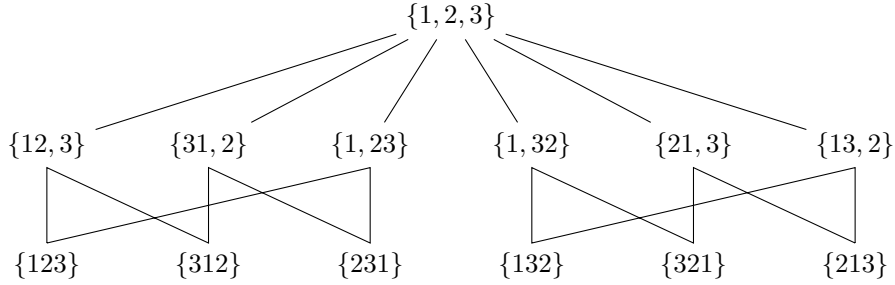

 As explained in \S\ref{sec: examples right operads}, the operad structure on its cohomology $h^\bullet(\Pi^{\operatorname{As}})$ is the operadic desuspension of the $\operatorname{As}$ operad. We now study its right operadic module $\htop{\bullet}(\Pi^{\operatorname{As}})$.
By contemplating the Hasse diagram of $\Pi^{\mathrm{As}}(\{1,2,3\})$ in Figure \ref{figComp2}, we get by \eqref{Relhbot} that $\htop{\bullet}(\Pi^{\mathrm{As}}(\{1,2,3\}))$ is up to a shift the reduced cohomology of the disjoint union of two circles. Therefore, it is not concentrated in just one degree. The results (obtained with SageMath) of the computation of $\htop{\bullet}(\Pi^{\mathrm{As}}(n))$ for $n\leq 5$, for $\KK=\mathbb{Z}$, are presented in Table \ref{tab:cohomSage}. We leave it to the motivated reader to detemine the structure of $\Lambda\htop{\bullet}(\Pi^{\mathrm{As}})$ as a (graded) right module over the $\mathrm{As}$ operad.

                \renewcommand{\arraystretch}{1.5}
                \begin{table}[h!]
                \centering
                \begin{tabular}{|c|c|c|}
                \hline
                $n$ & $\htop{\bullet}(\Pi^{\operatorname{As}} (n))$ & $\widehat{\mu}(\Pi^{\mathrm{As}}(n)) $\\ \hline
                1 & $\htop{0}=\ZZ$ & $1$ \\ \hline
                2 & $\htop{1}=\ZZ$ & $-1$ \\ \hline 
                3 & $\htop{1}=\ZZ$ , $\htop{2}=\ZZ^2$ & $1$
                \\ \hline
                4& $\htop{2}=\ZZ^7$ , $\htop{3}=\ZZ^6$ & $1$ 
                \\ \hline
                5 & $\htop{3}=\ZZ^{43}$ , $\htop{4}=\ZZ^{24}$ & $-19$ 
                \\ \hline
                \end{tabular}
                \vspace{.2cm}
                \caption{Experimental results for $\htop{\bullet}(\Pi^{\operatorname{As}}(n))$ for $n\leq 5$ (we only mention the non-zero cohomology groups).}
                \label{tab:cohomSage}
              \end{table}

              Regarding Euler characteristics, Proposition \ref{prop: generating series right operadic module} yields the identity:
              \begin{equation}\label{eq: generating series Pi As}
              \sum_{n\geq 1} \widehat{\mu}(\Pi^{\mathrm{As}}(n))  \frac{x^n}{n!} = \exp\left(\frac{x}{1+x}\right) -1.
              \end{equation}
This is \cite[A111884]{oeis}, and we see that $(-1)^{n-1}\widehat{\mu}(\Pi^{\mathrm{As}}(n))<0$ for all $4 \leq n \leq  12$, hence $\htop{\bullet}(\Pi^{\mathrm{As}}(n))$ cannot be concentrated in degree $n-1$ for those integers $n$. Therefore, the posets $\Pi^{\mathrm{As}}_+(n)$, obtained from $\Pi^{\mathrm{As}}$ by adding a least element, are not Cohen--Macaulay for $3 \leq n \leq  12$. 

\begin{rem}
One can refine the identity \eqref{eq: generating series Pi As} as follows. Let us denote by $f(n,k)$ the number of elements of rank $k$ in $\Pi^{\mathrm{As}}(n)$, i.e. the number of partitions of $\{1,\ldots,n\}$ into $k$ lists. The exponential generating series of these numbers equals
$$\sum_{n \geq 1, k \geq 1} f(n,k) \,t^k \frac{x^n}{n!} = \mathrm{C}_{\KK\mathrm{Com}}(t\,\mathrm{C}_{\KK\mathrm{As}}(x)) = \exp\left( \frac{tx}{1-x} \right) -1,$$
and \eqref{eq: generating series Pi As} follows from setting $t=-1$ and replacing $x$ with $-x$ in this equality.
The numbers $f(n,k)$ correspond to the OEIS sequence \cite[A105278]{oeis} and are given by:
\begin{equation*}
    f(n,k) = \binom{n}{k} \frac{(n-1)!}{(k-1)!}.
\end{equation*}
We mention the following recurrence relation, which does not appear in the OEIS at the time of writing:
\begin{equation}\label{eq: new recurrence relation partition into lists}
    f(n,k)= f(n-1,k-1)+(n+k-1)f(n-1,k).
\end{equation}
It follows from enumerating the $f(n,k)$ partitions of $\{1,\ldots,n\}$ into $k$ lists by looking at the position of $n$, as follows. The number of such partitions for which $n$ is alone in its list is $f(n-1,k-1)$. The others are obtained by inserting $n$ in one of the $f(n-1,k)$ partitions of $\{1,\ldots,n-1\}$ into $k$ lists. There are $(n+k-1)$ ways of doing so; indeed, if $\ell_1,\ldots, \ell_k$ denote the lengths of the lists, with $\ell_1+\cdots +\ell_k=n-1$, then the number of insertions is $(\ell_1+1)+\cdots +(\ell_k+1)=n+k-1$.
\end{rem}

\subsubsection{A worked example of a right operadic module: right-$\operatorname{Perm}$-decorated partitions}\label{sec: perm right operadic module}

Let us consider the operadic poset species of right-$\operatorname{Perm}$-decorated partitions, i.e. partitions with a pointed element in each block (see Figure \ref{fig:permright}).

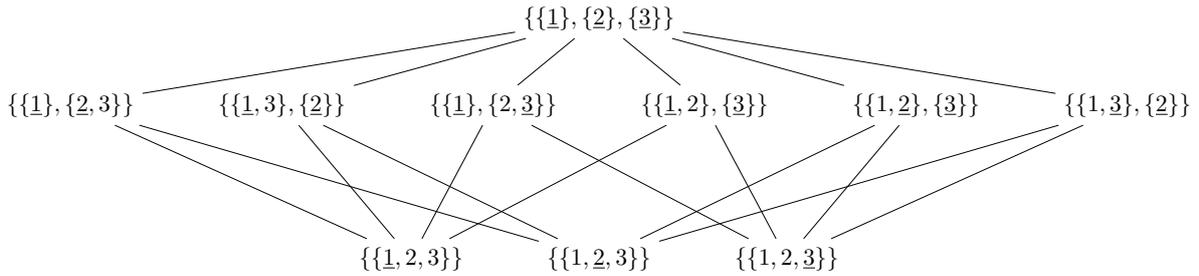
\begin{figure}[h]
    \centering
    \resizebox{\textwidth}{!}{
    \begin{tikzpicture}
\node (min) at (0,0) {$\{\{\underline{1}\},\{\underline{2}\},\{\underline{3}\}\}$};
\coordinate[below=1cm of min] (c1);
\node[left=0.5cm of c1] (3) {$\{\{\underline{1}\},\{2,\underline{3}\}\}$};
\node[left=1cm of 3] (2) {$\{\{\underline{1},3\},\{\underline{2}\}\}$};
\node[left=1cm of 2] (1) {$\{\{\underline{1}\},\{\underline{2},3\}\}$};
\node[right=0.5cm of c1] (4) {$\{\{\underline{1},2\},\{\underline{3}\}\}$};
\node[right=1cm of 4] (5) {$\{\{1,\underline{2}\},\{\underline{3}\}\}$};
\node[right=1cm of 5] (6) {$\{\{1,\underline{3}\},\{\underline{2}\}\}$};
\node[below=2cm of c1] (2p) {$\{\{1,\underline{2},3\}\}$};
\node[left=1cm of 2p] (1p) {$\{\{\underline{1},2,3\}\}$};
\node[right=1cm of 2p] (3p) {$\{\{1,2,\underline{3}\}\}$};
\draw (min)--(1);
\draw (min)--(2);
\draw (min)--(3);
\draw (min)--(4);
\draw (min)--(5);
\draw (min)--(6);
\draw (1p)--(1)--(2p);
\draw (1p)--(2)--(2p);
\draw (1p)--(3)--(3p);
\draw (1p)--(4)--(3p);
\draw (3p)--(5)--(2p);
\draw (3p)--(6)--(2p);
    \end{tikzpicture}
    }
    \caption{The poset of right-$\operatorname{Perm}$-decorated partitions on three elements $\Pi^{\operatorname{Perm}}(\{1,2,3\})$. The pointed element in each block is underlined.}
    \label{fig:permright}
\end{figure}

As explained in \S\ref{sec: examples right operads}, the operad structure on its cohomology $h^\bullet(\Pi^{\operatorname{Perm}})$ is the operadic desuspension of the $\operatorname{PreLie}$ operad. We now study its right operadic module $\htop{\bullet}(\Pi^{\operatorname{Perm}})$. Up to a shift, it is computed by the cohomology groups of the augmented posets $\Pi^{\operatorname{Perm}}_+(S)$, obtained by adding a least element. We denote by $\Pi^{\operatorname{Perm}}_+(S)^{\mathrm{op}}$ the dual poset, obtained by reversing the order. Recalling Definition \ref{DefTotSemiMod}, we have:

\begin{prop}\label{prop: Pi Perm Cohen Macaulay}
For every finite set $S$, the poset $\Pi^{\operatorname{Perm}}_+(S)^{\mathrm{op}}$ is totally semimodular, and therefore $\Pi^{\operatorname{Perm}}_+(S)$ is Cohen--Macaulay.
\end{prop}

\begin{proof}
By \cite[Lemma 1.10]{chapotonvallette}, every maximal interval of $\Pi^{\operatorname{Perm}}(S)^{\mathrm{op}}$ is totally semimodular. Therefore, all that remains is the case of the intervals of $\Pi^{\mathrm{Perm}}_+(S)^{\mathrm{op}}$ whose lower bound is the added least element. We now phrase things in terms of the order of $\Pi^{\mathrm{Perm}}$: let $(\pi,\xi), (\pi',\xi')\in \Pi^{\mathrm{Perm}}(S)$ be two pointed partitions of $S$ which are covered by a common element $(\beta,\nu)$, and let us produce an element $(\alpha,\mu)$ of $\Pi^{\mathrm{Perm}}_+(S)$ which is covered by both $(\pi,\xi)$ and $(\pi',\xi')$. If $\beta$ has $2$ blocks, then both $\pi$ and $\pi'$ consist of only one block and one can choose $(\alpha,\mu)$ to be the added least element of $\Pi^{\mathrm{Perm}}_+(S)$. Otherwise, $\beta$ has at least $3$ blocks, and since $\pi$ (resp. $\pi'$) is obtained from $\beta$ by merging exactly two blocks, there exists an element $s\in S$ which is pointed by both $\xi$ and $\xi'$. This implies that both $(\pi,\xi)$ and $(\pi',\xi')$ are larger than the pointed partition consisting of a block $S$ pointed by $s$. Therefore, $(\pi,\xi)$ and $(\pi',\xi')$ live in an interval of $\Pi^{\operatorname{Perm}}$ and applying again \cite[Lemma 1.10]{chapotonvallette} yields the result.
\end{proof}

\begin{rem} Recall that the existence of an EL-shelling implies the existence of a recursive atom ordering, while the converse is false by \cite{liCLvsEL}\footnote{We thank Sheila Sundaram for providing this reference.} . In \cite{whitneyTwins}, Gonz\'{a}lez D'Le\'{o}n, Hallam and Quiceno Dur\'{a}n give an EL-labeling of the maximal intervals of the poset $\Pi^{\mathrm{Perm}}(S)$. It would be interesting to determine whether this labeling can be extended to an EL-labeling of the poset $\Pi^{\operatorname{Perm}}_+(S)$. This would give another proof of the fact that this poset is Cohen--Macaulay and a stronger property of this poset: EL-shellability.
\end{rem}

\begin{prop} \label{prop:dim htop right Perm}
For all $n\geq 1$, the cohomology $\htop{\bullet}(\Pi^{\mathrm{Perm}}(n))$ is concentrated in degree $n-1$, and 
$$\dim\htop{n-1}(\Pi^{\mathrm{Perm}}(n)) = (n-1)^{n-1}.$$
\end{prop}

\begin{proof}
The first statement is a consequence of Proposition \ref{prop: Pi Perm Cohen Macaulay}. For the second statement, note that $\dim\htop{n-1}(\Pi^{\mathrm{Perm}}(n))=(-1)^{n-1}\widehat{\mu}(\Pi^{\mathrm{Perm}}(n))$, and that Proposition \ref{prop: generating series right operadic module} yields the identity:
\begin{equation}\label{eq: generating series Pi Perm}
\sum_{n\geq 1} \widehat{\mu}(\Pi^{\mathrm{Perm}}(n))  \frac{x^n}{n!} = \exp(-\mathrm{C}_{\mathrm{PreLie}}(-x))-1.
\end{equation}
Since $\mathrm{C}_{\mathrm{PreLie}}(x)\exp(-\mathrm{C}_{\mathrm{PreLie}}(x)) = x$, we get that $\widehat{\mu}(\Pi^{\mathrm{Perm}}(n)) = (-1)^{n-1}(n-1)^{n-1}$ (see \cite[A177885]{oeis}) and the claim follows.
\end{proof}

We leave it to the motivated reader to detemine the structure of $\Lambda\htop{\bullet}(\Pi^{\mathrm{Perm}})$ as a right module over the $\mathrm{PreLie}$ operad, and to potentially draw connections to the many combinatorial objects enumerated by the sequence $(n-1)^{n-1}$, see \cite[A000312]{oeis}.

\section{More examples of operadic poset species and their cohomology}\label{sec: more examples}

            \subsection{Non-singleton boolean posets and metabelian Lie algebras}\label{sec:metablie}

\begin{defi}
Let $S$ be a finite set. The \emph{non-singleton boolean poset} on $S$ is the subposet
$$\mathrm{NS}(S)\subset \Pi(S)$$
consisting of those partitions of $S$ with at most one block of cardinality $\geq 2$.
\end{defi}

The name comes from the fact that for $|S|\geq 2$ we can also describe $\mathrm{NS}(S)$ as the poset of non-singleton subsets of $S$ ordered by reverse inclusion, such a subset $T$ corresponding to the partition of $S$ whose only non-singleton block is $T$ if $T\neq\varnothing$, and to the partition of $S$ whose blocks are all singletons if $T=\varnothing$. We have $\mathrm{NS}(S)=\Pi(S)$ if $|S|\leq 3$, and the Hasse diagram of $\mathrm{NS}(4)$ is given in Figure \ref{fig:BuildingSetPoset}.

\begin{figure}[h!]
    \centering
    \resizebox{\textwidth}{!}{
\begin{tikzpicture}
    \node (min) at (0,0) {$\{\{1,2,3,4\}\}$};
\coordinate[above=1cm of min] (c1);
\coordinate[above=2cm of c1] (c2);
 \node[above=1cm of c2] (max) {$\{\{1\},\{2\},\{3\},\{4\}\}$};
\node[left=0.5cm of c1] (124) {$\{\{1,2,4\},\{3\}\}$};
\node[left=1cm of 124] (123) {$\{\{1,2,3\},\{4\}\}$};
\node[right=0.5cm of c1] (134) {$\{\{1,3,4\},\{2\}\}$};
\node[right=1cm of 134] (234) {$\{\{2,3,4\},\{1\}\}$};
\node[left=0.5cm of c2] (14) {$\{\{1,4\},\{2\},\{3\}\}$};
\node[left=1cm of 14] (13) {$\{\{1,3\},\{2\},\{4\}\}$};
\node[left=1cm of 13] (12) {$\{\{1,2\},\{3\},\{4\}\}$};
\node[right=0.5cm of c2] (23) {$\{\{2,3\},\{1\},\{4\}\}$};
\node[right=1cm of 23] (24) {$\{\{2,4\},\{1\},\{3\}\}$};
\node[right=1cm of 24] (34) {$\{\{3,4\},\{1\},\{2\}\}$};
\draw (min)--(123);
\draw (min)--(124);
\draw (min)--(134);
\draw (min)--(234);
\draw (123.north)--(12.south)--(124.north)--(14.south)--(134.north)--(13.south)--(123.north);
\draw (123.north)--(23.south)--(234.north)--(24.south)--(124.north);
\draw (134.north)--(34.south)--(234.north);
\draw (13)--(max)--(12);
\draw (14)--(max)--(23);
\draw (24)--(max)--(34);
\end{tikzpicture}
}

 \vspace{.3cm}
 \resizebox{0.7\textwidth}{!}{
\begin{tikzpicture}
    \node (min) at (0,0) {$\{1,2,3,4\}$};
\coordinate[above=1cm of min] (c1);
\coordinate[above=2cm of c1] (c2);
 \node[above=1cm of c2] (max) {$\emptyset$};
\node[left=0.5cm of c1] (124) {$\{1,2,4\}$};
\node[left=1cm of 124] (123) {$\{1,2,3\}$};
\node[right=0.5cm of c1] (134) {$\{1,3,4\}$};
\node[right=1cm of 134] (234) {$\{2,3,4\}$};
\node[left=0.5cm of c2] (14) {$\{1,4\}$};
\node[left=1cm of 14] (13) {$\{1,3\}$};
\node[left=1cm of 13] (12) {$\{1,2\}$};
\node[right=0.5cm of c2] (23) {$\{2,3\}$};
\node[right=1cm of 23] (24) {$\{2,4\}$};
\node[right=1cm of 24] (34) {$\{3,4\}$};
\draw (min)--(123);
\draw (min)--(124);
\draw (min)--(134);
\draw (min)--(234);
\draw (123.north)--(12.south)--(124.north)--(14.south)--(134.north)--(13.south)--(123.north);
\draw (123.north)--(23.south)--(234.north)--(24.south)--(124.north);
\draw (134.north)--(34.south)--(234.north);
\draw (13)--(max)--(12);
\draw (14)--(max)--(23);
\draw (24)--(max)--(34);
\end{tikzpicture}
}
    \caption{The Hasse diagram of the non-singleton boolean poset $\mathrm{NS}(4)$, in the language of partitions (top) or subsets (bottom).}
    \label{fig:BuildingSetPoset}
\end{figure}
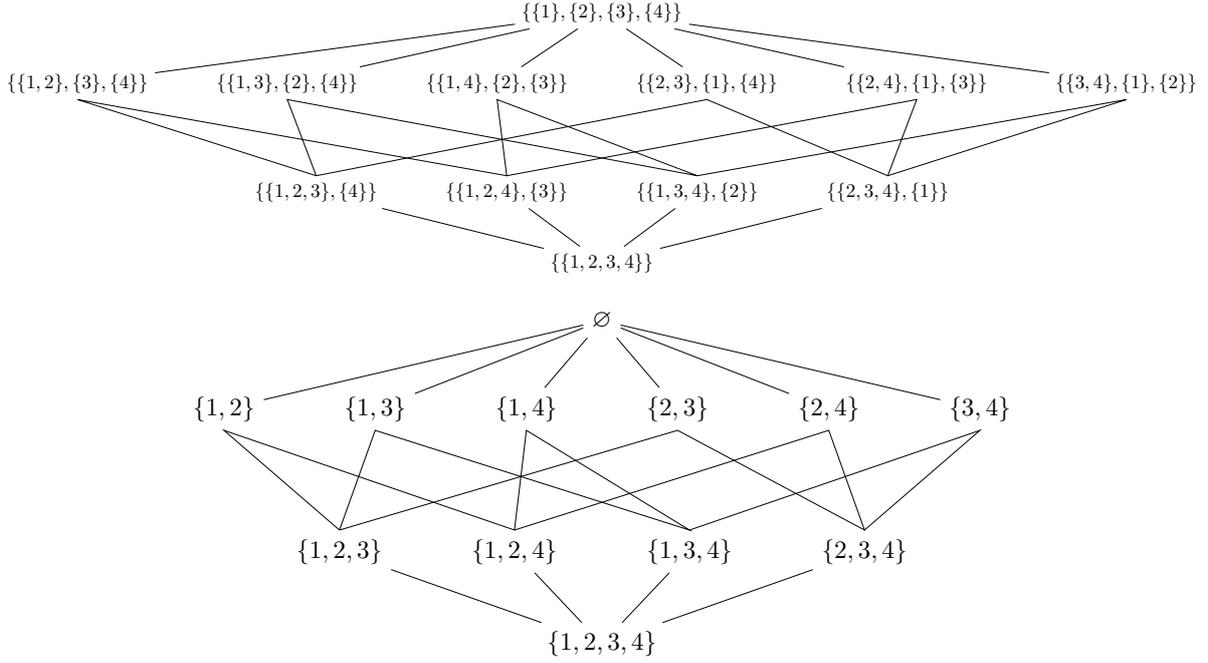

\begin{rem}
The partitions of $S$ with exactly one block of cardinality $\geq 2$ form the minimal building set of the dual of $\Pi(S)$ in the sense of \cite{deConciniProcesi, feichtnerkozlov}.
\end{rem}

For $\pi\in \mathrm{NS}(S)$, we have morphisms of posets
$$\varphi_\pi:\mathrm{NS}_{\leq \pi}(S) \To \mathrm{NS}(\pi) \qquad\qquad \mbox{ and } \qquad\qquad \psi_\pi:\mathrm{NS}_{\geq \pi}(S) \To \prod_{T\in \pi}\mathrm{NS}(T),$$
obtained as restrictions of \eqref{eq: phi and psi for partitions}. The former is injective but not surjective, and the latter is an isomorphism. Together with the injection $a:\mathrm{NS}\to \Pi$, they give the poset species $\mathrm{NS}$ the structure of an operadic poset species (that one could call an operadic poset \emph{subspecies} of $\Pi$.)

\begin{defi}
A Lie algebra $\mathfrak{g}$ is \emph{metabelian} if for all $a,b,c,d\in\mathfrak{g}$ we have $[[a,b],[c,d]]=0$.
\end{defi}

Metabelian Lie algebras are the algebras over an operad $\mathrm{MetabLie}$, which is a quotient of $\mathrm{Lie}$.

\begin{thm}\label{thm: metablie}
There exists a unique isomorphism of graded operads
$$\Lambda^{-1}\mathrm{MetabLie} \stackrel{\sim}{\To} h^\bullet(\mathrm{NS})$$
that sends the Lie bracket to the generator $[12<1|2]$ of $h^1(\mathrm{NS}(2))$.
\end{thm}

In order to prove this theorem, we first need to compute the rank of the cohomology of $\mathrm{NS}(S)$.

\begin{prop}\label{prop: rank of cohomology NS}
For every finite set $S$, the poset $\mathrm{NS}(S)$ is a geometric lattice, and hence is Cohen--Macaulay. Its only non-zero cohomology group is $h^{|S|-1}(\mathrm{NS}(S))$, which is free of rank $|S|-1$.
\end{prop}

\begin{proof}
We assume that $|S|\geq 2$ and identify $\mathrm{NS}(S)$ with the poset of non-singleton subsets of $S$, ordered by reverse inclusion. The least element is $S$, the greatest element is $\emptyset$, and the atoms are the sets $S \setminus \{a\}$, for $a\in S$.
\begin{enumerate}[$\triangleright$]
    \item It is clear that $\mathrm{NS}(S)$ is a lattice. The join of two elements $A$ and $B$ is given by $A \vee B = A \cap B$ if $A \cap B$ is not a singleton and by $A \vee B =\emptyset$ otherwise. The meet of two elements $A$ and $B$ is given by $A \wedge B =A \cup B$.
    \item This poset is atomistic because any non-singleton non-empty subset $A\subset S$ is the intersection of the sets $S\setminus \{b\}$ for $b\notin A$.
    \item The lattice $\mathrm{NS}(S)$ is graded, where the rank function is given by $\operatorname{rk}(A)=n-|A|$ if $A \neq \emptyset$ and $\operatorname{rk}(\emptyset)=n-1$. It remains to check that the following (semimodularity) inequality is satisfied for all $A,B\in\mathrm{NS}(S)$:
    \begin{equation}
        \operatorname{rk}(A)+\operatorname{rk}(B) \geq \operatorname{rk}(A \wedge B) + \operatorname{rk}(A \vee B).
    \end{equation}
    It is enough to treat the case $A\neq\varnothing$, $B\neq \varnothing$, for which $\operatorname{rk}(A)+\operatorname{rk}(B) = 2n-|A|-|B|$. On the other hand, we have
    $$\operatorname{rk}(A\wedge B) + \operatorname{rk}(A\vee B) = \begin{cases} 2n-|A\cup B|-|A\cap B| = 2n-|A|-|B| & \mbox{ if } A\vee B\neq\varnothing ;\\ 2n-1 - |A\cup B| = 2n-1-|A|-|B|+|A\cap B| & \mbox{ otherwise.}\end{cases}$$
    The inequality follows from the fact that $A \vee B = \emptyset$ if and only if $|A \cap B| \leq 1$.
\end{enumerate}
The non-singleton boolean poset is therefore a geometric lattice, and hence is Cohen--Macaulay by \cite{folkman}. Moreover, the Möbius number of $\mathrm{NS}(n)$ equals $\sum_{k=2}^n (-1)^{n-k} \binom{n}{k} = (-1)^{n-1}(n-1)$, and the second claim follows.
\end{proof}

\begin{proof}[Proof of Theorem \ref{thm: metablie}]
We consider the following commutative diagram, where the isomorphism $\Phi$ in the first row is that of Proposition \ref{prop: partitions and Lie}.
$$\diagram{
\Lambda^{-1}\mathrm{Lie} \ar[r]^-{\Phi} \ar@{->>}[d] & h^\bullet(\Pi) \ar[d] \\
\Lambda^{-1}\mathrm{MetabLie} \ar@{..>}[r]^-{\overline{\Phi}} & h^\bullet(\mathrm{NS})
}$$ 
The dotted arrow $\overline{\Phi}$ is (uniquely) well-defined since $\Phi([[1,2],[3,4]])$ is zero in $h^3(\mathrm{NS}(4))$ by definition, the partition $\{\{1,2\},\{3,4\}\}$ not being an element of $\mathrm{NS}(4)$. The morphism $h^{n-1}(\Pi(n))\to h^{n-1}(\mathrm{NS}(n))$ is surjective for all $n$ because every chain of maximal length in $\mathrm{NS}(n)$ is the image of itself by that map. Therefore, $\overline{\Phi}$ is surjective. By the work of Bahturin \cite[\S 4.7.1]{bahturinbook}\footnote{We thank Vladimir Dotsenko for providing this reference.}, the space $\mathrm{MetabLie}(n)$ is a free $\KK$-module of rank $n-1$ for all $n$, and Proposition \ref{prop: rank of cohomology NS} implies that $\overline{\Phi}$ is an isomorphism.
\end{proof}
    
        \subsection{Non-crossing 2-partitions, a.k.a.~parking functions}\label{sec:NC2P}

We recall from \cite{Edelman} the definition of non-crossing 2-partitions. (As explained in \cite{DOJVR}, they are in bijection with parking functions.)

First, let us recall that a \emph{non-crossing partition} of $\{1, \ldots, n\}$ is a partition of $\{1, \ldots, n\}$ such that for any $i<j<k<l$, if $i$ is in the same block as $k$ and $j$ is in the same block as $l$, then the four elements are in the same block. Graphically, one can associate to any partition $\pi$ a diagram as follows:  first align elements $\{1, \ldots, n\}$ from left to right and draw an arc between $i$ and $j$ if and only if $i$ and $j$ are in the same block of the partition and no $k$ between $i$ and $j$ is. Then a partition is non-crossing if and only if no arcs cross.

\begin{defi} A \emph{non-crossing 2-partition} of a finite set $S$ is a triple $(\pi,\rho,f)$ where
     \begin{enumerate}[$\triangleright$]
         \item $\pi$ is a partition of $S$,
         \item $\rho$ is a non-crossing partition of $\{1,\ldots,|S|\}$,
         \item $f:\pi\stackrel{\sim}{\to}\rho$ is a bijection between the blocks of $\pi$ and the blocks of $\rho$ which respects cardinalities, i.e. such that for all $T\in \pi$, $|f(T)|=|T|$.
     \end{enumerate}
 \end{defi}

 For instance, there are $3$ non-crossing 2-partitions of the set $S=\{1,2\}$, namely, one with one block, for which $\pi=\rho=\{\{1,2\}\}$ have one element, and two with two blocks, for which $\pi=\rho=\{\{1\},\{2\}\}$ consist of two elements and $f$ is either of the two permutations of the set $\{\{1\},\{2\}\}$. We picture them as follows:
 \begin{center}
 \begin{tikzpicture}[scale=0.6, anchor=base, baseline] 
   \tikzstyle{ver} = [circle, draw, fill, inner sep=0.5mm]
   \tikzstyle{edg} = [line width=0.6mm]
   \node[ver] at (1,0) {};
   \node[ver] at (2,0) {};
   \node      at (1,-0.6) {1};
   \node      at (2,-0.6) {2};
   \draw[edg] (1,0) to[bend left=90, looseness=2] (2,0);
 \end{tikzpicture}
 \hspace{1cm}
     \begin{tikzpicture}[scale=0.6, anchor=base, baseline]   
   \tikzstyle{ver} = [circle, draw, fill, inner sep=0.5mm]
   \tikzstyle{edg} = [line width=0.6mm]
   \node[ver] at (1,0) {};
   \node[ver] at (2,0) {};
   \node      at (1,-0.6) {1};
   \node      at (2,-0.6) {2};
 \end{tikzpicture}
 \hspace{1cm}
 \begin{tikzpicture}[scale=0.6, anchor=base, baseline] 
   \tikzstyle{ver} = [circle, draw, fill, inner sep=0.5mm]
   \tikzstyle{edg} = [line width=0.6mm]
   \node[ver] at (1,0) {};
   \node[ver] at (2,0) {};
   \node      at (1,-0.6) {2};
   \node      at (2,-0.6) {1};
 \end{tikzpicture}
 \end{center}
 We denote by $\Pi_2(S)$ the set of non-crossing partitions of $S$. It is endowed with a poset structure, where $(\pi,\rho,f)\leq (\pi',\rho',f')$ if and only if $\pi\leq \pi'$, $\rho\leq \rho'$, and for all blocks $T\in\pi$ and $T'\in\pi$, $T'\subset T\Rightarrow f'(T')\subset f(T)$. The poset $\Pi_2(\{1,2,3\})$ is pictured in Figure \ref{fig:poset2NCP}.

  \begin{figure}[h]
     \centering
     \resizebox{\textwidth}{!}{
   \includegraphics{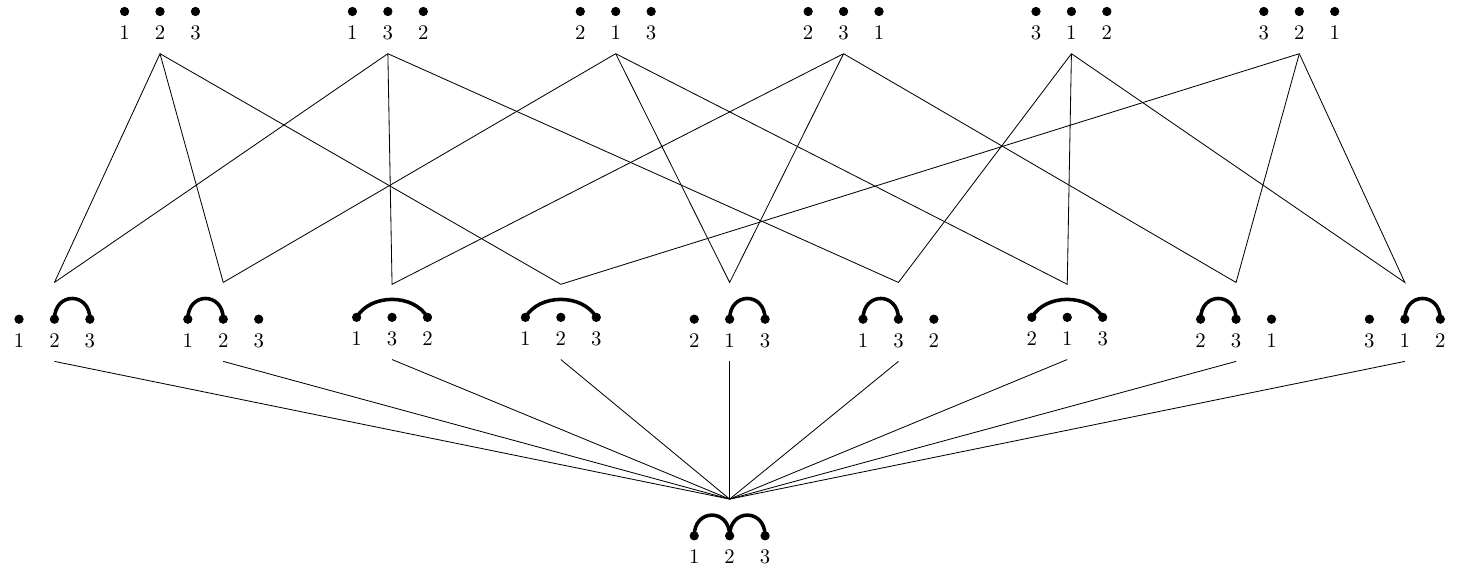}}
 \caption{The Hasse diagram of the poset of non-crossing 2-partitions $\Pi_2(\{1,2,3\})$.}
     \label{fig:poset2NCP}
 \end{figure}

We now endow $\Pi_2:S\mapsto \Pi_2(S)$ with the structure of an operadic poset species. There is an obvious morphism of poset species $a:\Pi_2\to \Pi$ defined by $a(\pi,\rho,f)=\pi$. Let $(\pi,\rho,f)$ be a non-crossing 2-partition of a finite set $S$. There is an obvious (iso)morphism
$$\psi_{(\pi,\rho,f)}:(\Pi_2)_{\geq (\pi,\rho,f)}(S)\To \prod_{T\in \pi}\Pi_2(T) \;\; ,\;\; (\beta,\mu,g)\mapsto \left((\beta_{|T},\mu_{|f(T)},g_{|T})\right)_{T\in \pi},$$
where $g_{|T}:\beta_{|T}\stackrel{\sim}{\to}\mu_{|f(T)}$ is the restriction of $g:\beta\stackrel{\sim}{\to}\mu$. In order to define $\varphi_{(\pi,\rho,f)}$, we first need an operation at the level of non-crossing partitions. For two non-crossing partitions $\mu\leq \rho$ of $\{1,\ldots,n\}$, we can consider the partition $\varphi_\rho(\mu)\in \Pi(\rho)$ and view it as a partition of $\{1,\ldots,|\rho|\}$ via the bijection $\sigma_\rho:\rho\stackrel{\sim}{\to} \{1,\ldots, |\rho|\}$ which orders the blocks of $\rho$ according to their minimal elements, i.e. such that $\sigma_\rho(T_1)<\sigma_\rho(T_2)$ if and only if $\min(T_1)<\min(T_2)$. One easily checks that $\varphi_\rho(\mu)$ is then a non-crossing partition of $\{1,\ldots,|\rho|\}$. We can then define
$$\varphi_{(\pi,\rho,f)}:(\Pi_2)_{\leq (\pi,\rho,f)} \To \Pi_2(\pi) \;\; , \;\; (\alpha,\mu,g)\mapsto (\varphi_\pi(\alpha),\varphi_\rho(\mu), g),$$
where we abuse notation and still denote by $g$ the bijection $\varphi_\pi(\alpha)\stackrel{\sim}{\to}\varphi_\rho(\mu)$ induced by $g:\alpha\stackrel{\sim}{\to}\mu$ and the obvious bijections $\alpha\simeq \varphi_\pi(\alpha)$ and $\mu\simeq \varphi_\rho(\mu)$.
 
\begin{prop}
This structure makes $\Pi_2$ into an operadic poset species.
\end{prop}

\begin{proof}
The structure morphisms $a$ satisfy \eqref{eq: axiom a min max compatible strong}. The commutativity of the diagrams \eqref{eq: commutative squares phi psi a} is obvious, as well as the equivariance and unitality axioms. We are left with proving the associativity axioms. In each case we only need to trace the non-crossing partitions since the diagrams obviously commute at the level of the underlying partitions. Let $(\pi,\rho,f)\leq (\pi',\rho',f')$ be non-crossing 2-partitions of a finite set $S$.
\begin{enumerate}[$\triangleright$]
\item (Composition of $\varphi$'s.) Let $k=|\rho|$, and denote by $\underline{\rho}=\varphi_{\rho'}(\rho)$ the non-crossing partition of $\{1,\ldots,|\rho'|\}$ with $k$ blocks induced by $\rho$. We only need to check that the following diagram of bijections commutes.
$$\diagram{
\rho \ar[rr]^-{\sigma_\rho} \ar[rd] && \{1,\ldots, k\} \\
& \underline{\rho} \ar[ru]_-{\sigma_{\underline{\rho}}}&
}$$
This is clear because for every block $T$ of $\rho$ we have
$$\min(T) = \min\left\{\min(T'), T'\in \rho'_{|T}\right\}.$$
\item (Composition of $\psi$'s.) This is clear.
\item (Composition of $\varphi$'s and $\psi$'s.) This is clear.
\end{enumerate}
\end{proof}

\begin{prop}\label{prop:operad-NC2P}
For all $n\geq 1$, the cohomology $h^\bullet(\Pi_2(n))$ is concentrated in degree $n-1$, and $h^{n-1}(\Pi_2(n))$ is a free $\KK$-module of rank
$$\mathrm{rk}\, h^{n-1}(\Pi_2(n)) = n!C_n,$$
where $C_n$ denotes the $n$th Catalan number. Furthermore, the $\mathfrak{S}_n$-module $h^{n-1}(\Pi_2(n))$ is isomorphic to the direct sum of $C_n$ copies of the regular representation. 
\end{prop}

\begin{proof}
By definition, $h^\bullet(\Pi_2(n))$ is isomorphic to the direct sum of all the shifted reduced cohomology groups of the maximal intervals of $\Pi_2(n)$. Those are copies indexed by permutations of $\{1,\ldots,n\}$ of the poset $\mathrm{NC}(n)$ of non-crossing partitions of $\{1,\ldots,n\}$, from which we see that $h^\bullet(\Pi_2(n))$ is isomorphic as an $\mathfrak{S}_n$-module to $\KK[\mathfrak{S}_n]\otimes_\KK h^\bullet(\mathrm{NC}(n))$ where $h^\bullet(\mathrm{NC}(n))$ is considered here as a trivial $\mathfrak{S}_n$-module. The poset obtained from $\mathrm{NC}(n)$ by removing its least element is Cohen--Macaulay by Bj\"{o}rner and Edelman \cite[Example 2.9]{bjornershellable} and Kreweras proved in \cite[Th\'{e}or\`{e}me 6]{Kreweras} that its M\"{o}bius number equals $C_n$ up to a sign. The claim follows.
\end{proof}

Proposition \ref{prop:operad-NC2P} shows that $h^\bullet(\Pi_2)$ has the same dimension in each arity as the (desuspension of the) magmatic operad $\operatorname{Mag}$ freely generated by an arity $2$ element (with no symmetry and no relation). However, it cannot be isomorphic to $\Lambda^{-1}\mathrm{Mag}$ as an operad. Indeed, by functoriality (see \S\ref{sec: functoriality}), if we let 
$$ 1 \prec 2 \defas  \begin{tikzpicture}[scale=0.6, anchor=base, baseline]
   \tikzstyle{ver} = [circle, draw, fill, inner sep=0.5mm]
   \tikzstyle{edg} = [line width=0.6mm]
   \node[ver] at (1,0) {};
   \node[ver] at (2,0) {};
   \node      at (1,-0.6) {1};
   \node      at (2,-0.6) {2};
   \draw[edg] (1,0) to[bend left=90, looseness=2] (2,0);
 \end{tikzpicture} < \begin{tikzpicture}[scale=0.6, anchor=base, baseline]
   \tikzstyle{ver} = [circle, draw, fill, inner sep=0.5mm]
   \tikzstyle{edg} = [line width=0.6mm]
   \node[ver] at (1,0) {};
   \node[ver] at (2,0) {};
   \node      at (1,-0.6) {1};
   \node      at (2,-0.6) {2};
 \end{tikzpicture} \; \in h^1(\Pi_2(2)),$$
 then
the sum $1\prec 2+2\prec 1 = a^*(\{1,2\}<\{\{1\},\{2\}\})$ has to satisfy the (desuspension of the) Jacobi identity. In fact, the next proposition shows that $1\prec 2$ even satisfies the pre-Lie identity. Since the space $\mathrm{PreLie}(3)$ has rank $9$, which is smaller than $\mathrm{rk}\,h^2(\Pi_2(3))=30$, this proposition implies that the operad $h^\bullet(\Pi_2)$ is not generated in arity $2$.
        
        \begin{prop}\label{prop: prelie in Pi 2}
        We have the equality in $h^2(\Pi_2(3))$:
        $$(1\prec 2)\prec 3+1\prec (2\prec 3) + (1\prec 3)\prec 2 + 1\prec (3\prec2)=0.$$
        In particular, the morphism of graded operads $a^*:\Lambda^{-1}\mathrm{Lie}\to h^\bullet(\Pi_2)$ factors through $\Lambda^{-1}\mathrm{PreLie}$.
        \end{prop} 

\begin{proof}
      
    Let us compute the operation 
    $$ (1\prec 3) \prec 2 = \left(\begin{tikzpicture}[scale=0.6, anchor=base, baseline]
   \tikzstyle{ver} = [circle, draw, fill, inner sep=0.5mm]
   \tikzstyle{edg} = [line width=0.6mm]
   \node[ver] at (1,0) {};
   \node[ver] at (2,0) {};
   \node      at (1,-0.6) {$*$};
   \node      at (2,-0.6) {2};
   \draw[edg] (1,0) to[bend left=90, looseness=2] (2,0);
 \end{tikzpicture} < \begin{tikzpicture}[scale=0.6, anchor=base, baseline]
   \tikzstyle{ver} = [circle, draw, fill, inner sep=0.5mm]
   \tikzstyle{edg} = [line width=0.6mm]
   \node[ver] at (1,0) {};
   \node[ver] at (2,0) {};
   \node      at (1,-0.6) {$*$};
   \node      at (2,-0.6) {2};
 \end{tikzpicture} \right) \circ_* \left(\begin{tikzpicture}[scale=0.6, anchor=base, baseline]
   \tikzstyle{ver} = [circle, draw, fill, inner sep=0.5mm]
   \tikzstyle{edg} = [line width=0.6mm]
   \node[ver] at (1,0) {};
   \node[ver] at (2,0) {};
   \node      at (1,-0.6) {1};
   \node      at (2,-0.6) {3};
   \draw[edg] (1,0) to[bend left=90, looseness=2] (2,0);
 \end{tikzpicture} < \begin{tikzpicture}[scale=0.6, anchor=base, baseline]
   \tikzstyle{ver} = [circle, draw, fill, inner sep=0.5mm]
   \tikzstyle{edg} = [line width=0.6mm]
   \node[ver] at (1,0) {};
   \node[ver] at (2,0) {};
   \node      at (1,-0.6) {1};
   \node      at (2,-0.6) {3};
 \end{tikzpicture} \right).$$ 
 This operadic composition has underlying partition $\pi = \{\{1,3\},\{2\}\}$, and is by definition a sum over all elements $x$ in $\Pi_2(\{1,2,3\})$ such that $a(x)=\pi$. There are 3 such $x$, namely:
 \begin{center} \utu \hspace{1cm} \dud \hspace{1cm} \udu\end{center} 
 Let us now compute:
     \begin{align*}
         \varphi^*_{\utu} \left(\begin{tikzpicture}[scale=0.6, anchor=base, baseline]
   \tikzstyle{ver} = [circle, draw, fill, inner sep=0.5mm]
   \tikzstyle{edg} = [line width=0.6mm]
   \node[ver] at (1,0) {};
   \node[ver] at (2,0) {};
   \node      at (1,-0.6) {$*$};
   \node      at (2,-0.6) {2};
   \draw[edg] (1,0) to[bend left=90, looseness=2] (2,0);
 \end{tikzpicture} < \begin{tikzpicture}[scale=0.6, anchor=base, baseline]
   \tikzstyle{ver} = [circle, draw, fill, inner sep=0.5mm]
   \tikzstyle{edg} = [line width=0.6mm]
   \node[ver] at (1,0) {};
   \node[ver] at (2,0) {};
   \node      at (1,-0.6) {$*$};
   \node      at (2,-0.6) {2};
 \end{tikzpicture} \right)  & = \uuu < \utu \\
       \varphi^*_{\dud} \left(\begin{tikzpicture}[scale=0.6, anchor=base, baseline]
   \tikzstyle{ver} = [circle, draw, fill, inner sep=0.5mm]
   \tikzstyle{edg} = [line width=0.6mm]
   \node[ver] at (1,0) {};
   \node[ver] at (2,0) {};
   \node      at (1,-0.6) {$*$};
   \node      at (2,-0.6) {2};
   \draw[edg] (1,0) to[bend left=90, looseness=2] (2,0);
 \end{tikzpicture} < \begin{tikzpicture}[scale=0.6, anchor=base, baseline]
   \tikzstyle{ver} = [circle, draw, fill, inner sep=0.5mm]
   \tikzstyle{edg} = [line width=0.6mm]
   \node[ver] at (1,0) {};
   \node[ver] at (2,0) {};
   \node      at (1,-0.6) {$*$};
   \node      at (2,-0.6) {2};
 \end{tikzpicture} \right)  & = 0 \\
       \varphi^*_{\udu} \left(\begin{tikzpicture}[scale=0.6, anchor=base, baseline]
   \tikzstyle{ver} = [circle, draw, fill, inner sep=0.5mm]
   \tikzstyle{edg} = [line width=0.6mm]
   \node[ver] at (1,0) {};
   \node[ver] at (2,0) {};
   \node      at (1,-0.6) {$*$};
   \node      at (2,-0.6) {2};
   \draw[edg] (1,0) to[bend left=90, looseness=2] (2,0);
 \end{tikzpicture} < \begin{tikzpicture}[scale=0.6, anchor=base, baseline]
   \tikzstyle{ver} = [circle, draw, fill, inner sep=0.5mm]
   \tikzstyle{edg} = [line width=0.6mm]
   \node[ver] at (1,0) {};
   \node[ver] at (2,0) {};
   \node      at (1,-0.6) {$*$};
   \node      at (2,-0.6) {2};
 \end{tikzpicture} \right)  & = \uuu < \udu.
     \end{align*}
     On the other hand, we have: 
    \begin{align*}
         \psi^*_{\utu} \left(\begin{tikzpicture}[scale=0.6, anchor=base, baseline]
   \tikzstyle{ver} = [circle, draw, fill, inner sep=0.5mm]
   \tikzstyle{edg} = [line width=0.6mm]
   \node[ver] at (1,0) {};
   \node[ver] at (2,0) {};
   \node      at (1,-0.6) {1};
   \node      at (2,-0.6) {3};
   \draw[edg] (1,0) to[bend left=90, looseness=2] (2,0);
 \end{tikzpicture} < \begin{tikzpicture}[scale=0.6, anchor=base, baseline]
   \tikzstyle{ver} = [circle, draw, fill, inner sep=0.5mm]
   \tikzstyle{edg} = [line width=0.6mm]
   \node[ver] at (1,0) {};
   \node[ver] at (2,0) {};
   \node      at (1,-0.6) {1};
   \node      at (2,-0.6) {3};
 \end{tikzpicture} \otimes \begin{tikzpicture}[scale=0.6, anchor=base, baseline]
   \tikzstyle{ver} = [circle, draw, fill, inner sep=0.5mm]
   \tikzstyle{edg} = [line width=0.6mm]
   \node[ver] at (1,0) {};
   \node      at (1,-0.6) {2};
 \end{tikzpicture} \right)  & = \utu < \utd \\
       \psi^*_{\dud} \left(\begin{tikzpicture}[scale=0.6, anchor=base, baseline]
   \tikzstyle{ver} = [circle, draw, fill, inner sep=0.5mm]
   \tikzstyle{edg} = [line width=0.6mm]
   \node[ver] at (1,0) {};
   \node[ver] at (2,0) {};
   \node      at (1,-0.6) {1};
   \node      at (2,-0.6) {3};
   \draw[edg] (1,0) to[bend left=90, looseness=2] (2,0);
 \end{tikzpicture} < \begin{tikzpicture}[scale=0.6, anchor=base, baseline]
   \tikzstyle{ver} = [circle, draw, fill, inner sep=0.5mm]
   \tikzstyle{edg} = [line width=0.6mm]
   \node[ver] at (1,0) {};
   \node[ver] at (2,0) {};
   \node      at (1,-0.6) {1};
   \node      at (2,-0.6) {3};
 \end{tikzpicture} \otimes \begin{tikzpicture}[scale=0.6, anchor=base, baseline]
   \tikzstyle{ver} = [circle, draw, fill, inner sep=0.5mm]
   \tikzstyle{edg} = [line width=0.6mm]
   \node[ver] at (1,0) {};
   \node      at (1,-0.6) {2};
 \end{tikzpicture} \right)  & =  \dud < \dut\\
       \psi^*_{\udu} \left(\begin{tikzpicture}[scale=0.6, anchor=base, baseline]
   \tikzstyle{ver} = [circle, draw, fill, inner sep=0.5mm]
   \tikzstyle{edg} = [line width=0.6mm]
   \node[ver] at (1,0) {};
   \node[ver] at (2,0) {};
   \node      at (1,-0.6) {1};
   \node      at (2,-0.6) {3};
   \draw[edg] (1,0) to[bend left=90, looseness=2] (2,0);
 \end{tikzpicture} < \begin{tikzpicture}[scale=0.6, anchor=base, baseline]
   \tikzstyle{ver} = [circle, draw, fill, inner sep=0.5mm]
   \tikzstyle{edg} = [line width=0.6mm]
   \node[ver] at (1,0) {};
   \node[ver] at (2,0) {};
   \node      at (1,-0.6) {1};
   \node      at (2,-0.6) {3};
 \end{tikzpicture} \otimes \begin{tikzpicture}[scale=0.6, anchor=base, baseline]
   \tikzstyle{ver} = [circle, draw, fill, inner sep=0.5mm]
   \tikzstyle{edg} = [line width=0.6mm]
   \node[ver] at (1,0) {};
   \node      at (1,-0.6) {2};
 \end{tikzpicture} \right)  & = \udu < \udt.
     \end{align*} 
   The operadic composition is then given by: 
    \begin{equation*}
        (1\prec 3) \prec 2 =  \uuu<\utu<\utd + \uuu<\udu<\udt .
   \end{equation*}     
The same kind of computation gives: 
\begin{equation*}
   1 \prec (3 \prec 2) = \uuu<\udd<\utd. 
\end{equation*}
Using the relation 
\begin{align*}
    d\left( \uuu < \utd \right) & =  \uuu< \utu< \utd \\& +  \uuu< \uud< \utd  \\
    &+ \uuu< \udd< \utd,
\end{align*}
we arrive the expression
\begin{align*}
    (1 \prec 3) \prec 2 + 1 \prec (3 \prec 2) & = \uuu<\udu<\udt \\ & - \uuu< \uud< \utd, 
\end{align*}
which is manifestly antisymmetric in the entries $2$, $3$. The claim follows.
\end{proof}

We now turn to the left operadic module $\hbottom{\bullet}(\Pi_2)$ (note that the right operadic module $\htop{\bullet}(\Pi_2)$ is trivial because every poset $\Pi_2(S)$ has a least element).

\begin{prop}[\cite{DOJVR}, Theorem 4.2]\label{prop:DOJVR}
For every $n\geq 1$, the cohomology $\hbottom{\bullet}(\Pi_2)$ is concentrated in degree $n-1$, and $\hbottom{n-1}(\Pi_2(n))$ is a free $\KK$-module of rank
$$\mathrm{rk}\,\hbottom{n-1}(\Pi_2(n)) = (n-1)^{n-1}.$$
Furthermore, the character for the action of the symmetric group on $\hbottom{n-1}(\Pi_2(n))$  is given by 
           \begin{equation}
           \sigma \mapsto (-1)^{n-z(\sigma)} (n-1)^{z(\sigma)-1},
           \end{equation}
           where $z(\sigma)$ is the number of cycles of the permutation $\sigma$. 
\end{prop}

Note that this character coincides with that of the action of the symmetric group on prime parking functions (see \cite{ArmstrongLoehrWarrington}). We leave it to the motivated reader to give a description of the (degree zero) operad $\Lambda h^\bullet(\Pi_2)$ and its left operadic module $\Lambda\hbottom{\bullet}(\Pi_2)$. A starting point would be the EL-labelings of Björner and Edelman \cite{Edelman} and Stanley \cite{stanleyparking}, which give rise to two different bases of $h^{n-1}(\mathrm{NC}(n))$ indexed by Catalan objets.

\subsection{Multilabeled trees, Kontsevich's operad of trees, and pre-Lie algebras}\label{sec: MLT}

Here, a tree is an abstract graph (without a planar embedding), without double edges or self-loops, in which two vertices are connected by exactly one path.

\subsubsection{The operadic poset species of multilabeled trees}

\begin{defi}
Let $S$ be a finite set. An $S$-multilabeled tree is a tree whose set of nodes is a partition of $S$.
\end{defi}

In other words, an $S$-multilabeled tree is a tree in which each node is labeled by a subset of $S$ and such that the labelings form a partition of $S$. We denote by $\mathrm{MLT}(S)$ the set of $S$-multilabeled trees. It is a poset for which $t\leq t'$ if $t$ is obtained from $t'$ by contracting edges and merging the corresponding blocks of the partitions. The Hasse diagram of the poset $\mathrm{MLT}(3)$ is shown in Figure \ref{fig:MLT3}.

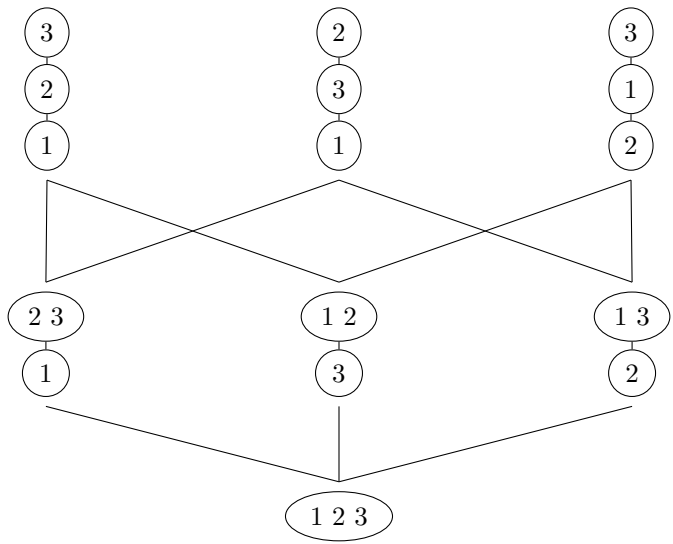
\begin{figure}[h!]
    \centering
\begin{tikzpicture}
    \node (min) at (0,0) {\begin{tikzpicture}
\node[draw, ellipse] {$1 \ 2 \ 3$};
    \end{tikzpicture}};

\node[above=1cm of min] (3-12) {\includegraphics{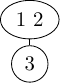}};
\node[left=2.6cm of 3-12] (1-23) {\includegraphics{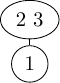}};
\node[right=2.6cm of 3-12](2-13) {\includegraphics{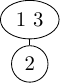}};

\node[above = 4cm of min] (132) {\includegraphics{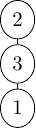}};
\node[left=3cm of 132] (123) {\includegraphics{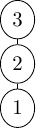}};
\node[right=3cm of 132] (213) {\includegraphics{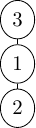}};

\draw (min.north) -- (3-12.south);
\draw (min.north) -- (1-23.south);
\draw (min.north) -- (2-13.south);

\draw (1-23.north) -- (123.south);
\draw (1-23.north) -- (132.south);
\draw (3-12.north) -- (123.south);
\draw (3-12.north) -- (213.south);
\draw (2-13.north) -- (213.south);
\draw (2-13.north) -- (132.south);
\end{tikzpicture}

    \caption{The poset $\mathrm{MLT}(3)$ of multilabeled trees on $\{1,2,3\}$.}
    \label{fig:MLT3}
\end{figure}

We now endow $\mathrm{MLT}:S\mapsto \mathrm{MLT}(S)$ with the structure of an operadic poset species. There is an obvious morphism of poset species $a:\mathrm{MLT}\to \Pi$ which forgets the tree structure. Let $t\in \mathrm{MLT}(S)$ and set $\pi=a(t)$. The morphism of posets
$$\varphi_t:\mathrm{MLT}_{\leq t}(S)\To \mathrm{MLT}(\pi)$$
relabels a multilabeled tree $\leq t$ by replacing a label $T$ by the set $\pi_{|T}$. The morphism of posets
$$\psi_t:\mathrm{MLT}_{\geq t}(S) \To \prod_{T\in \pi}\mathrm{MLT}(T)$$
sends a multilabeled tree $\geq t$ to the forest obtained by cutting the edges of $t$. One easily checks that the axioms of an operadic poset species are satisfied.

\subsubsection{Kontsevich's operad of trees}

We define an operad $\mathrm{Tree}$, that we call \emph{Kontsevich's operad of trees} because it is the supoperad of Kontsevich's operad of graphs $\mathrm{Gra}$ \cite{kontsevichoperadsmotives, kontsevichderived} spanned by trees. (With a little bit of effort, it would presumably be possible to fit the full operad of graphs in our formalism, but we will not try to do so here.)

\begin{defi}
Let $S$ be a finite set. An $S$-labeled tree on $S$ is a tree whose set of nodes is $S$.
\end{defi}

By definition, the $\KK$-module $\mathrm{Tree}(S)$ is spanned by pairs $(t,\omega)$ where $t$ is an $S$-labeled tree and $\omega$ is a linear order on the set $E(t)$ of edges of $t$, modulo the relation $(t,\omega)=\mathrm{sgn}(\sigma)(t,\sigma\cdot\omega)$ for every permutation $\sigma$ of $E(t)$. We declare that $\mathrm{Tree}(S)$ is concentrated in degree $|S|-1$, which is the number of edges of an $S$-labeled tree. The operadic composition morphism
$$\circ_*:\mathrm{Tree}(A\sqcup\{*\})\otimes \mathrm{Tree}(B)\To \mathrm{Tree}(A\sqcup B)$$
is defined by the formula
$$(t,\omega)\circ_* (t',\omega') = \sum (t'',\omega'')$$
where the sum ranges over  the $A\sqcup B$-labeled trees which have $t'$ as a $B$-labeled subtree and such that contracting $t'$ into a node labeled $*$ yields $t$; the ordering $\omega''$ on the edges of $t''$ is obtained by concatenating the orderings $\omega$ and $\omega'$. Put differently, in order to compute $(t,\omega)\circ_* (t',\omega')$ one disconnects $*$ in $t$, replaces it with $t'$, and reconnects the loose ends of $t$ to $t'$ in all possible ways, summing over all the possibilities. For instance, in arity $3$,
\begin{equation}\label{eq: composition tree}
\begin{tikzpicture}[grow=up, scale=0.5, level distance=2cm, baseline={([yshift=-.5ex]current bounding box.center)}]
\node[draw,ellipse] {$1$}
    child{node[draw, ellipse]{$*$}
    };
    \end{tikzpicture} 
\,\circ_*\,
\begin{tikzpicture}[grow=up, scale=0.5, level distance=2cm, baseline={([yshift=-.5ex]current bounding box.center)}]
\node[draw,ellipse] {$2$}
    child{node[draw, ellipse]{$3$}
    };
    \end{tikzpicture} 
\;=\;
\begin{tikzpicture}[grow=up, scale=0.5, level distance=2cm, baseline={([yshift=-.5ex]current bounding box.center)}]
\node[draw,ellipse] {$1$}
    child{node[draw, ellipse]{$2$}
        child{node[draw, ellipse]{$3$} edge from parent node[left,draw=none] {\tiny{II}}
        } edge from parent node[left,draw=none] {\tiny{I}}
    };
    \end{tikzpicture} 
\;+\,
\begin{tikzpicture}[grow=up, scale=0.5, level distance=2cm, baseline={([yshift=-.5ex]current bounding box.center)}]
\node[draw,ellipse] {$1$}
    child{node[draw, ellipse]{$3$}
        child{node[draw, ellipse]{$2$} edge from parent node[left,draw=none] {\tiny{II}}
        } edge from parent node[left,draw=none] {\tiny{I}}
    };
    \end{tikzpicture},
\end{equation}
where the labels I and II indicate the linear order on the edges.

\begin{prop}
There is an isomorphism of operads:
\begin{equation}\label{eq: MLT Tree}
h^\bullet(\mathrm{MLT})\simeq \mathrm{Tree}.
\end{equation}
\end{prop}

\begin{proof}
The maximal intervals of $\mathrm{MLT}(S)$ are in bijection with $S$-labeled trees, and the maximal interval corresponding to an $S$-labeled tree is isomorphic to the boolean poset on the set $E(t)$ of edges of $t$. Therefore we have that $h^\bullet(\mathrm{MLT}(S))$ is concentrated in degree $|S|-1$, with
$$h^{|S|-1}(\mathrm{MLT}(S)) \simeq \bigoplus_t \det(E(t)),$$
where the sum ranges over the $S$-labeled trees. This gives rise to an isomorphism of graded species $h^\bullet(\mathrm{MLT})\simeq \mathrm{Tree}$. Concretely, a generator $(t,e_1,\ldots, e_{n-1})$ of $\mathrm{Tree}(S)$, for $|S|=n$, corresponds to the (co)chain
$$t/(e_1,\ldots, e_{n-1}) < t/(e_2,\ldots, e_{n-1})< \cdots < t/e_{n-1} < t$$
in $\mathrm{MLT}(S)$. One then easily checks that we do get an isomorphism of operads.
\end{proof}

We note that by functoriality (see \S\ref{sec: functoriality}) we get the well-known morphism of operads
$$\Lambda^{-1}\mathrm{Lie}\To \mathrm{Tree} \;\; ,\;\; [1,2] \; \mapsto \; \begin{tikzpicture}[grow=up, scale=0.5, level distance=2cm, baseline={([yshift=-.5ex]current bounding box.center)}]
\node[draw,ellipse] {$1$}
    child{node[draw, ellipse]{$2$}
    };
    \end{tikzpicture}.$$

\begin{rem}
Contrary to its rooted version (see below), the operad $\mathrm{Tree}$ is not generated in arity $2$, because $\dim\mathrm{Lie}(3)=2<\dim \mathrm{Tree}(3)=3$.
\end{rem}

Our formalism produces a non-trivial left module $\hbottom{\bullet}(\mathrm{MLT})$ over the operad $\mathrm{Tree}$, which is the main focus of \cite{DDOJV}.

\subsubsection{The rooted variant, and pre-Lie algebras}\label{sec: MLRT}

\begin{defi}
Let $S$ be a finite set. An $S$-multilabeled rooted tree is an $S$-multilabeled tree equipped with a distinguished node called its root.
\end{defi}

The set $\mathrm{MLRT}(S)$ of $S$-multilabeled rooted trees is a poset as in the non-rooted case. The Hasse diagram of the poset $\mathrm{MLRT}(3)$ is shown on Figure \ref{fig:MLRT3}. The posets of multilabeled rooted trees originated from the work of Bishal Deb, Bérénice Delcroix-Oger and Matthieu Josuat-Vergès and will be further studied in \cite{DDOJV}.

\begin{figure}[h!]
    \centering
    \resizebox{\textwidth}{!}{
\begin{tikzpicture}
    \node (min) at (0,0) {\begin{tikzpicture}
\node[draw, rectangle] {$1 \ 2 \ 3$};
    \end{tikzpicture}};
\coordinate[above=2cm of min] (c1);
\coordinate[above=4cm of c1] (c2);
\node[left=0.8cm of c1] (2-13) {\includegraphics{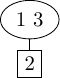}};
\node[left=1.6cm of 2-13](23-1) {\includegraphics{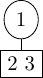}};
\node[left=1.6cm of 23-1] (1-23) {\includegraphics{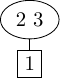}};
\node[right=0.8cm of c1] (12-3) {\includegraphics{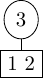}};
\node[right=1.6cm of 12-3] (13-2) {\includegraphics{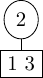}};
\node[right=1.6cm of 13-2](3-12) {\includegraphics{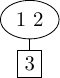}}; 
\node at (c2) (cor2) {\includegraphics{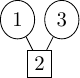}};
\node[left=1cm of c2] (cor1) {\includegraphics{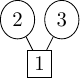}};
\node[left=1cm of cor1](231) {\includegraphics{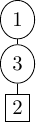}};
\node[left=1cm of 231] (132) {\includegraphics{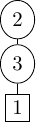}};
\node[left=1cm of 132] (123) {\includegraphics{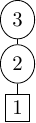}};
\node[right=1cm of c2] (cor3) {\includegraphics{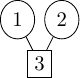}};
\node[right=1cm of cor3] (213) {\includegraphics{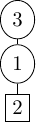}};
\node[right=1cm of 213] (312) {\includegraphics{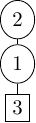}};
\node[right=1cm of 312] (321) {\includegraphics{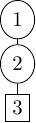}};
\draw (min.north)--(1-23.south);
\draw (123.south)--(1-23.north)--(132.south);
\draw (min.north)--(23-1.south);
\draw (321.south)--(23-1.north)--(231.south);
\draw (min.north)--(2-13.south);
\draw (213.south)--(2-13.north)--(231.south);
\draw (min.north)--(13-2.south);
\draw (132.south)--(13-2.north)--(312.south);
\draw (min.north)--(3-12.south);
\draw (321.south)--(3-12.north)--(312.south);
\draw (min.north)--(12-3.south);
\draw (123.south)--(12-3.north)--(213.south);
\draw (12-3.north)--(cor1.south)--(13-2.north);
\draw (12-3.north)--(cor2.south)--(23-1.north);
\draw (13-2.north)--(cor3.south)--(23-1.north);
\end{tikzpicture}
}
    \caption{The poset $\mathrm{MLRT}(3)$ of multilabeled rooted trees on $\{1,2,3\}$. The root is the node in a rectangle.}
    \label{fig:MLRT3}
\end{figure}
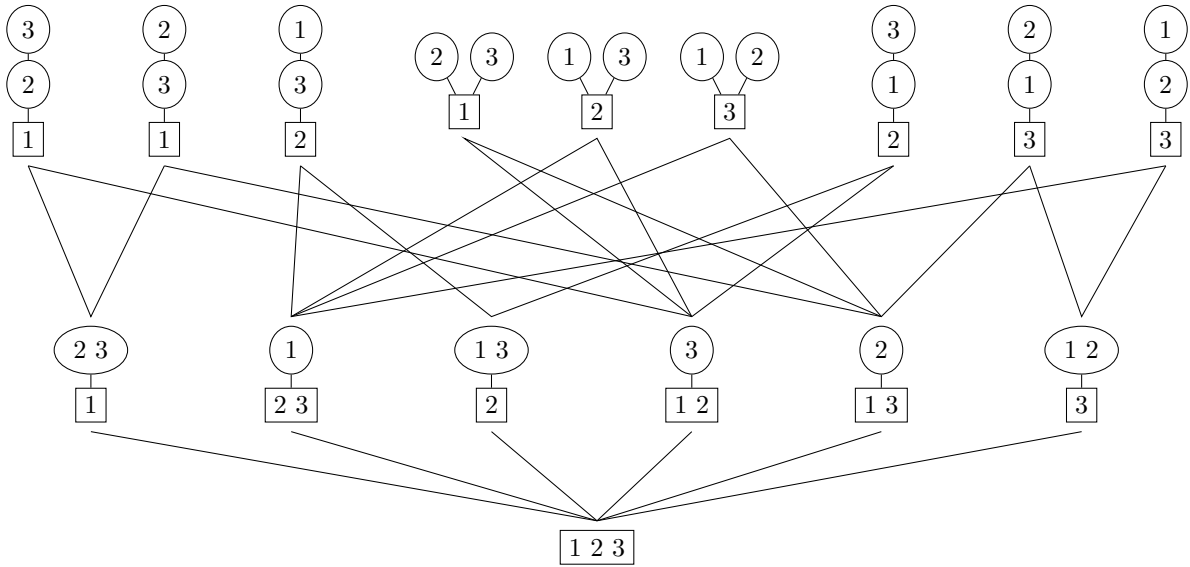

As in the non-rooted case, there is a natural operadic poset species structure on $\mathrm{MLRT}:S\mapsto \mathrm{MLRT}(S)$. Its cohomology is the rooted version of Kontsevich's operad of trees, which we denote by $\mathrm{RTree}$, hence an isomorphism of graded operads
\begin{equation}\label{eq: MLRT RTree}
h^\bullet(\mathrm{MLRT})\simeq \mathrm{RTree}\simeq \Lambda^{-1}\mathrm{PreLie},
\end{equation}
where the second isomorphism follows from \cite{ChapotonLivernet}. An example of operadic composition in $\mathrm{RTree}$ is as follows (compare with \eqref{eq: composition tree}):
$$ 
\begin{tikzpicture}[grow=up, scale=0.5, level distance=2cm, baseline={([yshift=-.5ex]current bounding box.center)}]
\node[draw,rectangle] {$1$}
    child{node[draw, ellipse]{$*$}
    };
    \end{tikzpicture} 
\,\circ_*\,
\begin{tikzpicture}[grow=up, scale=0.5, level distance=2cm, baseline={([yshift=-.5ex]current bounding box.center)}]
\node[draw,rectangle] {$2$}
    child{node[draw, ellipse]{$3$}
    };
    \end{tikzpicture} 
\;=\;
\begin{tikzpicture}[grow=up, scale=0.5, level distance=2cm, baseline={([yshift=-.5ex]current bounding box.center)}]
\node[draw,rectangle] {$1$}
    child{node[draw, ellipse]{$2$}
        child{node[draw, ellipse]{$3$} edge from parent node[left,draw=none] {\tiny{II}}
        } edge from parent node[left,draw=none] {\tiny{I}}
    };
    \end{tikzpicture}.
$$
Thanks to \eqref{eq: MLRT RTree}, our formalism therefore produces a non-trivial left pre-Lie module $\Lambda \hbottom{\bullet}(\mathrm{MLRT})$, which is the main focus of \cite{DDOJV}.

\begin{rem}
There is a natural morphism of operadic poset species (see \S\ref{sec: functoriality})
$$\mathrm{MLRT}\To \mathrm{MLT}$$
which forgets the root of a multilabeled rooted tree. It leads to a pullback morphism in cohomology, which thanks to \eqref{eq: MLT Tree} and \eqref{eq: MLRT RTree} gives rise to a morphism of graded operads:
$$\mathrm{Tree}\To \mathrm{RTree}\simeq \Lambda^{-1}\mathrm{PreLie}.$$
This morphism sends a labeled tree to the sum of all the corresponding labeled \emph{rooted} trees.
\end{rem}

\section{Conclusion and open questions}\label{sec: open questions}

In this article we have introduced and studied a formalism which produces, from a family $P$ of posets with the structure of an ``operadic poset species'', a graded operad $h^\bullet(P)$ along with left and right operadic modules on (variants of) poset cohomology, denoted by $\hbottom{\bullet}(P)$ and $\htop{\bullet}(P)$ respectively.

We collect in this final section some open questions that arose in the study of certain classes of examples. Most of those concern (left and right) decorated partitions and the corresponding (left and right) operadic modules, studied in \S\ref{sec: decorated partitions modules}, expanding on the data of Tables \ref{tab:checkmu} and \ref{tab:hatmu}. We work with the ring of coefficients $\KK=\QQ$ throughout.

\subsection{Right-$\mathrm{As}$-decorated partitions}

The operad $\mathrm{As}$, which encodes associative algebras, is both left- and right-basic. We refer to \S\ref{subsubsec: left As} for a study of left-$\mathrm{As}$-decorated partitions, and focus now on right-$\mathrm{As}$-decorated partitions.

Recall from \S\ref{sec: As right operadic module} that $\Pi^{\mathrm{As}}(n)$ denotes the poset of partitions equipped with a total order on each block (see the Hasse diagram for $n=3$ in Figure \ref{figComp2}). Proposition \ref{prop: right koszul} and the fact that the operad $\mathrm{As}$ is Koszul yield an isomorphism of operads
$$\Lambda h^\bullet(\Pi^{\mathrm{As}}) \simeq \mathrm{As},$$
which had already appeared in \cite{vallettepartitionposets} at the level of linear species.
On the variant $\Lambda\htop{\bullet}(\Pi^{\mathrm{As}})$ of poset cohomology, we therefore get a right operadic module over $\mathrm{As}$, which seems more difficult to compute. Indeed, the poset $\Pi^{\mathrm{As}}_+(n) := \Pi^{\mathrm{As}}(n)\cup\{\hat{0}\}$ is not always Cohen--Macaulay, as Table \ref{tab:cohomSage} shows, and therefore $\Lambda \htop{\bullet}(\Pi^{\mathrm{As}})$ is not concentrated in degree zero. The corresponding Euler characteristics are given by \eqref{eq: generating series Pi As}, while Proposition \ref{prop: cycle index series right operadic module} yields the $\mathfrak{S}_n$-equivariant Euler characterics:
\begin{equation}\label{eq: cycle index series for Pi As}
\sum_{k\geq 1}(-1)^k \mathrm{Z}_{\htop{k}(\Pi^{\mathrm{As}})} = \left(\exp\left(\sum_{n\geq 1}\frac{p_n}{n}\right)-1\right)\circ \frac{p_1}{1+p_1} = \exp\left(\sum_{n\geq 1} \frac{1}{n}\frac{p_n}{1+p_n}\right) -1 .
\end{equation}

\begin{tcolorbox}[colback=green!5!white,colframe=green!50!black,title=Open questions]
  \begin{itemize}
    \item Describe the cohomology groups $\htop{\bullet}(\Pi^{\mathrm{As}}(n))\simeq \widetilde{H}^{\bullet-1}(\Pi^{\mathrm{As}}_+(n))$, e.g. using combinatorial bases. The $\mathfrak{S}$-equivariant Euler characteristic is given by \eqref{eq: cycle index series for Pi As}.
      \item Describe the (graded) right $\mathrm{As}$-module $\Lambda \htop{\bullet}(\Pi^{\mathrm{As}})$.
  \end{itemize}
\end{tcolorbox}

\subsection{Right-$\mathrm{Perm}$-decorated partitions}
The operad $\mathrm{Perm}$ encoding permutative algebras, introduced by Chapoton \cite{chapotonendofoncteur}, is right-basic but not left-basic (\S\ref{sec: examples right operads}).
Recall from \S\ref{sec: perm right operadic module} that $\Pi^{\mathrm{Perm}}(n)$ denotes the poset of partitions equipped with a pointed element in each block (see the Hasse diagram for $n=3$ in Figure \ref{fig:permright}). Chapoton and Vallette \cite{chapotonvallette} proved that the maximal intervals in $\Pi^{\mathrm{Perm}}(n)$ are totally semimodular, and hence Cohen--Macaulay, and Proposition \ref{prop: right koszul} yields an isomorphism of operads
$$\Lambda h^\bullet(\Pi^{\mathrm{Perm}})\simeq \mathrm{PreLie},$$
which already appeared in \cite{chapotonvallette, vallettepartitionposets} at the level of linear species.
On the variant $\Lambda\htop{\bullet}(\Pi^{\mathrm{Perm}})$, we therefore get a right operadic module over $\mathrm{PreLie}$. We have proved in Proposition \ref{prop: Pi Perm Cohen Macaulay} that the poset $\Pi_+^{\operatorname{Perm}}(n)=\Pi^{\operatorname{Perm}}(n) \cup \{\hat{0}\}$ is Cohen--Macaulay and in Proposition \ref{prop:dim htop right Perm} that the dimension of its unique cohomology group is
$$\dim\htop{n-1}(\Pi^{\operatorname{Perm}}(n)) = (n-1)^{n-1},$$ 
which is, among others, the number of prime parking functions of length $n$. Furthermore, Proposition \ref{prop: cycle index series right operadic module} yields the following formula for the characters of the action of $\mathfrak{S}_n$ on those cohomology groups:
      \begin{equation}\label{eq: cycle index series for Pi Perm}
        \sum_{n\geq 1}(-1)^{n-1} \frac{1}{n!}\sum_{\sigma\in \mathfrak{S}_n} \mathrm{Tr}\left(\sigma\, \Big|\,\htop{n-1}(\Pi^{\operatorname{Perm}}(n))\right) p_{\lambda(\sigma)} = \left(\exp\left(\sum_{n\geq 1}\frac{p_n}{n}\right)-1\right)\circ \Sigma\mathrm{Z}_{\mathrm{PreLie}},
      \end{equation}
    where $\Sigma\mathrm{Z}_{\mathrm{PreLie}}$ is the plethystic inverse of $\mathrm{Z}_{\mathrm{Perm}} = p_1\exp(\sum_{n\geq 1}\frac{p_n}{n})$, or in other words:
    $$p_1= \left(p_1 \exp\left(\sum_{n\geq 1}\frac{p_n}{n}\right) \right)\circ \Sigma\mathrm{Z}_{\mathrm{PreLie}}.$$

\begin{tcolorbox}[colback=green!5!white,colframe=green!50!black,title=Open questions]
  \begin{itemize}
      \item Describe the cohomology groups $\htop{n-1}(\Pi^{\mathrm{Perm}}(n)) \simeq \widetilde{H}^{n-2}(\Pi^{\mathrm{Perm}}_+(n))$, e.g. using combinatorial bases. The character for the $\mathfrak{S}_n$-action is given by \eqref{eq: cycle index series for Pi Perm}.
      \item Is there a closed-form expression for the right-hand side of \eqref{eq: cycle index series for Pi Perm}?
      \item  Describe the (degree zero) right $\mathrm{PreLie}$-module $\Lambda\htop{\bullet}(\Pi^{\operatorname{Perm}})$.
  \end{itemize}
\end{tcolorbox}

\subsection{Left- and right-$\mathrm{Com}_2$-decorated partitions}

The operad $\operatorname{Com}_2$, introduced by Dotsenko and Khoroshkin \cite{DotsenkoKhoroshkinCom2}, is both left- and right-basic. It has the same cardinality in every arity as $\mathrm{Perm}$, but a different species structures: in arity $n\geq 1$ we have $\operatorname{Com}_2(n)=\{0,\ldots, n-1\}$ with the trivial action of the symmetric group $\mathfrak{S}_n$.

\subsubsection{Right-$\mathrm{Com}_2$-decorated partitions}

The poset $\Pi^{{\operatorname{Com}_2}}(n)$ is the poset of ``weighted partitions'' introduced in \cite{DotsenkoKhoroshkinCom2} and further studied by Gonz\'{a}lez D'Le\'{o}n and Wachs \cite{WeightedPart} and Gonz\'{a}lez D'Le\'{o}n, Hallam and Quiceno Dur\'{a}n \cite{whitneyTwins}. We refer to \cite[Figure 1]{WeightedPart} for the Hasse diagram of the poset $\Pi^{\mathrm{Com}_2}(3)$ (with the opposite convention).
Dotsenko and Khoroshkin proved that the maximal intervals in $\Pi^{\mathrm{Com}_2}(n)$ are Cohen--Macaulay, which by Proposition \ref{prop: right koszul} implies that $\QQ\mathrm{Com}_2$ is Koszul, and yields an isomorphism of operads
\begin{equation}\label{eq:iso-weighted-partitions-Lie2}
\Lambda h^\bullet(\Pi^{\mathrm{Com}_2}) \simeq \mathrm{Lie}_2,
\end{equation}
where $\mathrm{Lie}_2 = (\QQ\mathrm{Com}_2)^!$ was also introduced in \cite{DotsenkoKhoroshkinCom2}.  At the level of linear species, this isomorphism follows from \cite{DotsenkoKhoroshkinCom2} and Vallette's general formalism \cite{vallettepartitionposets}. Such an isomorphism was worked out more explicitly by Gonz\'{a}lez D'Le\'{o}n and Wachs \cite{WeightedPart}.

On the variant $\Lambda\htop{\bullet}(\Pi^{\mathrm{Com}_2})$, we therefore get a right module over the operad $\mathrm{Lie}_2$.
It was proved in \cite{WeightedPart} that $\Pi^{\operatorname{Com}_2}_+(n) := \Pi^{\operatorname{Com}_2}(n)\cup\{\hat{0}\}$ is EL-shellable, hence Cohen--Macaulay, which implies that $\Lambda \htop{\bullet}(\Pi^{\operatorname{Com}_2})$ is concentrated in degree zero. Its sequence of dimensions was worked out in \emph{op.~cit.} (and also follows from Proposition \ref{prop: generating series right operadic module}) and is given, in our notation, by
$$\dim\htop{n-1}(\Pi^{\operatorname{Com}_2}(n))=(n-1)^{n-1}.$$
Furthermore, Proposition \ref{prop: cycle index series right operadic module} yields the following formula for the characters of the action of $\mathfrak{S}_n$ on those cohomology groups:
      \begin{equation}\label{eq: cycle index series for Pi Com2}
        \sum_{n\geq 1}(-1)^{n-1} \frac{1}{n!}\sum_{\sigma\in\mathfrak{S}_n} \mathrm{Tr}\left( \sigma\,\Big|\,\htop{n-1}(\Pi^{\operatorname{Com}_2}(n)) \right) p_{\lambda(\sigma)} = \left(\exp\left(\sum_{n\geq 1}\frac{p_n}{n}\right)-1\right)\circ \Sigma\mathrm{Z}_{\mathrm{Lie_2}},
      \end{equation}
      where $\Sigma\mathrm{Z}_{\mathrm{Lie}_2}$ is the plethystic inverse of $\mathrm{Z}_{\mathrm{Com}_2}=\sum_{n\geq 1}n h_n= \sum_{n \geq 1}\sum_{\lambda} n \prod_{i=1}^n \frac{p_i^{\lambda_i}}{\lambda_i! i^{\lambda_i}}$ (sum over the tuples $\lambda=(\lambda_1, \ldots, \lambda_k)$, with $\lambda_i\geq 0$, satisfying $\sum_{i=1}^k i \lambda_i=n$).
      A formula for $\mathrm{Z}_{\mathrm{Lie}_2}$ was worked out more explicitly in \cite{DotsenkoKhoroshkinCom2}.

      \begin{tcolorbox}[colback=green!5!white,colframe=green!50!black,title=Open questions]
  \begin{itemize}
        \item Compare the isomorphism \eqref{eq:iso-weighted-partitions-Lie2} with that of \cite{WeightedPart}.
        \item Describe the cohomology groups $\htop{n-1}(\Pi^{\mathrm{Com}_2}(n))\simeq \widetilde{H}^{n-2}(\Pi^{\mathrm{Com}_2}_+(n))$, e.g. using combinatorial bases in the spirit of \cite{WeightedPart}. The character for the $\mathfrak{S}_n$-action is given by \eqref{eq: cycle index series for Pi Com2}.
        \item Is there a closed-form expression for the right-hand side of \eqref{eq: cycle index series for Pi Com2}?
      \item  Describe the (degree zero) right $\mathrm{Lie}_2$-module $\Lambda\htop{\bullet}(\Pi^{\operatorname{Com}_2})$. 
  \end{itemize}
\end{tcolorbox}

\subsubsection{Left-$\mathrm{Com}_2$-decorated partitions}

To our knowledge, the left-decorated version ${}^{\operatorname{Com}_2}\Pi$ has not been studied. The poset  ${}^{\operatorname{Com}_2}\Pi(n)$ of ``left-weighted partitions'' of $\{1,\ldots,n\}$ consists of pairs $(\pi,i)$ with $\pi\in \Pi(n)$ and $i\in\{0,\ldots,|\pi|-1\}$, and the covering relation is given by
\begin{equation*}
    (\pi,i) \lessdot (\pi',j) \quad \Longleftrightarrow \quad  \pi\lessdot \pi' \,\mbox{ and }\, j \in \{i,i+1\}.
\end{equation*}
The Hasse diagram of the poset ${}^{\operatorname{Com}_2}\Pi(3)$ is shown in Figure \ref{fig:Com2Pi}. 

\begin{figure}[h!]
    \centering
    \includegraphics[width=\linewidth]{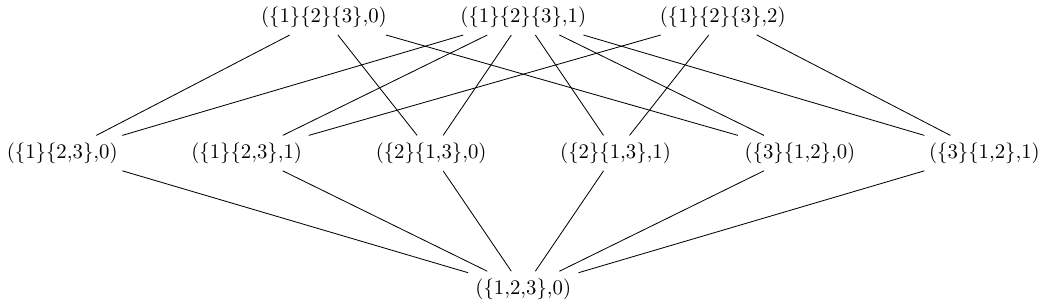}
    \caption{The poset of ``left-weighted partitions'' on three elements ${}^{\operatorname{Com}_2}\Pi(3)$. The poset ${}^{\operatorname{Com}_2}\Pi^+(3) := {}^{\operatorname{Com}_2}\Pi(3)\cup\{\hat{1}\}$ is Cohen--Macaulay with M\"{o}bius number $-4$.}
    \label{fig:Com2Pi}
\end{figure}

Since $\QQ\mathrm{Com}_2$ is Koszul, Proposition \ref{prop: left Koszul} implies that the maximal intervals of ${}^{\operatorname{Com}_2}\Pi(n)$ are Cohen--Macaulay and that we have an isomorphism of operads
\begin{equation}\label{eq: iso h of Com2 Pi Lie2}
\Lambda h^\bullet({}^{\operatorname{Com}_2}\Pi) \simeq \mathrm{Lie}_2.
\end{equation}
On the variant $\Lambda\hbottom{\bullet}({}^{\mathrm{Com}_2}\Pi)$, we therefore get a left operadic module over $\mathrm{Lie}_2$. We do not know if it is concentrated in degree zero, i.e. if the poset ${}^{\mathrm{Com}_2}\Pi^+(n) := {}^{\mathrm{Com}_2}\Pi(n) \cup\{\hat{1}\}$ is Cohen--Macaulay for all $n$. Proposition \ref{prop: generating series left operadic module} implies that the corresponding Möbius numbers are such that the  $(-1)^{n-1}\widecheck{\mu}({}^{\mathrm{Com}_2}\Pi(n))$ are given by the sequence \cite[A058863]{oeis}. The fact that the latter consists of positive numbers is consistent with the hypothesis that ${}^{\mathrm{Com}_2}\Pi^+(n)$ is indeed Cohen--Macaulay for all $n$. Furthermore, Proposition \ref{prop: cycle index series left operadic module} yields the $\mathfrak{S}_n$-equivariant Euler characteristics:
\begin{equation}\label{eq: cycle index series for Com2 Pi}\sum_{k\geq 1}(-1)^k \mathrm{Z}_{\hbottom{k}({}^{\mathrm{Com_2}}\Pi)} = \Sigma\mathrm{Z}_{\mathrm{Lie}_2}\circ \left(\exp\left(\sum_{n\geq 1}\frac{p_n}{n}\right)-1\right).
\end{equation}

\begin{tcolorbox}[colback=green!5!white,colframe=green!50!black,title=Open questions]
  \begin{itemize}   
     \item Give a combinatorial description of the isomorphism \eqref{eq: iso h of Com2 Pi Lie2} in the spirit of \cite{WeightedPart}.
      \item Is the poset ${}^{\mathrm{Com}_2}\Pi^+(n)$ Cohen--Macaulay for all $n$?
      \item Describe the cohomology groups $\hbottom{\bullet}({}^{\mathrm{Com}_2}\Pi(n))\simeq \widetilde{H}^{\bullet-1}({}^{\mathrm{Com}_2}\Pi^+(n))$, e.g. using combinatorial bases. The $\mathfrak{S}$-equivariant Euler characteristic is given by \eqref{eq: cycle index series for Com2 Pi}.
        \item Is there a closed-form expression for the right-hand side of \eqref{eq: cycle index series for Com2 Pi}?
      \item Describe the (graded) left $\operatorname{Lie}_2$-module $\Lambda \hbottom{\bullet}({}^{\mathrm{Com}_2}\Pi)$.
  \end{itemize}
\end{tcolorbox}

\subsection{Left-$\mathrm{Dup}$-decorated partitions}

The operad $\operatorname{Dup}$, encoding \emph{duplicial} algebras, was introduced by Loday \cite{LodayDup}, and is left-basic. It is the symmetrization of a non-symmetric operad.
An element of $\operatorname{Dup}(S)$ is a planar binary tree equipped with a labeling of its internal nodes by $S$. An element of the poset
${}^{\operatorname{Dup}}\Pi(S)$ is therefore an element of $\operatorname{Dup}(\pi)$, for a partition $\pi$ of $S$, i.e. a 
planar binary tree equipped with a labeling of its internal nodes by blocks of $\pi$.
The covering relations are given by the following local rules, for $T$ and $T'$ two disjoint subsets of $S$ and $\tau_l$, $\tau_r$, two subtrees:
\begin{align*}
    \begin{tikzpicture}[grow=up, scale=0.5, baseline=0.3cm]
        \node{$T \sqcup T'$}
           child{node{$\tau_r$}}
           child{node{$\tau_l$}};
    \end{tikzpicture} & \lessdot 
    \begin{tikzpicture}[grow=up, scale=0.5, baseline=0.3cm]    
    \node{$T$}
           child{node{$T'$} child{node{$\tau_r$}}child{}}
           child{node{$\tau_l$}};
    \end{tikzpicture} \\
        \begin{tikzpicture}[grow=up, scale=0.5, baseline=0.3cm]
        \node{$T\sqcup T'$}
           child{node{$\tau_r$}}
           child{node{$\tau_l$}};
    \end{tikzpicture} & \lessdot \begin{tikzpicture}[grow=up, scale=0.5, baseline=0.3cm]
        \node{$T$}
           child{node{$\tau_r$}}
           child{node{$T'$} child{} child{node{$\tau_l$}}};
    \end{tikzpicture}
\end{align*}
The Hasse diagram of the poset ${}^{\operatorname{Dup}}\Pi(3)$ is shown in Figure \ref{fig:DupPi}. 

\begin{figure}
    \centering
    \includegraphics[width=0.9\textheight, angle=90]{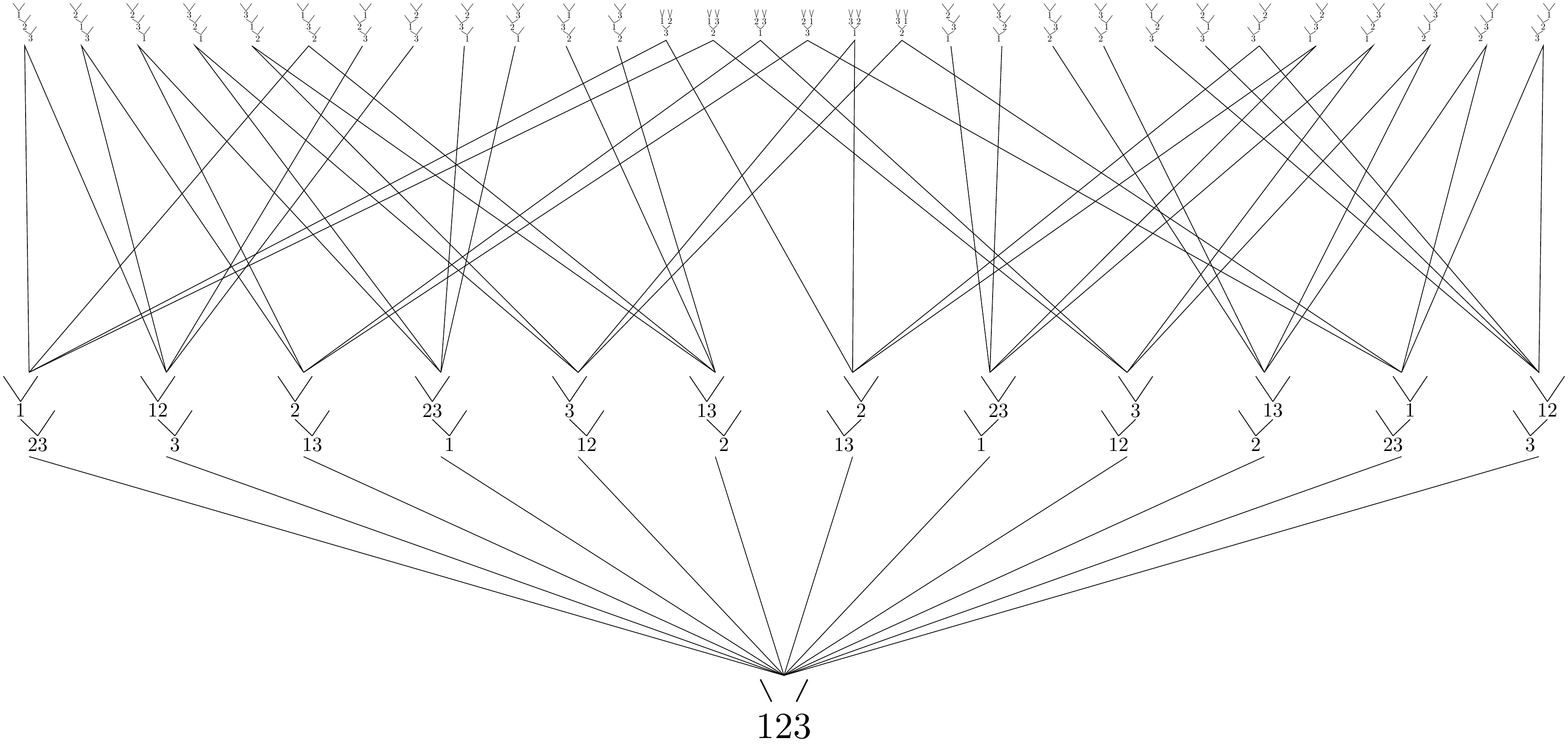}
    \caption{The Hasse diagram of the poset of left-$\mathrm{Dup}$-decorated partitions on three elements, ${}^{\operatorname{Dup}}\Pi(3)$. The poset ${}^{\operatorname{Dup}}\Pi^+(3):={}^{\operatorname{Dup}}\Pi(3)\cup\{\hat{1}\}$ is Cohen--Macaulay with M\"{o}bius number $-7$.}
    \label{fig:DupPi}
\end{figure}

Since $\operatorname{Dup}$ is a Koszul operad by \cite[5.4]{LodayDup}, Proposition \ref{prop: left Koszul} implies that the maximal intervals of ${}^{\operatorname{Dup}}\Pi(S)$ are Cohen--Macaulay and that we have an isomorphism of operads 
$$\Lambda\h{\bullet}({}^{\operatorname{Dup}}\Pi) \simeq (\QQ\mathrm{Dup})^!.$$
The Koszul dual operad $(\QQ\mathrm{Dup})^!$ was described in \emph{loc.~cit.}, and $(\QQ\mathrm{Dup})^!(n)$ consists of $n$ copies of the regular representation of $\mathfrak{S}_n$. On the variant $\Lambda\hbottom{\bullet}({}^{\mathrm{Dup}}\Pi)$, we get a left operadic module over $(\QQ\mathrm{Dup})^!$. We do not know if it is concentrated in degree zero, i.e. if the poset ${}^{\mathrm{Dup}}\Pi^+(n) := {}^{\mathrm{Dup}}\Pi(n) \cup\{\hat{1}\}$ is Cohen--Macaulay for all $n$. Proposition \ref{prop: generating series left operadic module} yields the Möbius numbers $(-1)^{n-1}\widecheck{\mu}({}^{\mathrm{Dup}}\Pi(n))=2^n-1$. The fact that these numbers are all positive is consistent with the hypothesis that ${}^{\mathrm{Dup}}\Pi^+(n)$ is indeed Cohen--Macaulay for all $n$. Furthermore, Proposition \ref{prop: cycle index series left operadic module} yields the $\mathfrak{S}_n$-equivariant Euler characteristic:
\begin{equation}\label{eq: cycle index series for Dup Pi}\sum_{k\geq 1}(-1)^k \mathrm{Z}_{\hbottom{k}({}^{\mathrm{Dup}}\Pi)} = \frac{p_1}{(1+p_1)^2}\circ \left(\exp\left(\sum_{n\geq 1}\frac{p_n}{n}\right)-1\right) = \exp\left(-\sum_{n\geq 1}\frac{p_n}{n}\right) - \exp\left(-2\sum_{n\geq 1}\frac{p_n}{n}\right).
\end{equation}

\begin{tcolorbox}[colback=green!5!white,colframe=green!50!black,title=Open questions]
  \begin{itemize}
      \item Is the poset ${}^{\operatorname{Dup}}\Pi^+(n)$ Cohen--Macaulay for all $n$?
      \item Describe the cohomology groups $\hbottom{\bullet}({}^{\mathrm{Dup}}\Pi(n))\simeq \widetilde{H}^{\bullet-1}({}^{\mathrm{Dup}}\Pi^+(n))$, e.g. using combinatorial bases (see \cite[A000225]{oeis}). The $\mathfrak{S}$-equivariant Euler characteristic is given by \eqref{eq: cycle index series for Dup Pi}.
      \item  Describe the (graded) left $(\QQ\mathrm{Dup})^!$-module $\Lambda\hbottom{\bullet}({}^{\operatorname{Dup}}\Pi)$.
  \end{itemize}
\end{tcolorbox}

\subsection{Right-$\mathrm{Trias}$-decorated partitions}

The operad $\operatorname{Trias}$, encoding \emph{triassociative} algebras, was introduced by Loday and Ronco \cite{LodayRoncoTrias}, and is right-basic. It is the symmetrization of a non-symmetric operad. The set $\mathrm{Trias}(S)$ consists of a list of all the elements of $S$ together with a non-empty subset of pointed elements (a ``multipointed list''). It has cardinality $(2^n-1)n!$ if $|S|=n$. The poset 
$\Pi^{\operatorname{Trias}}(S)$ therefore consists of partitions of $S$ whose blocks are multipointed lists, where $\pi\leq \pi'$ if and only if the blocks of $\pi$ are concatenations of the blocks of $\pi'$ and all the pointed elements in $\pi$ are pointed in $\pi'$. The Hasse diagram of the poset $\Pi^{\mathrm{TriAs}}(2)$ is shown in Figure \ref{fig:Trias}. Those posets have been previously considered by Vallette \cite{vallettepartitionposets} who called them ``multipointed ordered partition posets''. 

\begin{figure}[h!]
    \centering
    \includegraphics{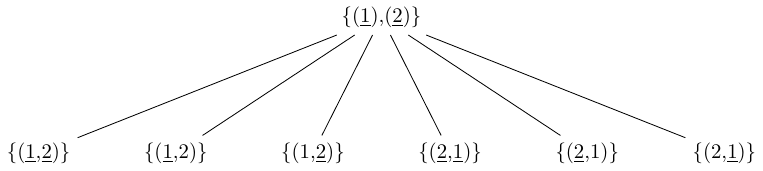}
    \caption{The Hasse diagram of the poset $\Pi^{\mathrm{Trias}}(2)$ of right-$\mathrm{Trias}$-decorated partitions on $2$ elements, viewed as partitions of $\{1,2\}$ whose blocks are multipointed lists (the multipointed elements are underlined).}
    \label{fig:Trias}
\end{figure}

As explained in \cite{LodayRoncoTrias}, the operad $\mathrm{Trias}$ is Koszul, with Koszul dual the operad $\mathrm{TriDend}$ encoding tridendriform algebras, where $\mathrm{TriDend}(S)$ has a basis consisting of planar trees with sectors indexed by $S$. Therefore, Proposition \ref{prop: right koszul} implies that the maximal intervals in $\Pi^{\mathrm{Trias}}(n)$ are Cohen--Macaulay for all $n$, and that we have an isomorphism of operads
$$\Lambda h^\bullet(\Pi^{\mathrm{Trias}})\simeq \mathrm{TriDend},$$
which already appeared in \cite{vallettepartitionposets} at the level of linear species. On the variant $\Lambda\htop{\bullet}(\Pi^{\mathrm{Trias}})$, we get a right operadic module over $\mathrm{TriDend}$. We do not know if it is concentrated in degree zero, i.e. if the poset $\Pi^{\mathrm{Trias}}_+(n) := \Pi^{\mathrm{Trias}}(n) \cup\{\hat{0}\}$ is Cohen--Macaulay for all $n$. Proposition \ref{prop: generating series right operadic module} yields the Möbius numbers \cite[A383991]{oeis}, which beg for a combinatorial interpretation. It is believable that these numbers alternate in sign, which would be consistent with the hypothesis that $\Pi^{\mathrm{Trias}}_+(n)$ is indeed Cohen--Macaulay for all $n$. Furthermore, Proposition \ref{prop: cycle index series right operadic module} yields the $\mathfrak{S}_n$-equivariant Euler characteristic:
\begin{equation}\label{eq: cycle index series for Pi Trias}
 \sum_{k\geq 0}(-1)^k \mathrm{Z}_{\htop{k}(\Pi^{\mathrm{Trias}})}  =  \left( \exp\Bigg(\sum_{n\geq 1}\frac{p_n}{n}\Bigg) -1 \right) \circ \Sigma \mathrm{Z}_{\mathrm{TriDend}} = \exp\left(\sum_{n\geq 1}\frac{1}{n}\frac{1+3p_n+\sqrt{1+6p_n+p_n^2}}{4p_n}\right)-1.
\end{equation}

\begin{tcolorbox}[colback=green!5!white,colframe=green!50!black,title=Open questions]
  \begin{itemize}
      \item Is the poset $\Pi^{\operatorname{Trias}}_+(n)$ Cohen--Macaulay for all $n$?
      \item  Find a formula and/or a combinatorial interpretation for \cite[A383991]{oeis}.
      \item Describe the cohomology groups $\htop{\bullet}(\Pi^{\mathrm{Trias}}(n))\simeq \widetilde{H}^{\bullet-1}(\Pi^{\mathrm{Trias}}_+(n))$, e.g. using combinatorial bases (see \cite[A383991]{oeis}). The $\mathfrak{S}$-equivariant Euler characteristic is given by \eqref{eq: cycle index series for Pi Trias}.
      \item Describe the (graded) right $\mathrm{TriDend}$-module $\Lambda\htop{\bullet}(\Pi^{\mathrm{Trias}})$.
  \end{itemize}
\end{tcolorbox}

\subsection{Right-$\mathrm{ComTrias}$-decorated partitions}

The $\operatorname{ComTrias}$, encoding \emph{commutative-triassociative} algebras, was introduced by \cite{vallettepartitionposets}, and is right-basic. The set $\mathrm{ComTrias}(S)$ can be described as the set of non-empty subsets of $S$, also known as ``multipointings'' of $S$. Therefore, the poset $\Pi^{\mathrm{ComTrias}}(S)$ consists of partitions of $S$ where each block is equipped with a multipointing.
It was proved by Chapoton and Vallette \cite{chapotonvallette} that every poset $\Pi^{\mathrm{ComTrias}}(n)$ is totally semi-modular and therefore Cohen--Macaulay. Proposition \ref{prop: right koszul} then applies and shows that $\mathrm{ComTrias}$ is Koszul, its Koszul dual being denoted by $\mathrm{PostLie}$, the operad encoding \emph{post-Lie} algebras. We then get an isomorphism of operads
$$\Lambda h^\bullet(\Pi^{\mathrm{ComTrias}})\simeq \mathrm{PostLie},$$
which first appeared in \cite{vallettepartitionposets} at the level of linear species.
On the variant $\Lambda\htop{\bullet}(\Pi^{\mathrm{ComTrias}})$, we get a right operadic module over $\mathrm{PostLie}$. We do not know if it is concentrated in degree zero, i.e. if the poset $\Pi^{\mathrm{ComTrias}}_+(n) := \Pi^{\mathrm{ComTrias}}(n) \cup\{\hat{0}\}$ is Cohen--Macaulay for all $n$. Proposition \ref{prop: generating series right operadic module} implies that the corresponding Möbius numbers are such that $\widehat{\mu}(\Pi^{\mathrm{ComTrias}}(n))=(-1)^{n-1}\frac{(2n)!}{n!}$ \cite[A001813]{oeis}. Their sign pattern is consistent with the hypothesis that $\Pi^{\mathrm{ComTrias}}_+(n)$ is indeed Cohen--Macaulay for all $n$. Furthermore, Proposition \ref{prop: cycle index series right operadic module} yields the $\mathfrak{S}_n$-equivariant Euler characteristic:
\begin{equation}\label{eq: cycle index series for Pi ComTrias}
 \sum_{k\geq 0}(-1)^k \mathrm{Z}_{\htop{k}(\Pi^{\mathrm{ComTrias}})}  =  \left( \exp\Bigg(\sum_{n\geq 1}\frac{p_n}{n}\Bigg) -1 \right) \circ \Sigma \mathrm{Z}_{\mathrm{PostLie}}
\end{equation}
where $\Sigma\mathrm{Z}_{\mathrm{PostLie}}$ is the plethystic inverse of $\mathrm{Z}_{\mathrm{ComTrias}} = \exp\left(2\sum_{n\geq 1}\frac{p_n}{n}\right) - \exp\left(\sum_{n\geq 1}\frac{p_n}{n}\right).$

\begin{tcolorbox}[colback=green!5!white,colframe=green!50!black,title=Open questions]
  \begin{itemize}
      \item Is the poset $\Pi^{\operatorname{ComTrias}}_+(n)$ Cohen--Macaulay for all $n$?
      \item Describe the cohomology groups $\htop{\bullet}(\Pi^{\mathrm{ComTrias}}(n))\simeq \widetilde{H}^{\bullet-1}(\Pi^{\mathrm{ComTrias}}_+(n))$, e.g. using combinatorial bases (see \cite[A001813]{oeis}). The $\mathfrak{S}$-equivariant Euler characteristic is given by \eqref{eq: cycle index series for Pi ComTrias}.
      \item Describe the (graded) right $\mathrm{PostLie}$-module $\Lambda\htop{\bullet}(\Pi^{\mathrm{ComTrias}})$.
  \end{itemize}
\end{tcolorbox}

\subsection{Non-crossing $2$-partitions}

In \S\ref{sec:NC2P} we have studied an operadic poset species consisting of the non-crossing $2$-partition posets $\Pi_2(n)$. We have proved in Proposition \ref{prop:operad-NC2P} that $h^\bullet(\Pi_2(n))$ and $\hbottom{\bullet}(\Pi_2(n))$ are both concentrated in degree $n-1$, yielding a degree zero operad $\Lambda h^\bullet(\Pi_2)$ and a degree zero left operadic module $\Lambda\hbottom{\bullet}(\Pi_2)$.

\begin{tcolorbox}[colback=green!5!white,colframe=green!50!black,title=Open questions]
  \begin{itemize}
      \item Describe the (degree zero) operad $\Lambda h^\bullet(\Pi_2(n))$ on the cohomology groups $h^{n-1}(\Pi_2(n))$. The underlying linear species is that of planar binary trees with a labeling of the leaves. However, $\Lambda h^\bullet(\Pi_2(n))$ is not the Magmatic operad on one generator since the binary generator satisfies the pre-Lie identity (Proposition \ref{prop: prelie in Pi 2}).
      \item Describe the (degree zero) left operadic module $\Lambda\hbottom{\bullet}(\Pi_2(n))$ on the cohomology groups $\hbottom{n-1}(\Pi_2(n))\simeq \widetilde{H}^{n-2}(\Pi_2(n)\cup\{\hat{1}\})$. The underlying linear species is, up to the sign representation, that of prime parking functions (Proposition \ref{prop:DOJVR}). 
  \end{itemize}
\end{tcolorbox}

\bibliographystyle{hyperamsalpha}
\bibliography{biblio}

\providecommand{\bysame}{\leavevmode\hbox to3em{\hrulefill}\thinspace}
\providecommand{\MR}{\relax\ifhmode\unskip\space\fi MR }
\providecommand{\MRhref}[2]{%
  \href{http://www.ams.org/mathscinet-getitem?mr=#1}{#2}
}
\providecommand{\href}[2]{#2}
\begin{thebibliography}{GDHQD23}

\bibitem[ALW16]{ArmstrongLoehrWarrington}
D.~Armstrong, N.~A. Loehr, and G.~S. Warrington, \emph{Rational parking
  functions and {Catalan} numbers},
  \href{http://dx.doi.org/10.1007/s00026-015-0293-6}{Ann. Comb. \textbf{20}
  (2016)}, no.~1, 21--58.

\bibitem[Bah21]{bahturinbook}
Y.~Bahturin, \href{http://dx.doi.org/10.1515/9783110566659}{\emph{Identical
  relations in {L}ie algebras}}, 2nd ed., De Gruyter Expositions in
  Mathematics, vol.~68, De Gruyter, Berlin, 2021.

\bibitem[BDO20]{BurgunderDO}
E.~Burgunder and B.~Delcroix-Oger,
  \href{http://dx.doi.org/10.4171/204-1/2}{\emph{Structure theorems for
  dendriform and tridendriform algebras}}, Algebraic combinatorics, resurgence,
  moulds and applications (CARMA). Volume 1, Berlin: European Mathematical
  Society (EMS), 2020, pp.~19--66 (English).

\bibitem[Bj{\"o}80]{bjornershellable}
A.~Bj{\"o}rner, \emph{Shellable and {C}ohen-{M}acaulay partially ordered sets},
  \href{http://dx.doi.org/10.2307/1999881}{Trans. Amer. Math. Soc. \textbf{260}
  (1980)}, no.~1, 159--183.

\bibitem[BMMM01]{BradyMccammondMeierMiller}
N.~Brady, J.~McCammond, J.~Meier, and A.~Miller, \emph{The pure symmetric
  automorphisms of a free group form a duality group},
  \href{http://dx.doi.org/10.1006/jabr.2001.8944}{J. Algebra \textbf{246}
  (2001)}, no.~2, 881--896.

\bibitem[Bra44]{brandt}
A.~J. Brandt, \emph{The free {L}ie ring and {L}ie representations of the full
  linear group}, \href{http://dx.doi.org/10.2307/1990324}{Trans. Amer. Math.
  Soc. \textbf{56} (1944)}, 528--536.

\bibitem[BW83]{BjornerWachsLexSh}
A.~Bj\"{o}rner and M.~Wachs, \emph{On lexicographically shellable posets},
  \href{http://dx.doi.org/10.2307/1999359}{Trans. Amer. Math. Soc. \textbf{277}
  (1983)}, no.~1, 323--341.

\bibitem[Cha01a]{ChapotonPerm}
F.~Chapoton, \emph{An endofunctor in the category of operads}, Dialgebras and
  related operads, Berlin: Springer, 2001, pp.~105--110 (French).

\bibitem[Cha01b]{chapotonendofoncteur}
\bysame, \href{http://dx.doi.org/10.1007/3-540-45328-8_4}{\emph{Un endofoncteur
  de la cat{\'e}gorie des op{\'e}rades}}, Dialgebras and Related Operads,
  Springer Berlin Heidelberg, Berlin, Heidelberg, 2001, pp.~105--110.

\bibitem[Cha07]{chapotonhyperarbres}
F.~Chapoton, \emph{Hyperarbres, arbres enracinés et partitions pointées},
  \href{http://dx.doi.org/10.4310/HHA.2007.v9.n1.a8}{Homology Homotopy Appl.
  \textbf{9} (2007)}, no.~1, 193--212.

\bibitem[Cha25]{operadia}
F.~Chapoton, \emph{Operadia}, 2025. \url{https://operadia.pythonanywhere.com/}.

\bibitem[CL01]{ChapotonLivernet}
F.~Chapoton and M.~Livernet, \emph{Pre-{L}ie algebras and the rooted trees
  operad}, \href{http://dx.doi.org/10.1155/S1073792801000198}{Internat. Math.
  Res. Notices (2001)}, no.~8, 395--408.

\bibitem[Cor22]{coron}
B.~Coron, \emph{Matroids, {F}eynman categories, and {K}oszul duality}, 2022.
  \href{http://arxiv.org/abs/2211.12370}{{\tt arXiv:2211.12370}}.

\bibitem[CV06]{chapotonvallette}
F.~Chapoton and B.~Vallette, \emph{Pointed and multi-pointed partitions of type
  {$A$} and {$B$}}, \href{http://dx.doi.org/10.1007/s10801-006-8346-x}{J.
  Algebraic Combin. \textbf{23} (2006)}, no.~4, 295--316.

\bibitem[DCP95]{deConciniProcesi}
C.~De~Concini and C.~Procesi, \emph{Wonderful models of subspace arrangements},
  \href{http://dx.doi.org/10.1007/BF01589496}{Selecta Math. (N.S.) \textbf{1}
  (1995)}, no.~3, 459--494.

\bibitem[DDOJV25]{DDOJV}
B.~Deb, B.~Delcroix-Oger, and M.~Josuat-Verg{\`e}s, \emph{Poset structures on
  multilabelled trees and hypertrees}. In preparation.

\bibitem[DJ24]{dupontjuteau}
C.~Dupont and D.~Juteau, \emph{The localization spectral sequence in the
  motivic setting}, \href{http://dx.doi.org/10.2140/agt.2024.24.1431}{Algebr.
  Geom. Topol. \textbf{24} (2024)}, no.~3, 1431--1466.

\bibitem[DK07]{DotsenkoKhoroshkinCom2}
V.~V. Dotsenko and A.~S. Khoroshkin, \emph{Character formulas for the operad of
  two compatible brackets and for the bi-{Hamiltonian} operad},
  \href{http://dx.doi.org/10.1007/s10688-007-0001-3}{Funct. Anal. Appl.
  \textbf{41} (2007)}, no.~1, 1--17 (English).

\bibitem[DK10]{DotsenkoKhoroshkin}
V.~Dotsenko and A.~Khoroshkin, \emph{Gr\"{o}bner bases for operads},
  \href{http://dx.doi.org/10.1215/00127094-2010-026}{Duke Math. J. \textbf{153}
  (2010)}, no.~2, 363--396.

\bibitem[DO17]{OgerSemiPointedPartition}
B.~Delcroix-Oger, \emph{Semi-pointed partition posets and species},
  \href{http://dx.doi.org/10.1007/s10801-016-0727-1}{J. Algebraic Combin.
  \textbf{45} (2017)}, no.~3, 857--886.

\bibitem[DOD25]{DOD2}
B.~Delcroix-Oger and C.~Dupont, \emph{Hypertrees, post-{L}ie and pre-{L}ie
  (working title)}. In preparation.

\bibitem[DOJVR22]{DOJVR}
B.~Delcroix-Oger, M.~Josuat-Verg{\`e}s, and L.~Randazzo, \emph{Some properties
  of the parking function poset},
  \href{http://dx.doi.org/10.37236/10714}{Electron. J. Comb. \textbf{29}
  (2022)}, no.~4, Paper 4.42, 49 pp.

\bibitem[Ede80a]{Edelman}
P.~H. Edelman, \emph{Chain enumeration and non-crossing partitions},
  \href{http://dx.doi.org/10.1016/0012-365X(80)90033-3}{Discrete Math.
  \textbf{31} (1980)}, 171--180.

\bibitem[Ede80b]{EdelmanZeta}
\bysame, \emph{Zeta polynomials and the {M}\"obius function},
  \href{http://dx.doi.org/10.1016/S0195-6698(80)80034-5}{European J. Combin.
  \textbf{1} (1980)}, no.~4, 335--340.

\bibitem[FK04]{feichtnerkozlov}
E.~M. Feichtner and D.~N. Kozlov, \emph{Incidence combinatorics of
  resolutions}, \href{http://dx.doi.org/10.1007/s00029-004-0298-1}{Selecta
  Mathematica \textbf{10} (2004)}, no.~1, 37.

\bibitem[Fol66]{folkman}
J.~Folkman, \emph{The homology groups of a lattice}, J. Math. Mech. \textbf{15}
  (1966), 631--636.

\bibitem[Fre04]{fressepartitionposets}
B.~Fresse, \href{http://dx.doi.org/10.1090/conm/346/06287}{\emph{Koszul duality
  of operads and homology of partition posets}}, Homotopy theory: relations
  with algebraic geometry, group cohomology, and algebraic {$K$}-theory,
  Contemp. Math., vol. 346, Amer. Math. Soc., Providence, RI, 2004,
  pp.~115--215.

\bibitem[GDHQD23]{whitneyTwins}
R.~S. Gonz\'{a}lez~D'Le\'{o}n, J.~Hallam, and Y.~A. Quiceno~D., \emph{Whitney
  twins, {W}hitney duals, and operadic partition posets}, 2023.
  \href{http://arxiv.org/abs/2307.07480}{{\tt arXiv:2307.07480}}.

\bibitem[GDW16]{WeightedPart}
R.~S. Gonz{\'a}lez~D'Le{\'o}n and M.~L. Wachs, \emph{On the (co)homology of the
  poset of weighted partitions},
  \href{http://dx.doi.org/10.1090/tran/6483}{Trans. Am. Math. Soc. \textbf{368}
  (2016)}, no.~10, 6779--6818 (English).

\bibitem[Han81]{Han81}
P.~Hanlon, \emph{The fixed-point partition lattices},
  \href{http://dx.doi.org/10.2140/pjm.1981.96.319}{Pac. J. Math. \textbf{96}
  (1981)}, 319--341 (English).

\bibitem[Hof10]{Hoffbeck}
E.~Hoffbeck, \emph{A {P}oincar\'e--{B}irkhoff--{W}itt criterion for {K}oszul
  operads}, \href{http://dx.doi.org/10.1007/s00229-009-0303-2}{Manuscripta
  Math. \textbf{131} (2010)}, no.~1-2, 87--110.

\bibitem[HW95]{hanlonwachs}
P.~Hanlon and M.~Wachs, \emph{On {L}ie {$k$}-algebras},
  \href{http://dx.doi.org/10.1006/aima.1995.1038}{Adv. Math. \textbf{113}
  (1995)}, no.~2, 206--236.

\bibitem[JMM07]{jensenMcCammondMeier}
C.~Jensen, J.~McCammond, and J.~Meier, \emph{The {Euler} characteristic of the
  {Whitehead} automorphism group of a free product.},
  \href{http://dx.doi.org/10.1090/S0002-9947-07-03967-0}{Trans. Am. Math. Soc.
  \textbf{359} (2007)}, no.~6, 2577--2595.

\bibitem[Joh63]{Johnson}
S.~M. Johnson, \emph{Generation of permutations by adjacent transposition},
  \href{http://dx.doi.org/10.2307/2003846}{Math. Comp. \textbf{17} (1963)},
  282--285.

\bibitem[Joy86]{joyal}
A.~Joyal, \emph{Foncteurs analytiques et esp{\`e}ces de structures},
  Combinatoire {\'e}num{\'e}rative (Berlin, Heidelberg) (G.~Labelle and
  P.~Leroux, eds.), Springer Berlin Heidelberg, 1986, pp.~126--159.

\bibitem[Kon99]{kontsevichoperadsmotives}
M.~Kontsevich, \emph{Operads and motives in deformation quantization},
  \href{http://dx.doi.org/10.1023/A:1007555725247}{Lett. Math. Phys.
  \textbf{48} (1999)}, no.~1, 35--72.

\bibitem[Kon19]{kontsevichderived}
\bysame, \href{http://dx.doi.org/10.24033/ast}{\emph{Derived
  {G}rothendieck-{T}eichm\"uller group and graph complexes [after {T}.
  {W}illwacher]}}, no. 407, 2019, pp.~Exp. No. 1126, 183--211. S\'eminaire
  Bourbaki. Vol. 2016/2017. Expos\'es 1120--1135.

\bibitem[Kre72]{Kreweras}
G.~Kreweras, \emph{Sur les partitions non crois{\'e}es d'un cycle},
  \href{http://dx.doi.org/10.1016/0012-365X(72)90041-6}{Discrete Math.
  \textbf{1} (1972)}, 333--350.

\bibitem[Li24]{liCLvsEL}
T.~Li, \emph{{CL}-shellable posets with no {EL}-shellings},
  \href{http://dx.doi.org/10.1090/proc/16602}{Proc. Am. Math. Soc. \textbf{152}
  (2024)}, no.~5, 1821--1830.

\bibitem[Liv06]{LivernetNAP}
M.~Livernet, \emph{A rigidity theorem for pre-{Lie} algebras},
  \href{http://dx.doi.org/10.1016/j.jpaa.2005.10.014}{J. Pure Appl. Algebra
  \textbf{207} (2006)}, no.~1, 1--18 (English).

\bibitem[Lod01]{LodayDias}
J.-L. Loday, \emph{Dialgebras}, pp.~7--66, Springer Berlin Heidelberg, Berlin,
  Heidelberg, 2001. \url{https://doi.org/10.1007/3-540-45328-8_2}.

\bibitem[Lod08]{LodayDup}
\bysame, \emph{Generalized bialgebras and triples of operads}, Ast{\'e}risque,
  vol. 320, Paris: Soci{\'e}t{\'e} Math{\'e}matique de France, 2008 (English).
  \url{smf4.emath.fr/Publications/Asterisque/2008/320/html/smf_ast_320.html}.

\bibitem[LR03]{LodayRoncoDipt}
J.-L. Loday and M.~Ronco, \emph{Cofree {Hopf} algebras.},
  \href{http://dx.doi.org/10.1016/S1631-073X(03)00288-7}{C. R., Math., Acad.
  Sci. Paris \textbf{337} (2003)}, no.~3, 153--158 (French).

\bibitem[LR04]{LodayRoncoTrias}
J.-L. Loday and M.~Ronco,
  \href{http://dx.doi.org/10.1090/conm/346/06296}{\emph{Trialgebras and
  families of polytopes}}, Homotopy theory: relations with algebraic geometry,
  group cohomology, and algebraic {$K$}-theory, Contemp. Math., vol. 346, Amer.
  Math. Soc., Providence, RI, 2004, pp.~369--398.
  \url{https://doi.org/10.1090/conm/346/06296}.

\bibitem[LR06]{LodayRonco2as}
J.-L. Loday and M.~Ronco, \emph{On the structure of cofree hopf algebras},
  \href{http://dx.doi.org/doi:10.1515/CRELLE.2006.025}{Journal für die reine
  und angewandte Mathematik \textbf{2006} (2006)}, no.~592, 123--155.
  \url{https://doi.org/10.1515/CRELLE.2006.025}.

\bibitem[LV12]{lodayvallette}
J.-L. Loday and B.~Vallette,
  \href{http://dx.doi.org/10.1007/978-3-642-30362-3}{\emph{Algebraic operads}},
  Grundlehren der mathematischen Wissenschaften [Fundamental Principles of
  Mathematical Sciences], vol. 346, Springer, Heidelberg, 2012.

\bibitem[MM96]{McCulloughMiller}
D.~McCullough and A.~Miller, \emph{Symmetric automorphisms of free products},
  \href{http://dx.doi.org/10.1090/memo/0582}{Mem. Amer. Math. Soc. \textbf{122}
  (1996)}, no.~582, viii+97.

\bibitem[MM04]{mccammondmeier}
J.~McCammond and J.~Meier, \emph{The hypertree poset and the {$\ell^2$}-{B}etti
  numbers of the motion group of the trivial link},
  \href{http://dx.doi.org/10.1007/s00208-003-0499-5}{Math. Ann. \textbf{328}
  (2004)}, no.~4, 633--652.

\bibitem[MY91]{MendezYang}
M.~Mendez and J.~Yang, \emph{M\"obius species},
  \href{http://dx.doi.org/10.1016/0001-8708(91)90051-8}{Adv. Math. \textbf{85}
  (1991)}, no.~1, 83--128.

\bibitem[NT20]{NovelliThibonTridup}
J.-C. Novelli and J.-Y. Thibon,
  \href{http://dx.doi.org/10.4171/204-1/6}{\emph{Duplicial algebras, parking
  functions, and {Lagrange} inversion}}, Algebraic combinatorics, resurgence,
  moulds and applications (CARMA). Volume 1, Berlin: European Mathematical
  Society (EMS), 2020, pp.~263--290 (English).

\bibitem[{OEI}25]{oeis}
{OEIS Foundation Inc.}, \emph{{The On-Line Encyclopedia of Integer Sequences
  (OEIS)}}, 2025. \url{http://oeis.org/}.

\bibitem[Oge13]{ogerhomologyhypertree}
B.~Oger, \emph{Action of the symmetric groups on the homology of the hypertree
  posets}, \href{http://dx.doi.org/10.1007/s10801-013-0432-2}{J. Algebraic
  Combin. \textbf{38} (2013)}, no.~4, 915--945.

\bibitem[Pei81]{PeirceAs}
B.~Peirce, \emph{Linear associative algebra.}, Am. J. Math. \textbf{4} (1881),
  99--230 (English). \url{archive.org/details/linearassociati00peirgoog}.

\bibitem[SS25]{sagansundaram}
B.~E. Sagan and S.~Sundaram, \emph{Ordered set partition posets}, 2025.
  \href{http://arxiv.org/abs/2506.23355}{{\tt arXiv:2506.23355}}.

\bibitem[Sta74]{StanleyZeta}
R.~P. Stanley, \emph{Combinatorial reciprocity theorems},
  \href{http://dx.doi.org/10.1016/0001-8708(74)90030-9}{Advances in Math.
  \textbf{14} (1974)}, 194--253.

\bibitem[Sta82]{Stan82}
\bysame, \emph{Some aspects of groups acting on finite posets},
  \href{http://dx.doi.org/10.1016/0097-3165(82)90017-6}{J. Comb. Theory, Ser. A
  \textbf{32} (1982)}, 132--161.

\bibitem[Sta97]{stanleyparking}
\bysame, \href{http://dx.doi.org/10.37236/1335}{\emph{Parking functions and
  noncrossing partitions}}, vol.~4, 1997, pp.~Research Paper 20, approx. 14.
  The Wilf Festschrift (Philadelphia, PA, 1996).

\bibitem[Ste64]{Steinhaus}
H.~Steinhaus, \emph{One hundred problems in elementary mathematics}, Basic
  Books, Inc., Publishers, New York, 1964. With a foreword by Martin Gardner.

\bibitem[Tro62]{Trotter}
H.~F. Trotter, \emph{Algorithm 115: Perm},
  \href{http://dx.doi.org/10.1145/368637.368660}{Commun. ACM \textbf{5}
  (1962)}, no.~8, 434–435.

\bibitem[Val07]{vallettepartitionposets}
B.~Vallette, \emph{Homology of generalized partition posets},
  \href{http://dx.doi.org/10.1016/j.jpaa.2006.03.012}{J. Pure Appl. Algebra
  \textbf{208} (2007)}, no.~2, 699--725.

\end{thebibliography}

\end{document}